\documentclass[a4paper,11pt,reqno]{amsart}

\usepackage{amsfonts}
\usepackage{amssymb}
\usepackage{amscd}
\usepackage{amsthm}
\usepackage{a4wide}
\usepackage{mathrsfs}
\usepackage[applemac]{inputenc}
\usepackage{amsmath}
\usepackage{amssymb}
\usepackage{amsthm}
\usepackage{textcomp}
\usepackage{graphicx}
\usepackage{enumerate}
\usepackage{mathrsfs}
\usepackage{frcursive}
\usepackage{tikz}
\usepackage[cyr]{aeguill}
\usepackage{xspace}
\usepackage{hyperref}
\usepackage{appendix}

\newtheorem{defn}{Definition}[section]
\newtheorem{lemma}[defn]{Lemma}
\newtheorem{prop}[defn]{Proposition}
\newtheorem{theo}[defn]{Theorem}
\newtheorem{coro}[defn]{Corollary}

\newtheorem{claim}{Claim}
\newtheorem{rk}[defn]{Remark}

\def\Ric{\mathop{\rm Ric}\nolimits}
\def\R{\mathop{\rm R}\nolimits}

\def\Rm{\mathop{\rm Rm}\nolimits}
\def\tr{\mathop{\rm tr}\nolimits}

\def\vol{\mathop{\rm vol}\nolimits}
\def\eucl{\mathop{\rm eucl}\nolimits}

\def\vol{\mathop{\rm Vol}\nolimits}

\def\inj{\mathop{\rm inj}\nolimits}

\def\div{\mathop{\rm div}\nolimits}

\def\A{\mathop{\rm A}\nolimits}

\def\AVR{\mathop{\rm AVR}\nolimits}

\def\cat#1{{\mathfrak{#1}}}

\def\Li{\mathop{\rm \mathscr{L}}\nolimits}

\def\Diff{\mathop{\rm Diff}\nolimits}

\def\Cod{\mathop{\rm Cod}\nolimits}
\def\Ric{\mathop{\rm Ric}\nolimits}

\def\Rm{\mathop{\rm Rm}\nolimits}
\def\tr{\mathop{\rm tr}\nolimits}

\def\vol{\mathop{\rm vol}\nolimits}
\def\eucl{\mathop{\rm eucl}\nolimits}

\def\vol{\mathop{\rm Vol}\nolimits}

\def\inj{\mathop{\rm inj}\nolimits}

\def\div{\mathop{\rm div}\nolimits}

\def\A{\mathop{\rm A}\nolimits}

\def\AVR{\mathop{\rm AVR}\nolimits}

\def\cat#1{{\mathfrak{#1}}}

\def\Li{\mathop{\rm \mathscr{L}}\nolimits}

\def\Diff{\mathop{\rm Diff}\nolimits}

\def\Sym{\mathop{\rm Sym}\nolimits}

\def\supp{\mathop{\rm supp}\nolimits}

\def\Har{\mathop{\rm \mathscr{H}}\nolimits}
\def\crit{\mathop{\rm Crit}\nolimits}
\def\C{\mathop{\rm \mathscr{C}}\nolimits}

\newcommand{\bh}{\bar{h}}

\title{Smoothing out positively curved metric cones by Ricci expanders}
\author[Alix Deruelle]{Alix Deruelle}
\address[Alix Deruelle]{Mathematics Institute, University of Warwick, Gibbet Hill Rd, Coventry, West Midlands CV4 7AL}
\email{A.Deruelle@warwick.ac.uk}

\begin{document}
\begin{abstract}
We investigate the possibility of desingularizing a positively curved metric cone by an expanding gradient Ricci soliton with positive curvature operator. This amounts to study the deformation of such geometric structures. As a consequence, we prove that the moduli space of conical positively curved gradient Ricci expanders is connected.
\end{abstract}
\maketitle

\section{Introduction}
Given a metric cone $C(X)$ over a smooth compact Riemannian manifold $(X,g_X)$ endowed with its Euclidean cone metric $g_{C(X)}:=dr^2+r^2g_X$, we focus on sufficient conditions on the section $(X,g_X)$ ensuring the desingularization of this metric cone by the Ricci flow. More specifically, we look for an expanding gradient Ricci soliton smoothing out $\left(C(X),dr^2+r^2g_X\right)$ instantaneously. Recall that an expanding gradient Ricci soliton is a triplet $(M^n,g,\nabla f)$ where $(M^n,g)$ is a complete Riemannian manifold and $f$ is a smooth function (called the potential function) such that the following static equation holds on $M^n$ :
\begin{eqnarray}\label{eq-egs}
\Ric(g)+\frac{g}{2}=\nabla^{g,2}f.
\end{eqnarray}

In other words, we ask for the Bakry-Émery tensor $\Ric(g)+\nabla^{g,2}(-f)$ to be constantly negative. Formally speaking, one can associate to any expanding gradient Ricci soliton a corresponding Ricci flow $g(\tau):=(1+\tau)\phi_{\tau}^*g$, living on $(-1,+\infty)$ where $(\phi_{\tau})_{\tau}$ is the one parameter family of diffeomorphisms generated by the vector field $-\nabla^g f/(1+\tau)$. Note that these expanding structures are the only candidate among Ricci solitons that could do the job. In some sense, from a Ricci flow perspective, it seems that Ricci expanders should be at least generically the best critical metrics when one is unable to find an Einstein metric, a shrinking or a steady gradient Ricci soliton. The purpose of this article is to confirm these heuristics in some particular setting we describe now. 

In this paper, we will stick to the case where $X$ is diffeomorphic to the standard sphere $\mathbb{S}^{n-1}$ so that there is no topological restriction to do so : the first obvious necessary condition to desingularize such a metric cone is that the section of the cone has to bound but it has nothing to do with the expanding structure. We actually answer affirmatively when the section of the cone is positively curved, but before giving a statement, we need to recall a couple of definitions.

First, let us start with a very general definition of a Ricci flow with rough initial condition inspired by \cite{Sie-Phd} at least.

\begin{defn}
Let $(X,d_X)$ be a metric space with diameter less than $\pi$ and let $\left(C(X),d_{C(X)}\right)$ be the metric cone endowed with the Euclidean metric : $$d_{C(X)}((x,t),(y,s)):=\sqrt{t^2+s^2-2ts\cos(d_X(x,y))}.$$ Then, a smooth family $(M^n,g(t))_{t\in(0,T)}$, with $T\in(0,+\infty]$, of complete Riemannian manifolds is a Ricci flow with initial condition $(C(X),d_{C(X)},o)$ if it satisfies for some point $p\in M^n$ :
\begin{eqnarray*}
\left\{
\begin{array}{rl}
&\partial_tg=-2\Ric(g(t)),\quad\mbox{on $M^n\times (0,T),$}\\
&\\
& \lim_{t\rightarrow 0}(M^n,d_{g(t)},p)=\left(C(X),d_{C(X)},o\right),
\end{array}
\right.
\end{eqnarray*}
where the convergence is in term of the pointed Gromov-Hausdorff topology.
\end{defn}

When $(M^n,g(t))_{t\in(0,+\infty)}$ is an expander, the two notions of rough initial condition and asymptotic cone (in the sense of Gromov) coincide. Indeed, assume that $p$ is in the critical set of $f$, an assumption that will be satisfied in our setting, then 

\begin{eqnarray*}
\left(C(X),d_{C(X)},o\right)&=&\lim_{t\rightarrow 0}(M^n,d_{g(t)},p)=\lim_{t\rightarrow 0}(M^n,t^{1/2}\phi_{t-1}^*d_g,p)\\
&=&\lim_{t\rightarrow 0}(M^n,t^{1/2}d_g,\phi_{t-1}(p))=\lim_{r\rightarrow +\infty}(M^n,r^{-1}d_{g},p).
\end{eqnarray*}
Regarding our main problem, this means that it amounts to study the asymptotics of expanding gradient Ricci solitons that are \textit{asymptotically conical} whose definition taken from \cite{Der-Asy-Com-Egs} is recalled below :
\begin{defn}
An expanding gradient Ricci soliton $(M^n,g,\nabla f)$ is asymptotically conical with asymptotic cone $\left(C(X),dr^2+r^2g_X,r\partial r/2\right)$ if there exists a compact $K\subset M$, a positive radius $R$, a diffeomorphism $\phi:M\setminus K\rightarrow C(X)\setminus B(o,R)$ and a sequence of nonnegative functions $(f_k)_k$ defined on $\mathbb{R}_+$ such that
\begin{eqnarray}
&&\sup_{\partial B(o,r)}\arrowvert\nabla^k\left(\phi_*g-g_{C(X)}\right)\arrowvert_{g_{C(X)}}=\textit{O}(f_{k}(r)),\quad \forall k\in \mathbb{N},\label{class-def-asy-con}\\
&&f(\phi^{-1}(r,x))=\frac{r^2}{4},\quad\forall (r,x)\in C(X)\setminus B(o,R), \label{cond-exp-1}\\
&&f_k(r)=\textit{o}(1),\quad \mbox{as $r\rightarrow+\infty$, $\forall k\geq 0$}.
\end{eqnarray}

\end{defn}

In \cite{Der-Asy-Com-Egs}, we actually proved that the convergence is only polynomial at rate $\tau=2$, i.e. $f_{k}(r)=r^{-2-k}$ for any nonnegative integer $k$ when the cone is not Ricci flat : this rate is sharp. We are now in a position to state the first main theorem of this paper :

\begin{theo}\label{theo-I}
Let $(X,g_X)$ be a smooth simply connected compact Riemmanian manifold such that $\Rm(g_X)\geq 1$. Then there exists a unique expanding gradient Ricci soliton with nonnegative curvature operator asymptotic to $\left(C(X),dr^2+r^2g_X,r\partial_r/2\right)$.
\end{theo} 

The second aim of the paper is to give a classification of asymptotically conical expanding gradient Ricci soliton with positive curvature operator in dimension greater than $2$. Indeed, in dimension $2$, Kotschwarr \cite{Kot-Unp-Not} proved that an expanding gradient Ricci soliton with non negative scalar curvature is rotationally symmetric hence reducing the equation (\ref{eq-egs}) to an O.D.E.. The O.D.E. analysis then furnishes a one parameter family of expanding gradient Ricci soliton $(\Sigma^2,g_c)_{c\in(0,1]}$ asymptotical to $\left(C(\mathbb{S}^1),dr^2+(cr)^2d\theta^2)\right)_{c\in(0,1]}$, $\Sigma^2$ being diffeomorphic to $\mathbb{R}^2$. The analogous higher dimensional examples have been provided by Bryant in unpublished notes [Chap. $1$,\cite{Cho-Lu-Ni-I}] : this is again a one parameter family $(\mathbb{R}^n,g_c)_{c\in(0,1]}$ that is rotationally symmetric, that has non negative curvature operator and that is asymptotic to $\left(C(\mathbb{S}^{n-1}),dr^2+(cr)^2g_{\mathbb{S}^{n-1}}\right)_{c\in(0,1]}$. We are able to show the following :

\begin{theo}\label{theo-II}
Let $(M^n,g,\nabla f)$, $n\geq 3$, be an expanding gradient Ricci soliton with positive curvature that is asymptotically conical. Then there exists a one parameter family of asymptotically conical expanding gradient Ricci solitons with positive curvature connecting $(M^n,g,\nabla f)$ to a Bryant soliton $(\mathbb{R}^n,g_c,\nabla f_c)$ with some $c$ depending on the volume of the section of the asymptotic cone of $(M^n,g)$.
\end{theo}
In other words, theorem \ref{theo-II} shows that the moduli space of positively curved expanding gradient Ricci solitons smoothly coming out of a metric cone is connected.

To tackle this problem, i.e. theorems \ref{theo-I} and \ref{theo-II}, we use a continuity method. Roughly speaking, the main idea comes from the following observation : given a simply connected Riemannian manifold $(X,g_X)$ such that $\Rm(g_X)\geq 1$, one can start the (normalized) Ricci flow and end up with a one parameter family of metrics $(g(s))_{s\in[0,+\infty]}$ on $X$ of constant volume connecting $(X,g_X)$ to $(X,c^2g_{\mathbb{S}^{n-1}})$ where $c^{n-1}=\vol(X,g_X)/\vol(\mathbb{S}^{n-1},g_{\mathbb{S}^{n-1}})$ thanks to the fundamental work of Böhm and Wilking \cite{Boh-Wil}. Therefore, the initial metric cone $(C(X),dr^2+r^2g_X)$ is connected to the cone $(C(\mathbb{S}^{n-1}),dr^2+(cr)^2g_{\mathbb{S}^{n-1}})$ which is smoothed out by the corresponding Bryant soliton ! As usual, one splits the proof in two parts : the openness and the closedness of the set of such solutions. The compactness of positively curved expanding gradient Ricci solitons has been proved in \cite{Der-Asy-Com-Egs}. Indeed, we proved the following theorem : 

\begin{theo}\label{Compactness-II}[\cite{Der-Asy-Com-Egs}]
The class
\begin{eqnarray*}
\cat{M}^{\vol}_{Exp}(n,(\Lambda_k)_{k\geq 0},V_0)&:=&\{\mbox{$(M^n,g,\nabla f,p)$ normalized expanding gradient Ricci soliton. s.t.} 
\\&&\Rm(g)\geq 0\quad;\quad\crit(f)=\{p\}\quad;\\
&&\quad\AVR(g)\geq V_0\quad ;\quad \A_g^k(\Rm(g))\leq\Lambda_k,\quad\forall k\geq 0 \}
\end{eqnarray*}
is compact in the pointed $C^{\infty}$ topology. \\

Moreover, let a sequence $(M_i,g_i,\nabla f_i,p_i)_i$ be in $\cat{M}^{\vol}_{Exp}(n,(\Lambda_k)_{k\geq 0},V_0)$. Then there exists a subsequence converging in the pointed $C^{\infty}$ topology to an expanding gradient Ricci soliton $(M_{\infty},g_{\infty},\nabla f_{\infty},p_{\infty})$ in $\cat{M}^{\vol}_{Exp}(n,(\Lambda_k)_{k\geq 0},V_0)$ whose asymptotic cone\\ $(C(X_{\infty}),g_{C(X_{\infty})},o_{\infty})$ is the limit in the Gromov-Hausdorff topology of the sequence of the asymptotic cones $(C(X_i),g_{C(X_i)},o_i)_i$ with $C^{\infty}$ convergence outside the apex. 
\end{theo}

We explain how theorem \ref{Compactness-II} solves the closedness of the set of solutions as in theorems \ref{theo-I} and \ref{theo-II}. As the operator is supposed to be nonnegative, the potential function is a proper strictly convex function (proposition \ref{pot-fct-est}), in particular, its critical set is reduced to a point and the underlying manifold is diffeomorphic to $\mathbb{R}^n$. Therefore, the assumption $\crit(f)=\{p\}$ only suggests that we decide to mark this Riemannian manifold by this particular point.

We recall that the invariants $ (\A_g^k(\Rm(g)))_{k\geq 0}$ are defined by $$ \A_g^k(\Rm(g)):=\limsup_{r_p\rightarrow+\infty}r_p^{2+k}\arrowvert\nabla^k\Rm(g)\arrowvert.$$ By the Gauss equations applied to the section of the asymptotic cone of a Ricci expander, bounding these invariants uniformly amounts to bound  the covariant derivatives of the curvature of the link uniformly which is actually the case in our setting described above. Moreover, the asymptotic volume ratio $\AVR(g):=\lim_{r\rightarrow +\infty}\vol B(p,r)/r^n$, for some point $p\in M^n$  is proportional to the volume of the link that is constant here, hence bounded from below uniformly.

Therefore, it remains to provide a proof of the openness of such set of solutions. Before going deeper into the details of the proof, we mention works on analogous subjects that share similarities with this approach. As a kind of Dirichlet problem at infinity, we benefitted from the works on conformally compact Einstein manifolds by Anderson \cite{And-Pre-Con-Inf}, Biquard \cite{Biq-Met-Asy-Ein}, Graham-Lee \cite{Gra-Lee-Ein-Con-Inf}, Lee \cite{Lee-Hyp} to mention a few. More closely, this project on Ricci expanders has been motivated by the thesis of Siepmann \cite{Sie-Phd} on deformations of Ricci expanders in a Kähler setting : we will mention all along the paper the technical difficulties we had to face when one drops the Kähler assumption. Finally,  the main source of inspiration comes from the work of Schulze and Simon \cite{Sch-Sim} : they actually show that any asymptotic cone of a generic Riemannian manifold with positive curvature operator and positive asymptotic volume ratio can be smoothed out by a positively curved expanding gradient Ricci soliton where the convergence is only given in term of the pointed Gromov-Hausdorff topology. Therefore, our results make precise the picture in the case the convergence to the asymptotic cone is smooth. More precisely, recall that the asymptotic cone of an $n$-dimensional Riemannian manifold with nonnegative curvature operator (nonnegative sectional curvature suffices) and positive asymptotic volume ratio is a metric cone over a compact Alexandrov space homeomorphic to $\mathbb{S}^{n-1}$ of curvature not less than $1$ : see \cite{Sch-Sim} and the references therein. Therefore, theorem \ref{theo-I} provides a complete understanding of which metric cone can appear as the asymptotic cone of a Ricci expander in a smooth setting. \\

We now give both the structure of the paper and the main steps of the proof :\\
\begin{enumerate}
\item Fix once and for all an expanding gradient Ricci soliton $(M^n,g_0,\nabla^{g_0}f_0)$ with positive curvature asymptotic to $(C(\mathbb{S}^{n-1}),dr^2+r^2\gamma_0,r\partial_r/2)$. Given a perturbed metric $\gamma_1$ of $\gamma_0$ on the section of the cone, we want to solve the following problem :
\begin{eqnarray}\label{E-system}
\left\{
\begin{array}{rl}
&-2\Ric(g_1)-g_1+\Li_{\nabla^{g_1}f_0}(g_1)=0\\
&\\
& \lim_{r\rightarrow +\infty}(M^n,r^{-2}g,p)=(C(\mathbb{S}^{n-1}),dr^2+r^2\gamma_1,o).
\end{array}
\right.
\end{eqnarray}
Note that we are fixing the potential function : it turns out that it simplifies the analysis, the main reason being that on any asymptotically conical expanding gradient Ricci soliton the potential functions are all the same outside a compact set when pulled back on their respective asymptotic cone. This choice is also adopted in \cite{Sie-Phd}.

The first issue one usually faces when trying to prescribe some Ricci curvature is the lack of strict ellipticity of the system (\ref{E-system}). The is due to the invariance under the whole group of diffeomorphisms of $M^n$. To circumvent this problem, one uses the so called DeTurck's trick, that amounts to fix a gauge once and for all, by adding the right Lie derivative so that the linearization becomes strictly elliptic. Indeed, the linearization of (\ref{E-system}) at $(M^n,g_0)$ is :
\begin{eqnarray*}
D_{g_0}\left(-2\Ric(g_1)-g_1+\Li_{\nabla^{g_1}f_0}(g_1)\right)(h)&=&\Delta_{g_0,f_0}h+2\Rm(g_0)\ast h\\
&&-\Li_{\div_{g_0,f_0}h-\nabla^{g_0}\tr_{g_0}h/2}(g_0),
\end{eqnarray*}
where $h$ is a symmetric $2$-tensor, where $\Delta_{g_0,f_0}h:=\Delta_{g_0}h+\nabla^{g_0}_{\nabla^{g_0}f_0}h$ is the weighted laplacian and $\div_{g_0,f_0}h:=\div_{g_0}h+h(\nabla^{g_0}f_0)$ is the weighted divergence and finally $(\Rm(g_0)\ast h)_{ij}:=\Rm(g_0)_{iklj}h_{kl}$. The operator $\Delta_{g_0,f_0}$ is also known as the Witten laplacian when considered on differential forms and belongs to the class of Fokker-Planck operators.  \\

\item
\begin{itemize}
\item We are then reduced to study the Fredholm properties of the linearized operator $L$ also called the weighted Lichnerowicz operator defined by :
$$L:=\Delta_{g_0,f_0}+2\Rm(g_0)\ast .$$
As we said previously, the convergence to the asymptotic cone of $(M^n,g_0,\nabla^{g_0}f_0)$ is only polynomial in the distance to a fixed point. This prevents us from using weighted Sobolev spaces modeled on $L^2(e^{f_0}d\mu(g))$ on which $L$ or $\Delta_{g_0,f_0}$ is symmetric. Hence the inevitable use of adequate Hölder spaces : see section \ref{Sec-Fct-Spa}.

Now, essentially because $L$ has unbounded coefficients, Da Prato and Lunardi \cite{DaP-Lun-Orn-Uhl} showed that such operator is neither strongly continuous nor analytic in the space of bounded continuous functions on $(\mathbb{R}^n,\eucl)$. Nonetheless, adapting the proof of Lunardi in the Euclidean case \cite{Lun-Sch-Est} to a Riemannian setting, we are able to make sense of the semi-group and the resolvent associated to this weighted laplacian.  The main reason why it works comes from the basic observation :
\begin{eqnarray}\label{jus-con-res}
f_0^{\alpha}\circ \Delta_{g_0,f_0}\circ f_0^{-\alpha}\simeq \Delta_{g_0,f_0}-\alpha,
\end{eqnarray}
up to compact perturbations, for any real number $\alpha$. As the potential function is quadratic in the distance, computation (\ref{jus-con-res}) shows that this weighted laplacian preserves tensors with some given polynomial decay at infinity and when $\alpha$ is positive, one can make sense of the resolvent. In section \ref{Sec-Sol-C^0}, we state the main results that prove the existence of solutions in the previously mentioned Hölder spaces with low regularity : theorems \ref{Main-theo-weig-sch-est}, \ref{first-iso-sch}. The proof of theorem \ref{Main-theo-weig-sch-est} is postponed to section \ref{sub-sec-proof-the-Lun} where we recall several facts about interpolation theory in a Riemannian setting. We then state and prove in section \ref{Sec-Sol-C^k} theorem \ref{iso-weighted-lap-II} that is the analogue of theorem  \ref{first-iso-sch} with higher regularity.\\

\item As $L$ is a compact perturbation of the weighted laplacian, it remains to show that it is injective so that $L$ is an injective Fredholm operator of index $0$, i.e. an isomorphism. For that purpose, we use for the first time the assumption on the sign of the curvature. More precisely, the main tool is the maximum principle for symmetric $2$-tensors due to Hamilton that has been used several times in the setting of Ricci solitons : \cite{Bre-Rot-3d}, \cite{Cho-Exp-Asy-Con}, \cite{Der-Asy-Com-Egs} to mention a few : this corresponds to corollaries \ref{sec-iso-sch}, \ref{sec-iso-sch-II}.\\

\end{itemize}
\item It suffices now to apply the implicit function theorem for Banach spaces to get an expanding deformation of $(M^n,g_0,\nabla^{g_0}f_0)$ : this is the purpose of section \ref{sec-exp-str-I}. If $\gamma_1$ is regular enough but not smooth, this implicit expanding soliton is not smooth at infinity even if it is (locally) analytic by the results of Bando \cite{Ban-Ana}. This actually provides, up to the knowledge of the author, the first example of expanding Ricci solitons that are not smoothly asymptotic to their asymptotic cone : see theorem \ref{Loc-Dif-Ban}.

Now, what if $\gamma_1$ is $C^{\infty}$ ? We actually have to use the Nash-Moser theorem to circumvent this inevitable loss of derivatives. We follow the presentation of the Nash-Moser theorem due to Hamilton \cite{Ham-Nas-Mos} dealing with the so called tame category for Fréchet spaces : see section \ref{Sec-Fct-Spa-Fre} for the definitions of this category. It turns out that our previous choice of function spaces is not in the tame category mainly because bounded continuous tensors do not extend in an obvious way at infinity (or on the section of the asymptotic cone). This is why we study the realization of the weighted laplacian  in the space of continuous tensors converging to $0$ at infinity : see section \ref{Sec-Fct-Spa-Fre} for the definitions of the relevant Fréchet function spaces and section \ref{Sec-Sol-C^k_0} where we state and prove the main theorems \ref{iso-weighted-lap-I-bis} and \ref{iso-weighted-lap-II-bis} analogous to theorems \ref{iso-weighted-lap-II} and  \ref{first-iso-sch}.
Finally, it is well-known that one needs to invert the linearized operator on a whole neighborhood of $(M^n,g_0)$ (not just at this point) : this is the content of theorem \ref{Lin-Iso-Fre} in section \ref{Sec-Tam-Lin-Ope}. This will lead us directly to the existence of Fréchet deformations with the help of the Nash-Moser theorem : see theorem \ref{Loc-Dif-Fre} of section \ref{Sec-Def-Exp-Fre}.
 \\

\item So far, what we managed to implicitly produce is an asymptotically conical expanding Ricci soliton that is not GRADIENT a priori, that is a triplet $(M^n,g_1,V_1)$ satisfying the static equation $$2\Ric(g_1)+g_1=\Li_{V_1}(g_1).$$
To compensate for this issue, we need an ad-hoc no breather theorem in the spirit of Perelman : theorem \ref{trivial-exp-sol-gra}. In our setting, we can actually guarantee that $(M^n,g_1)$ has positive curvature operator, something that is not straightforward since the radial curvatures of a metric cone always vanish, i.e. it is not an obvious open condition. It turns out that the decay of the radial curvatures of an expanding gradient Ricci soliton can be controlled from below : \cite{Der-Asy-Com-Egs}. These curvatures are essentially given by the divergence of the curvature operator. Therefore, if the metric $g_1$ is sufficiently regular at infinity ($C^3$ at least), the curvature operator of $g_1$ is positive. 

Section \ref{cao-ham-sec} starts with some general considerations that motivate the introduction of some sophisticated objects such as the Hamilton matrix Harnack quadratic and some adequate entropy : proposition \ref{Mot-Ent-Cao-Ham}.
Then, the rest of this section is devoted to give two different proofs of theorem \ref{trivial-exp-sol-gra}.

On one hand, we are able to adapt the arguments of Hamilton \cite{Ham-Ete-Ric} and Ni \cite{Ni-Mon-Kah-Ric} in section \ref{Sec-Fir-Pro-No-Bre} to get rid of non gradient expanding Ricci soliton. This proof is based on the maximum principle and fully uses the Hamilton matrix Harnack estimate.

On the other hand, section \ref{Sec-Sec-Pro-No-Bre} gives an alternative proof based on some entropy introduced by Cao and Hamilton \cite{Cao-Ham-Har} for compact Riemannian manifolds with nonnegative curvature operator. Thanks to the work of Zhang \cite{Zha-Log}, we are able to justify the monotonicity of this entropy along a Ricci flow  with nonnegative curvature operator and controlled entropy at infinity. The proof uses the trace Harnack inequality due to Hamilton. In some sense, only the assumptions on the geometry at infinity differ. We end this section by giving a proof of theorem \ref{theo-I}.

\end{enumerate}

\textbf{Acknowledgements.}
 I would like to thank Felix Schulze and Peter Topping for all these fruitful conversations we had on Ricci expanders.\\

The author is supported by the EPSRC on a Programme Grant entitled ‘Singularities of Geometric Partial Differential Equations’ (reference number EP/K00865X/1).

\section{Banach deformations}\label{Sec-Ban-Def-Egs}
\subsection{Functions spaces}\label{Sec-Fct-Spa}
In the spirit of \cite{Sie-Phd}, we define the following function (weighted) Hölder spaces for a complete Riemannian manifold $(M,g)$. Let $E$ be a tensor bundle over $M$, i.e. $(\otimes^rT^*M)\otimes(\otimes^sTM)$, or a subbundle such as the exterior bundles $\Lambda^rT^*M$ or the symmetric bundles $S^rT^*M$ : we will mainly consider the cases where $E$ is either the trivial line bundle over $M$ or the bundle of symmetric $2$-tensors $S^2T^*M$. 
Let $\theta\in(0,1)$ and let $k$ be a nonnegative integer.

We start with a couple of definitions.\\

\begin{itemize}

\item Let $(M,g)$ be a Riemannian manifold and let $V$ be a smooth vector field on $M$. The \textit{weighted laplacian} (with respect to $V$) is 
$\Delta_{V}T:=\Delta T+\nabla_{V}T$,
where $T$ is a tensor on $M$.

\item Let $(M,g)$ be a Riemannian manifold and let $w:M\rightarrow\mathbb{R}$ be a smooth function on $M$. The \textit{weighted laplacian} (with respect to $w$) is 
$\Delta_{w}T:=\Delta_{\nabla w} T$,
where $T$ is a tensor on $M$.

\end{itemize}

\begin{itemize}
\item $C^{k,\theta}(M,E):=\{h\in C^{k,\theta}_{loc}(M,E)\quad|\quad\|h\|_{C^{k,\theta}(M,E)}<+\infty\}$, where 
$$\|h\|_{C^{k,\theta}(M,E)}:=\sum_{i=0}^k\sup_M\arrowvert\nabla^ih\arrowvert+\left[\nabla^kh\right]_{\theta},$$ where 
$$\left[\nabla^kh\right]_{\theta}:=\sup_{x\in M}\sup_{y\in B(x,\delta)\setminus\{x\}}\frac{\arrowvert\nabla^kh(x)-P_{x,y}^*\nabla^kh(y)\arrowvert}{d(x,y)^{\theta}},$$ where $\delta$ is a fixed positive constant depending on the injectivity radius of $(M,g)$ and $P_{x,y}$ denotes the parallel transport along the unique minimizing geodesic from $x$ to $y$.\\

\item $C_{con}^{k,\theta}(M,E):=\{h\in C^{k,\theta}_{loc}(M,E)\quad|\quad\|h\|_{C_{con}^{k,\theta}(M,E)}<+\infty\}$ where
$$\|h\|_{C_{con}^{k,\theta}(M,E)}:=\sum_{i=0}^k\|(r_p^2+1)^{i/2}\nabla^ih\|_{C^{0,\theta}(M,E)},$$
where $r_p$ denotes the distance function to a fixed point $p\in M$.

\item Let $V$ be a smooth vector field on $M$.  
$$D^{k+2}_{V}(M,E):=\{h\in \cap_{p\geq 1}W^{k+2,p}_{loc}(M,E)\quad|\quad h\in C_{con}^{k}(M,E)\quad|\quad \Delta_{V} h\in C_{con}^{k}(M,E)\},$$ equipped with the norm 

$$\|h\|_{D^{k+2}_{V}(M,E)}:=\|h\|_{C_{con}^{k}(M,E)}+\|\Delta_{V}h\|_{C^{k}_{con}(M,E)}.$$

$$D^{k+2,\theta}_{V}(M,E):=\{h\in C^{k+2,\theta}_{loc}(M,E)\quad|\quad h\in C_{con}^{k,\theta}(M,E)\quad|\quad \Delta_{V} h\in C_{con}^{k,\theta}(M,E)\},$$ equipped with the norm 

$$\|h\|_{D^{k+2,\theta}_{V}(M,E)}:=\|h\|_{C_{con}^{k,\theta}(M,E)}+\|\Delta_{V} h\|_{C^{k,\theta}_{con}(M,E)}.$$\\

\item Let $w:M\rightarrow\mathbb{R}_+$ be a smooth function on $M$. 

\begin{eqnarray*}
C_{con,w}^{k,\theta}(M,E):=w^{-1}C_{con}^{k,\theta}(M,E),\\
\end{eqnarray*}
 endowed with the norm $\|h\|_{C_{con,w}^{k,\theta}(M,E)}:=\|wh\|_{C_{con}^{k,\theta}(M,E)}.$\\
Similarly, we define 
\begin{eqnarray*}
D^{k+2,\theta}_{w,V}(M,E):=w^{-1}\cdot D^{k+2,\theta}_{V}(M,E),\\
\end{eqnarray*}
 endowed with the norm $\|h\|_{D^{k+2,\theta}_{w,V}(M,E)}:=\|wh\|_{D^{k+2,\theta}_{V}(M,E)}.$\\

\end{itemize}

\begin{rk} \label{rk-non-int-spa}
\begin{itemize}
\item It is not obvious that the spaces $D_{w,X}^{k+2,\theta}(M,E)$ are Banach spaces, a priori.

\item
We point out that the spaces $C_{con}^{k,\theta}(M,E)$ are not equal to the interpolation spaces $(C^k_{con}(M,E),C^{k+1}_{con}(M,E))_{\theta,\infty}$. Indeed, one can identify these interpolation spaces as follows : 
\begin{eqnarray*}
&&(C^k_{con}(M,E),C^{k+1}_{con}(M,E))_{\theta,\infty}=\left\{h\in C^k_{con}(M,E) \quad|\quad \left[(r_p^2+1)^{k/2}\nabla^kh\right]_{con,\theta}<+\infty\right\},\\
&&\left[H\right]_{con,\theta}:=\sup_{x\in M}\sup_{y\in B(x,\delta r_p(x))\setminus\{x\}}\min\left\{r_p(x)^{\theta},r_p(y)^{\theta}\right\}\frac{\arrowvert H(x)-P_{x,y}^*H(y)\arrowvert}{d(x,y)^{\theta}},
\end{eqnarray*}
where $H$ is a tensor on $M$ and $\delta$ is a fixed positive constant depending on a lower bound of $\inf_{x\in M}\inj(x,g)/r_p(x)$, where $\inj(x,g)$ denotes the injectivity radius at $x$ in $M$ for the metric $g$.
\end{itemize}

\end{rk}
The weights $w$ we consider are polynomial instead of being exponential in the distance as in \cite{Sie-Phd}. The main reason is that the convergence to the asymptotic cone is polynomial in the non negatively curved case whereas the convergence is exponential as soon as the asymptotic cone is Ricci flat : \cite{Sie-Phd} and \cite{Der-Asy-Com-Egs}.

If $(M^n,g,\nabla f)$ is a non trivial expanding gradient Ricci soliton, we define $$v:=f+\mu(g)+n/2.$$ By proposition \ref{pot-fct-est}, $v$ is a positive eigenfunction for the weighted laplacian $\Delta_f$ associated to the eigenvalue $1$. Moreover, if $(M^n,g,\nabla f)$ is asymptotically conical, then $f$, hence $v$ is equivalent to $r_p^2/4$. From now on, $v$ will play the role of the distance function $r_p$, for some $p\in M$.

We study here the Fredholm properties of the weighted Lichnerowicz operator 
$$L_0:=\Delta_f+2\Rm(g)\ast,$$
 in weighted spaces $D^{k+2,\theta}_{f^{\alpha},\nabla f}(M,E):= D^{k+2,\theta}_{v^{\alpha},\nabla f}(M,E)$ for some positive $\alpha$ and some $\theta\in(0,1)$. For this purpose, it is convenient to consider the weighted Lichnerowicz operator defined by 
\begin{eqnarray*}
&&L_{\alpha}h:=\left(\Delta_{v-\gamma\ln v}-\alpha+2\Rm(g)\ast+\alpha(\alpha+1)\arrowvert\nabla\ln v\arrowvert^2\right)h,\quad h\in S^2T^*M.
\end{eqnarray*}
The reason for introducing such operator is given by the following formal relation :
\begin{eqnarray*}
L_{\alpha}h=v^{\alpha}L_0(v^{-\alpha}h),\quad h\in S^2T^*M.
\end{eqnarray*}
Therefore, for $h\in D_{f^{\alpha},\nabla f}^{k+2,\theta}(M,S^2T^*M)$, $L_0h=L_0h_0\in C_{con,f^{\alpha}}^{k,\theta}(M,S^2T^*M)$ if and only if $L_{\alpha}h_{\alpha}\in C^{k,\theta}_{con}(M,S^2T^*M)$ where $h_{\alpha}:=v^{\alpha}h$.

\subsection{Solutions in $D^{2,\theta}_{\nabla f}(M,E)$}\label{Sec-Sol-C^0}
The main aim of this section is to establish the existence of $C^{2,\theta}$ solutions for families of weighted operators. We start with the following theorem that is an adaptation of \cite{Lun-Sch-Est} in a Riemannian setting.
\begin{theo}(Lunardi)\label{Main-theo-weig-sch-est}
Let $(M^n,g)$ be a complete Riemannian manifold with positive injectivity radius. Let $V$ be a smooth vector field on $M$. Assume that, for some large integer $k\geq 1$, there exists a constant $K(k)$ such that
\begin{eqnarray}
\|\Rm(g)\ast V\|_{C^k(M,E)}+\|\Rm(g)\|_{C^k(M,E)}+\|\nabla V\|_{C^{k-1}(M,E)}\leq K(k),
\end{eqnarray}
where $\Rm(g)\ast V:=\Rm(g)(V,\cdot,\cdot,\cdot)$. Assume there exists a smooth function $\phi:M\rightarrow \mathbb{R}$ such that 
\begin{eqnarray*}
\lim_{+\infty}\phi=+\infty,\quad\sup_M\arrowvert\Delta_g\phi\arrowvert+\arrowvert\nabla_{V}\phi\arrowvert<+\infty.
\end{eqnarray*}
 Then, 
 \begin{itemize}
 \item For any $H\in C^{0}(M,E)$ and $\lambda>0$, there exists a unique tensor $h\in D_{V}^2(M,E)$ and a positive constant $C$ such that
 \begin{eqnarray*}
\Delta_{V}h-\lambda h=H,\quad \|h\|_{D^2_{V}(M,E)}\leq C\|H\|_{C^{0}(M,E)},
\end{eqnarray*}
Moreover, $D_{V}^2(M,E)$ is continuously embedded in $C^{1,\theta}(M,E)$ for any $\theta\in(0,1)$.\\
 \item For any $H\in C^{0,\theta}(M,E)$, with $\theta\in(0,1)$, and $\lambda>0$, there exists a unique tensor $h\in C^{2,\theta}(M,E)$ and a constant $C$ independent of $H$ such that 
\begin{eqnarray*}
\Delta_{V}h-\lambda h=H,\quad \|h\|_{C^{2,\theta}(M,E)}\leq C\|H\|_{C^{0,\theta}(M,E)}.
\end{eqnarray*}
\end{itemize}

\end{theo}

The proof is almost verbatim the proof of Lunardi \cite{Lun-Sch-Est}. We only give the main steps in section \ref{sub-sec-proof-the-Lun}.

The main application of theorem \ref{Main-theo-weig-sch-est} concerns expanding gradient Ricci solitons $(M^n,g,\nabla f)$ where $V$ is cooked up with the help of the potential function $f$. 

\begin{theo}\label{first-iso-sch}
Let $(M^n,g,\nabla f)$ be an expanding gradient Ricci soliton with bounded curvature and positive injectivity radius such that $f$ is proper. Let $\alpha >0$.
\begin{itemize}
 \item There exists a positive constant $C$ such that, for any $H\in C^{0}(M,E)$, there exists a unique tensor $h\in D_{\nabla f}^2(M,E)$ satisfying
 \begin{eqnarray*}
(\Delta_{v-2\alpha \ln v}-\alpha)h=H,\quad \|h\|_{D^2_{\nabla f}(M,E)}\leq C\|H\|_{C^{0}(M,E)}.
\end{eqnarray*}
Moreover, $D_{\nabla f}^2(M,E)$ embeds continuously in $C^{\theta}(M,E):=C^{\lfloor\theta\rfloor,\theta-\lfloor\theta\rfloor}(M,E)$, for any $\theta\in(0,2)$, i.e. there exists a positive constant $C$ such that for any $h\in D_{\nabla f}^2(M,E)$,
\begin{eqnarray*}
\|h\|_{C^{\theta}(M,E)}\leq C\|h\|_{D_{\nabla f}^2(M,E)}^{1-\theta/2}\|h\|_{C^0(M,E)}^{\theta/2}.
\end{eqnarray*}

 \item There exists a positive constant $C$ such that, for any $H\in C^{0,\theta}(M,E)$, with $\theta\in(0,1)$, there exists a unique tensor $h\in C^{2,\theta}(M,E)$ satisfying
 \begin{eqnarray*}
(\Delta_{v-2\alpha \ln v}-\alpha)h=H,\quad \|h\|_{C^{2,\theta}(M,E)}\leq C\|H\|_{C^{0,\theta}(M,E)}.
\end{eqnarray*}
\end{itemize}

Moreover, the operator $\Delta_f: D_{f^{\alpha},\nabla f}^{2,\theta}(M,E)\rightarrow C_{con,f^{\alpha}}^{0,\theta}(M,E)$ is an isomorphism of Banach spaces.

\end{theo}

\begin{proof}
It suffices to verify the assumptions of theorem \ref{Main-theo-weig-sch-est} for $V:=\nabla v-2\alpha\nabla\ln v$.

 First, by Shi's estimates \cite{Shi-Def}, the covariant derivatives of the curvature operator are bounded on $M^n$. 
 
 Secondly, by lemma \ref{id-EGS}, $\Rm(g)(V,\cdot,\cdot,\cdot)$ is bounded since the first covariant derivative of the curvature is bounded and one verifies easily that the covariant derivatives of $V$ are bounded on $M$.
 
 Finally, define $\phi:=\ln v$, by lemma \ref{id-EGS} together with proposition \ref{pot-fct-est},
 \begin{eqnarray*}
&&\sup_M\arrowvert\Delta_g\phi\arrowvert=\sup_M\left\arrowvert\frac{\Delta_gv}{v}-|\nabla\ln v|^2\right\arrowvert<+\infty\\
&&\sup_M\arrowvert\Delta_{V}\phi\arrowvert=\sup_M\left\arrowvert1-(2\alpha+1)|\nabla\ln v|^2\right\arrowvert<+\infty.
\end{eqnarray*}

Now, $\Delta_f$ is injective on $D_{f^{\alpha},\nabla f}^{2}(M,E)$ and  
\begin{eqnarray*}
\Delta_fh&=&(v^{\alpha}\Delta_fv^{-\alpha})h_{\alpha}\\
&=&\left(\Delta_{v-2\alpha\ln v}-\alpha+\alpha(\alpha+1)\arrowvert\nabla\ln v\arrowvert^2\right)h_{\alpha}\\
&=:&\left(\Delta_{v-2\alpha\ln v}-\alpha\right)h_{\alpha}+ K_{\alpha}h_{\alpha},
\end{eqnarray*}
where $h_{\alpha}:=v^{\alpha}h$ and $K_{\alpha}h_{\alpha}:=\alpha(\alpha+1)\arrowvert\nabla\ln v\arrowvert^2h_{\alpha}$.

We claim that $K_{\alpha}:D^{2,\theta}_{\nabla f}(M,E)\rightarrow C_{con}^{0,\theta}(M,E)$ is a compact operator for any $\theta\in(0,1)$. Indeed, let $(h_n)_n$ be a bounded sequence in $D^{2,\theta}_{\nabla f}(M,E)$. Then, since $D^{2,\theta}_{\nabla f}(M,E)$ continuously embeds in $C^{2,\theta}(M,E)$, there exists a subsequence still denoted by $(h_n)_n$ converging to $h\in D^{2,\theta}_{\nabla f}(M,E)$ in the $C_{loc}^{2}(M,E)$-topology by Ascoli's theorem. Now, if $\epsilon>0$ is given, as $f$ is proper, there exists a compact $K_{\epsilon}\subset M$ such that, for $k=0,1$,
\begin{eqnarray*}
|\nabla^{k+1}\ln v|^2(x)\leq \epsilon,\quad \forall x\in M\setminus K_{\epsilon}.
\end{eqnarray*}

Therefore,
\begin{eqnarray*}
\|K_{\alpha}(h_n-h)\|_{C_{con}^{0,\theta}(M,E)}&\leq& \|h_n-h\|_{C^1(M,E)} \|K_{\alpha}\|_{C^1(M\setminus K_{\epsilon},E|_{M\setminus K_{\epsilon}})}\\
&&+\|h_n-h\|_{C^1(K_{\epsilon},E|_{K_{\epsilon}})} \|K_{\alpha}\|_{C^1(M,E)}\\
&\leq& C\epsilon,
\end{eqnarray*}
for $n$ large enough, i.e. $(h_n)_n$ converges to $h$ in the $C_{con}^{0,\theta}(M,E)$ topology.

 Hence, one deduces that $\Delta_f$ is an injective Fredholm operator of index $0$, i.e. an isomorphism.
 
\end{proof}

Ultimately, we want to prove that for some positive $\alpha$, $L_{\alpha}$ is an isomorphism in case $(M^n,g,\nabla f)$ is an expanding gradient Ricci soliton with positive curvature. Now, 
\begin{eqnarray*}
L_{\alpha}&=&v^{\alpha}\circ\Delta_f\circ v^{-\alpha}+2\Rm(g)\ast\\
&=&\widetilde{L_{\alpha}}+ \widetilde{ K_{\alpha}},
\end{eqnarray*}
where $\widetilde{L_{\alpha}}$ is an isomorphism by corollary \ref{first-iso-sch}. We claim that $\widetilde{ K_{\alpha}}:D^{2,\theta}_{\nabla f}(M,E)\rightarrow C_{con}^{0,\theta}(M,E)$ is a compact operator if the curvature goes to $0$ at infinity, for any $\theta\in(0,1)$. Indeed, let $(h_n)_n$ be a bounded sequence in $D_{\nabla f}^{2,\theta}(M,E)$. Then, there exists a subsequence still denoted by $(h_n)_n$ converging to $h\in D_{\nabla f}^{2,\theta}(M,E)$ in the $C_{loc}^{2}(M,E)$ topology. Now, if $\epsilon>0$ is given, there exists a compact $K_{\epsilon}\subset M$ such that, for $k=0,1$,
\begin{eqnarray*}
|\nabla^k\Rm(g)(x)|\leq \epsilon,\quad \forall x\in M\setminus K_{\epsilon},
\end{eqnarray*}
since by local Shi's estimates (see \cite{Der-Asy-Com-Egs} for a proof on expanders), $\lim_{+\infty}\nabla^k\Rm(g)=0$ if $\lim_{+\infty}\Rm(g)=0$.
Therefore,
\begin{eqnarray*}
\|\widetilde{K_{\alpha}}(h_n-h)\|_{C_{con}^{0,\theta}(M,E)}&\leq& \|h_n-h\|_{C^1(M,E)} \| \widetilde{ K_{\alpha}}\|_{C^1(M\setminus K_{\epsilon},E|_{M\setminus K_{\epsilon}})}\\
&&+\|h_n-h\|_{C^1(K_{\epsilon},E|_{K_{\epsilon}})} \| \widetilde{ K_{\alpha}}\|_{C^1(M,E)}\\
&\leq& C\epsilon,
\end{eqnarray*}
for $n$ large enough, i.e. $(h_n)_n$ converges to $h$ in the $C_{con}^{0,\theta}(M,E)$ topology.

In particular, $L_{\alpha}$ is an isomorphism if and only if it is injective since it is a compact perturbation of an isomorphism (i.e. an injective  Fredholm operator of index $0$). Moreover, by construction, $L_{\alpha}$ is injective if and only if $L_0$ is. We are able to prove the following in the particular case $E=S^2T^*M$ :
\begin{coro}\label{sec-iso-sch}
Let $(M^n,g,\nabla f)$ be an expanding gradient Ricci soliton with positive curvature operator and such that the curvature goes to zero at infinity. 

Then, for any $\alpha>0$, $L_{\alpha}$ is an isomorphism or equally,
\begin{eqnarray*}
L_0:D^{2,\theta}_{f^{\alpha},\nabla f}(M,S^2T^*M)\rightarrow C^{0,\theta}_{con,f^{\alpha}}(M,S^2T^*M)
\end{eqnarray*}
is an isomorphism.
\end{coro}

\begin{proof}
As we noticed before, it suffices to prove that $L_0$ is injective on $D^{2,\theta}_{f^{\alpha},\nabla f}(M,S^2T^*M)$. This is achieved with the help of the maximum principle for symmetric $2$-tensors due to Hamilton \cite{Ham-Fou} adapted by Brendle in \cite{Bre-Rot-3d} in the case of Ricci solitons. This has been already used in that setting in \cite{Der-Sta-Sge}. We sketch the argument here for the convenience of the reader.
Let $h$ be a symmetric $2$-tensor such that $\Delta_fh+2\Rm(g)\ast h=0$. The crucial observation is that $L_{\nabla f}g$ satisfies also 
 $\Delta_f\Li_{\nabla f}(g)+2\Rm(g)\ast \Li_{\nabla f}g=0$ and that it is positive definite in case of nonnegative Ricci curvature. Then, one uses $L_{\nabla f}g$ as a barrier tensor to prove that for any height $t$ large enough,
 \begin{eqnarray*}
\sup_{f\leq t}\arrowvert h\arrowvert\leq C\sup_{f=t}\arrowvert h\arrowvert,
\end{eqnarray*}
for a positive constant $C$ independent of $t$. Hence the result since $h$ tends to zero at infinity.

\end{proof}

\subsection{Proof of Theorem \ref{Main-theo-weig-sch-est}}\label{sub-sec-proof-the-Lun}
From now on, let $(M^n,g)$ be a complete Riemannian manifold satisfying the assumptions of theorem \ref{Main-theo-weig-sch-est}.

\subsubsection{Uniqueness}
 
\begin{prop}\label{uni-ppe-max}
Let $\lambda$ be a positive real number and $h\in \cap_{p\geq 1}W^{2,p}_{loc}(M,E)$ be a bounded tensor such that $\Delta_{V}h-\lambda h=0$. Then $h=0$.
\end{prop}
\begin{proof}

First of all, by elliptic regularity, $h$ is smooth. Consider the mollified norm of $h$ : $h_{\epsilon}:=\sqrt{\arrowvert h\arrowvert^2+\epsilon^2}$, for $\epsilon>0$. Then, 
\begin{eqnarray*}
\Delta_{V}h_{\epsilon}-\lambda h_{\epsilon}&=&\frac{1}{h_{\epsilon}}\left(<h,\Delta_{V}h>-\lambda h_{\epsilon}^2+\arrowvert\nabla h\arrowvert^2-\frac{\arrowvert\nabla\arrowvert h\arrowvert^2\arrowvert^2}{4h_{\epsilon}^2}\right)\\
&\geq&-\lambda\epsilon.
\end{eqnarray*}
Now, if $h_{\epsilon,k}:=h_{\epsilon}-\phi/k$ for $k\geq 1$, then $\lim_{k\rightarrow+\infty}\sup_Mh_{\epsilon,k}=\sup_Mh_{\epsilon}$ and since $\phi$ is an exhaustion function, $\sup_Mh_{\epsilon,k}=h_{\epsilon,k}(x_k)$ for some $x_k\in M$. Without loss of generality, $\phi$ can be assumed to be nonnegative. As $h_{\epsilon,k}$ satisfies
\begin{eqnarray*}
\Delta_{V}h_{\epsilon,k}-\lambda h_{\epsilon,k}\geq -\lambda\epsilon -k^{-1}\sup_M(\Delta_{V}\phi),
\end{eqnarray*}
one has, when evaluating the previous differential inequality at $x_k$,
\begin{eqnarray*}
\lambda\sup_M h_{\epsilon,k}\leq \lambda\epsilon+k^{-1}\sup_M(\Delta_{V}\phi).
\end{eqnarray*}

By letting $k$ go to $+\infty$, we have $\lambda\sup_M h_{\epsilon}\leq \lambda\epsilon$.\\

By letting $\epsilon$ go to $0$, we get $\lambda\sup_M \arrowvert h\arrowvert\leq0$.

\end{proof}

 \subsubsection{A priori estimates}
As in \cite{Lun-Sch-Est}, one needs to solve the Cauchy problem associated to $\Delta_{V}$ first.
For that purpose, the main step is to find a priori estimates on covariant derivatives of the solution to the Cauchy problem up to order $3$. More precisely,
\begin{theo}\label{a-priori-cauchy-sol}
Let $h_0\in C^0(M,E)$ and let $(h(t,\cdot))_{t\in[0,T)}$ be a bounded smooth classical solution of the Cauchy problem with initial condition $h_0$. Assume that, for some integer $k\geq 1$, there exists a constant $K(k)$ such that
\begin{eqnarray}
\|\Rm(g)\ast V\|_{C^k(M,E)}+\|\Rm(g)\|_{C^k(M,E)}+\|\nabla V\|_{C^{k-1}(M,E)}\leq K(k),
\end{eqnarray}
where $\Rm(g)\ast {V}:=\Rm(g)({V},\cdot,\cdot,\cdot)$.\\

 Then, for any $T>0$, there exists $C=C(n,\lambda,k,T)$ such that
 \begin{eqnarray*}
\arrowvert h(t)\arrowvert^2_{C^0(M,E)}+\sum_{i=1}^k\frac{(\alpha t)^i}{i}\arrowvert\nabla^ih(t)\arrowvert^2_{C^0(M,E)}
\leq C\arrowvert h_0\arrowvert^2_{C^0(M,E)}.
\end{eqnarray*}
Moreover, if $(M^n,g,\nabla f)$ is an expanding gradient Ricci soliton, and ${V}:=\nabla f$, then, if $E=M\times\mathbb{R}$,
\begin{eqnarray*}
\arrowvert u(t)\arrowvert_{C^0(M)}^2+\alpha t\arrowvert\nabla u(t)\arrowvert_{C^0(M)}^2\leq t^{\alpha-2}e^{t}\arrowvert u_0\arrowvert_{C^0(M)}^2,
\end{eqnarray*}
for any positive time $t$ and any positive $\alpha$.
\end{theo}
\begin{proof}[Proof of theorem \ref{a-priori-cauchy-sol}]
The proof consists in deriving the evolution equation satisfied by $$s(t,x):=\arrowvert h\arrowvert^2(t,x)+\sum_{i=1}^k\frac{(\alpha t)^i}{i}\arrowvert \nabla^ih\arrowvert^2(t,x),$$ where $\alpha$ is to be defined later. Then, one applies the maximum principle to finish the proof. The computation is lengthy and similar to \cite{Lun-Sch-Est}. Nonetheless, it seems that this computation gives new results in the case of an expanding gradient Ricci soliton for functions ($E=M\times \mathbb{R}$) that are interesting in their own right.

On one hand, 
\begin{eqnarray*}
\partial_ts&=&2<h,\Delta_{V}h>+\alpha\sum_{i=1}^k(\alpha t)^{i-1}\arrowvert\nabla^i h\arrowvert^2\\
&&+2\sum_{i=1}^k\frac{(\alpha t)^i}{i}<\nabla^i(\Delta_{V}h),\nabla^ih>.
\end{eqnarray*}
On the other hand,
\begin{eqnarray*}
\Delta_{V}s&=&2\arrowvert\nabla h\arrowvert^2+2\sum_{i=1}^k\frac{(\alpha t)^i}{i}\arrowvert\nabla^{i+1}h\arrowvert^2\\
&&+2<\Delta_{V}h,h>+2\sum_{i=1}^k\frac{(\alpha t)^i}{i}<\Delta_{V}(\nabla^ih),\nabla^ih>,\\
\end{eqnarray*}
Therefore, if $k\geq 2$,
\begin{eqnarray*}
\partial_ts-\Delta_{V}s&=&(\alpha-2)\arrowvert\nabla h\arrowvert^2+\sum_{i=1}^{k-1}\left(\alpha-\frac{2}{i}\right)(\alpha t)^i\arrowvert\nabla^{i+1}h\arrowvert^2-\frac{2(\alpha t)^k}{k}\arrowvert\nabla^{k+1}h\arrowvert^2\\
&&+2\sum_{i=1}^k\frac{(\alpha t)^i}{i}<[\nabla^i,\Delta_{V}]h,\nabla^ih>.
\end{eqnarray*}
If $k=1$, then,
\begin{eqnarray*}
\partial_ts-\Delta_{V}s&=&(\alpha-2)\arrowvert\nabla h\arrowvert^2-2(\alpha t)\arrowvert\nabla^{2}h\arrowvert^2+2\alpha t<[\nabla,\Delta_{V}]h,\nabla h>.
\end{eqnarray*}

It remains to identify the commutators $[\nabla^k,\Delta_{V}]h$ for $k\geq 1$.
In general, one can prove by induction on $k$ that
\begin{eqnarray*}
[\nabla^k,\Delta]h&=&\sum_{j=0}^k\nabla^jh\ast \nabla^{k-j}\Rm(g),
\end{eqnarray*}
and
\begin{eqnarray*}
[\nabla^k,\nabla_{{V}}] h&=&\sum_{j=0}^{k-1}\nabla^{k-j}{V}\ast\nabla^{j+1}h+\nabla^{k-1-j}(\Rm(g)\ast {V})\ast\nabla^jh,\\
\end{eqnarray*}
where $\Rm(g)\ast {V}:=\Rm(g)({V},\cdot,\cdot,\cdot)$.\\

In particular, by assumptions made on the vector field and the curvature operator, one has, for any $i\in\{1,...,k\}$,
\begin{eqnarray*}
\arrowvert[\nabla^i,\Delta_{V}]h\arrowvert\leq C\sum_{j=0}^i\arrowvert\nabla^jh\arrowvert.
\end{eqnarray*}
Therefore, if $k\geq 2$ and if $\alpha\leq \frac{2}{k-1}$, then
\begin{eqnarray*}
\partial_ts-\Delta_{V}s&\leq&C\sum_{i=1}^k\sum_{j=0}^i\frac{(\alpha t)^{\frac{j}{2}}}{i}\arrowvert\nabla^jh\arrowvert(\alpha t)^{\frac{i-j}{2}}\arrowvert\nabla^ih\arrowvert\\
&\leq&C(\alpha,k,T)s,
\end{eqnarray*}
for $t\in(0,T]$.\\

Again, if $k=1$, and if $E=M\times \mathbb{R}$, one has the more precise estimates : 
\begin{eqnarray*}
[\nabla,\Delta_{{V}}] u=-\Ric(g)(\nabla u)+\nabla {V}_k\nabla_k u,
\end{eqnarray*}
where $u$ is a smooth function on $M^n$. In particular, if $(M^n,g,\nabla f)$ is an expanding gradient Ricci soliton and if ${V}=\nabla f$ then,
\begin{eqnarray*}
\partial_t s-\Delta_{ f}s&=&(\alpha-2+\alpha t)\arrowvert\nabla u\arrowvert^2-2(\alpha t)\arrowvert\nabla^{2}u\arrowvert^2,
\end{eqnarray*}
where $s:=u^2+\alpha t \arrowvert\nabla u\arrowvert^2$. Therefore, 
\begin{eqnarray*}
\partial_t s-\Delta_{\nabla f}s &\leq&(\alpha-2+\alpha t)\arrowvert\nabla u\arrowvert^2\\
&\leq&\left(\frac{\alpha-2}{t}+1\right)s,
\end{eqnarray*}
if $\alpha=2$. 

The estimates follow  by applying the maximum principle in a similar way we did in the proof of proposition \ref{uni-ppe-max} together with the fact from the theory of parabolic equations that $\lim_{t\rightarrow 0} t^{k/2}\nabla^kh(t,x)=0$ for $x\in M^n$ and $k\geq 1$.

\end{proof}
\subsubsection{The Cauchy problem}
The second step consists in solving the Cauchy problem :
\begin{theo}[Lunardi]\label{Cau-Pro}
For any $h_0\in C^0(M,E)$, there exists a unique bounded classical solution $h(t)$ to the Cauchy problem
\[
\left\{
\begin{array}{rl}
&\partial_th=\Delta_{V}h \quad\mbox{on}\quad M^n\times (0,+\infty),\\
&\\
& h(0,x)=h_0(x), \quad x\in M^n.
\end{array}
\right.
\]

\end{theo}

\begin{proof}[Sketch of proof of theorem \ref{Cau-Pro}]

Uniqueness follows as in proposition \ref{uni-ppe-max}. 

Indeed, let $(h(t))_{t\in[0,T]}$ be a smooth bounded solution to the Cauchy problem. Then, by applying the maximum principle to $ h_{\epsilon}(t)-\phi/k$, where $k$ is a positive integer and where $h_{\epsilon}(t)$ is the mollified norm of $h(t)$ defined in the proof of proposition \ref{uni-ppe-max}, one shows that $\sup_{M\times [0,T]}\arrowvert h(t)\arrowvert  \leq \sup_M \arrowvert h_0\arrowvert.$\\

Concerning the existence, this is done by approximating the Cauchy problem by standard ones, i.e. with second order differential operators with bounded coefficients. Indeed, let $k$ be a positive integer, let $F$ be a smooth exhaustion function on $M$ such that 
\begin{eqnarray*}
 F(x)\leq c(r_p(x)+1), \quad \forall x\in M,\quad \arrowvert\nabla F\arrowvert+\arrowvert\nabla^2F\arrowvert\leq c,
\end{eqnarray*}
for some point $p\in M$ and some positive constant $c$. The existence of such function uses only the boundedness of the curvature and is ensured by theorem $3.6$ of \cite{Shi-Kah}. In the case of a Ricci expander, the square root of $v$ would be a perfect candidate as soon as $f$ is proper.

 Now, let $\psi:\mathbb{R}_+\rightarrow\mathbb{R}_+$ be a smooth function such that $\psi(x)=1$ if $0\leq x\leq 1 $ and $\psi(x)=0$ if $x\geq 2$. Then consider $\psi_k(x):=\psi(F(x)/k)$ and define a smooth vector field ${V}_k$ on $M$ by ${V}_k:=\psi_k{V}$. By construction, the vector field ${V}_k$ is bounded on $M$ and its derivatives are bounded uniformly in $k$. The sequence $({V}_k)_k$ converges uniformly on compact sets to ${V}$ as $k$ tends to $+\infty$. Therefore, by the theory of parabolic equations with bounded coefficients, the Cauchy problem
 
 \[
\left\{
\begin{array}{rl}
&\partial_th_k=\Delta_{{V}_k}h_k \quad\mbox{on}\quad M\times (0,+\infty),\\
&\\
& h_k(0,x)=h_0(x), \quad x\in M,
\end{array}
\right.
\]

has a unique classical bounded solution $(h_k(t))_{t\in(0,+\infty)}$. One can apply the a priori estimates from theorem \ref{a-priori-cauchy-sol} to bound the covariant derivatives of the sequence $(h_k(t))_k$ independently of $k$ by invoking the maximum principle with the help of the exhaustion function $\phi$ and by using the assumptions $\sup_M\arrowvert\Delta_g\phi\arrowvert+\arrowvert\nabla_{V}\phi\arrowvert<+\infty$ : 
\begin{eqnarray*}
\sup_M\arrowvert\nabla_{{V}_k}\phi\arrowvert\leq\sup_M\arrowvert\nabla_{{V}}\phi\arrowvert<+\infty.
\end{eqnarray*}

Then, one shows that the sequence $(h_k(t))_{t\in(0,+\infty)}$ is equicontinuous on compact sets of $[0,+\infty)\times M$ by the same arguments used in \cite{Lun-Sch-Est} in order to solve the initial Cauchy problem.

\end{proof}

To sum it up, we are able to define a semigroup of linear operators $T(t)$ in $C^0(M,E)$ by 
\begin{eqnarray*}
(T(t)h_0)(x):=h(t,x),\quad t\geq 0,\quad x\in M,\quad h_0\in C^0(M,E),
\end{eqnarray*}
where $(h(t))_{t\in(0,+\infty)}$ is the unique solution to the Cauchy problem given by theorem \ref{Cau-Pro}. By the semigroup law and a priori estimates given by theorem \ref{a-priori-cauchy-sol} applied to $T=1$, one gets,
\begin{eqnarray*}
\|T(t)\|_{\Li(C^0(M,E),C^k(M,E))}&\leq&\frac{Ce^{\omega t}}{t^{k/2}},\quad k\in\{1,2,3\},\\
\|T(t)\|_{\Li(C^3(M,E))}&\leq&Ce^{\omega t},\quad t>0,
\end{eqnarray*}
for some real number $\omega$ and some constant $C$ time independent. In order to get Schauder estimates, we need to recall several facts about real interpolation on Banach spaces. The notations and the presentation are taken from the notes \cite{Lun-Not-Int}.\\

 If $Y\subset X$ are two Banach spaces, then one defines the interpolation space $(X,Y)_{\theta,\infty}$ for $\theta\in(0,1)$ by
 \begin{eqnarray*}
(X,Y)_{\theta,\infty}:=\{h\in X\quad|\quad\|h\|_{\theta,\infty}:=\sup_{t\in(0,1)}t^{-\theta}K(t,h,X,Y)<+\infty\},
\end{eqnarray*}
where
\begin{eqnarray*}
K(t,h,X,Y):=\inf\{\|a\|_X+t\|b\|_Y\quad|\quad h=a+b,\quad (a,b)\in X\times Y\}.
\end{eqnarray*}

One can check that $(X,Y)_{\theta,\infty}$ is a Banach space. In fact, we are able to identify this interpolation space with classical Hölder spaces in our setting. We believe that the following proposition is well-known but, in order to keep this paper as self-contained as possible, we provide a proof of it. 

\begin{prop}\label{Int-Spa-Hol}
Let $(M,g)$ be a complete Riemannian manifold with bounded curvature, bounded covariant derivatives of the curvature operator and positive injectivity radius. Then the following holds.
\begin{enumerate}
\item For $\theta\in(0,1)$, and $k\in\mathbb{N}$,
\begin{eqnarray}
(C^k(M,E),C^{k+1}(M,E))_{\theta,\infty}=C^{k,\theta}(M,E).\label{Int-Spa-Unresc-Sch}
\end{eqnarray}
\item Let $0\leq \theta_1\leq\theta_2\leq 1$, $0\leq\theta\leq 1$.Then, if $(1-\theta)\theta_1+\theta\theta_2$ is not an integer,
\begin{eqnarray*}
(C^{\theta_1}(M,E),C^{\theta_2}(M,E))_{\theta,\infty}=C^{(1-\theta)\theta_1+\theta\theta_2}(M,E),
\end{eqnarray*}
where $C^{\theta}(M,E):=C^{\lfloor\theta\rfloor,\theta-\lfloor\theta\rfloor}(M,E)$.\\


\end{enumerate}
\end{prop}

\begin{proof}[Proof of proposition \ref{Int-Spa-Hol}]
\begin{itemize}
\item We only prove $(C^0(M,E),C^{1}(M,E))_{\theta,\infty}=C^{0,\theta}(M,E).$ \\

The proof for an arbitrary nonnegative integer $k$ is similar. \\

 Concerning the first inclusion, let $h\in(C^0(M,E),C^{1}(M,E))_{\theta,\infty}$. Then,
\begin{eqnarray*}
\| h\|_{C^0(M,E)}\leq K(1,h,C^0(M,E),C^1(M,E))\leq \| h\|_{\theta,\infty}.
\end{eqnarray*}

Moreover, for a decomposition $h=a+b$, with $a\in C^0(M,E)$ and $b\in C^1(M,E)$, 
\begin{eqnarray*}
\arrowvert h(x)-h(y)\arrowvert&\leq& \arrowvert a(x)-a(y)\arrowvert+\arrowvert b(x)-b(y)\arrowvert\\
&\leq&2\| a\|_{C^0(M,E)}+\| b\|_{C^1(M,E)}d_g(x,y),
\end{eqnarray*}
for any $x,y\in M$. By taking the infimum over such decompositions of $h$, one gets,
\begin{eqnarray*}
\arrowvert h(x)-h(y)\arrowvert&\leq& 2K(d_g(x,y)/2,h,C^0(M,E),C^1(M,E))\\
&\leq&2^{1-\theta} d_g(x,y)^{\theta}\| h\|_{\theta,\infty}.
\end{eqnarray*}

To prove the converse inclusion, in the Euclidean case $M^n=\mathbb{R}^n$, one usually considers the convolution of $h$ by a compactly supported function. In general, this procedure is meaningful at small scales only. Then, the proof consists in building a suitable partition of unity in order to patch these local convolutions. 

More precisely, we will use the following lemma due to Cheeger and Gromov \cite{Che-Gro-Cha-Num} :
\begin{lemma}[Cheeger-Gromov]\label{Goo-Cov-Che-Gro}
Let $(M^n,g)$ be a complete Riemmannian manifold with $\Ric(g)\geq-g$. Fix $\rho_0>0$ and $\lambda>1$. Then, for all $0<\rho\leq\rho_0$, there is a covering of $M^n$ by sets $U_1$,..., $U_N$ such that
\begin{itemize}
\item Each $U_i$ is a union of disjoint metric balls of radius $\rho$, and the distance between the centers of each pair of balls is at least $2\lambda \rho$.
\item $N\leq N(n,\rho_0,\lambda).$
\end{itemize}
\end{lemma}

The next crucial observation that goes back to Gromov \cite{Gro-Boo} is that one only needs the exponential map (of some point) to be a local diffeomorphism in order to perform a local convolution. This works well in case of bounded curvature. Indeed, by standard comparison estimates [Chap. $6$, \cite{Pet}], if the sectional curvature $K_g$ is less than a constant $K$ then the exponential map of a point $p$ is non singular on the open ball $B(0_p,\pi/2\sqrt{K})$, where $0_p\in T_pM$.
From now on, assume (by rescaling the metric) that $\Ric(g)\geq -g$ and let $U_1$, ..., $U_N$ be a covering of $M^n$ given by \ref{Goo-Cov-Che-Gro} with $\rho\leq\rho_0$ positive small enough so that the exponential maps of the centers $(x_i)_i$ of the balls of the covering $(U_j)_{1\leq j\leq N}$ are non singular.

By lemma $5.3$ in \cite{Che-Gro-Bou-Neu}, we can regularize the distance functions to the points $(x_i)_i$ by a convolution defined on $B(x_i,\rho_0)$  by
\begin{eqnarray*}
f_{i,\delta}(x):=\frac{1}{\delta^n}\int_{T_xM}d_g(x_i,\exp_x(v))\chi\left(\frac{\arrowvert v\arrowvert}{\delta}\right)d\mu(v),
\end{eqnarray*}
where $d\mu$ is the Lebesgue measure induced on $T_xM$ by the metric $g$, where $\chi$ is a nonnegative smooth cut-off function of unit mass, i.e. $\int_{\mathbb{R}^n}\chi=1$ and where $\delta$ is a positive parameter small enough to be defined later. One can show that 
\begin{eqnarray*}
\arrowvert f_{i,\delta}(\cdot)-d_g(x_i,\cdot)\arrowvert\leq C\delta, \quad\arrowvert\nabla^l f_{i,\delta}\arrowvert\leq c(n,k_0)\delta^{1-l},
\end{eqnarray*}
for $l\geq 1$, where $k_0$ is a bound on the curvature and its covariant derivatives up to a certain order depending on $l$ and $C$ is a positive constant independent of $\delta$. Choose $\delta$ proportional to $\rho$ such that 
\begin{eqnarray*}
\{f_{i,\delta}\leq \rho\}\subset B_g(x_i,2\rho).
\end{eqnarray*}
Let $\psi:\mathbb{R}_+\rightarrow\mathbb{R}_+$ be a smooth function such that $\psi\equiv 1$ on $[0,1]$ and $\psi\equiv 0$ on $[2,+\infty)$. Then define
\begin{eqnarray*}
\widetilde{\phi_{i,\rho}}(x):=\psi\left(\frac{2f_{i,\delta}(x)}{\rho}\right),
\end{eqnarray*}
so that $\supp \widetilde{\phi_{i,\rho}}\subset B_g(x_i,2\rho)$. By construction, the sum $\sum_i\widetilde{\phi_{i,\rho}}$ is well-defined, and $\sum_i\widetilde{\phi_{i,\rho}}\geq 1$.
 
We get a partition of unity $(\phi_{i,\rho})_i$ for the covering $(B_g(x_i,2\rho))_i$ with $\lambda=3$ as in lemma \ref{Goo-Cov-Che-Gro} by considering 
\begin{eqnarray*}
\phi_{i,\rho}:=\frac{\widetilde{\phi_{i,\rho}}}{\sum_i\widetilde{\phi_{i,\rho}}}.
\end{eqnarray*}

We are now in a position to prove the converse inclusion. The proof mimics the proof in the Euclidean case (e.g. \cite{Lun-Not-Int}). Define, for $\rho\leq \rho_0$, and $h\in C^{\theta}(M,E)$,
\begin{eqnarray*}
b_{\rho}:=\sum_{i}\phi_{i,\rho}h(x_i),\quad a_{\rho}:= h-b_{\rho}.
\end{eqnarray*}

By construction, one checks that 
\begin{eqnarray*}
|a_{\rho}(x)|&\leq& \sum_i\phi_{i,\rho}(x)\arrowvert h(x_i)-h(x)\arrowvert\\&\leq& \sum_i\phi_{i,\rho}(x) d_g^{\theta}(x_i,x)\| h\|_{C^{\theta}(M,E)}\\
&\leq&(2\rho)^{\theta}\|h\|_{C^{\theta}(M,E)},\\
\|b_{\rho}\|_{C^0(M,E)}&\leq&\|h\|_{C^{0}(M,E)}\leq\|h\|_{C^{\theta}(M,E)}\\
|\nabla b_{\rho}(x)|&\leq&\sum_i|\nabla\phi_{i,\rho}||h(x_i)-h(x)|\\
&\leq&C\rho^{\theta-1}\|h\|_{C^{\theta}(M,E)},
\end{eqnarray*}
where $C=C(N)$ with $N$ as in lemma \ref{Goo-Cov-Che-Gro}. Therefore,
\begin{eqnarray*}
\rho^{-\theta}K(\rho,h,C^0(M,E),C^1(M,E))&\leq&\rho^{-\theta}(\|a_{\rho}\|_{C^0(M,E)}+\rho\|b_{\rho}\|_{C^1(M,E)})\\
&\leq& C\|h\|_{C^{\theta}(M,E)},
\end{eqnarray*}
i.e. $h\in (C^0(M,E),C^1(M,E))_{\theta,\infty}$.\\

\item Again, we only prove $(C^0(M,E),C^2(M,E))_{\theta,\infty}=C^{2\theta}(M,E)$, if $2\theta< 1$. The other cases can be proved similarly. The main ingredient of the proof is the Reiteration theorem (e.g. \cite{Lun-Not-Int}). \\

We need to recall two definitions first. \\

Let $X,Y,E$ be Banach spaces such that $Y\subset E\subset X$. Let $\theta\in[0,1]$.\\
\begin{itemize}
\item a Banach space $E$ belongs to the class $J_{\theta}(X,Y)$ if there exists a positive constant $c$ such that $\|y\|_E\leq c\|y\|_X^{1-\theta}\|y\|_Y^{\theta}$, for any $y\in Y$.\\
\item a Banach space $E$ belongs to the class $K_{\theta}(X,Y)$ if there exists a positive constant $c$ such that  $t^{-\theta}K(t,x,X,Y)\leq c\|x\|_E$, for any $x\in E$ and $t\in(0,1)$.
\end{itemize}

\begin{claim}\label{reit-verif}
$C^1(M,E)\in J_{1/2}(C^0(M,E),C^2(M,E))\cap K_{1/2}(C^0(M,E),C^2(M,E)).$
\end{claim}

\begin{proof}[Proof of claim \ref{reit-verif}]
The proof of $C^1(M,E)\in J_{1/2}(C^0(M,E),C^2(M,E))$ is standard. We reproduce it for the convenience of the reader. It is adapted from \cite{Lun-Not-Int}.
It suffices to prove it for functions on $M$. Let $T$ be a function on $M$ belonging to $C^2(M,\mathbb{R}):=C^2(M,E)$ where $E=M\times \mathbb{R}$. Let $x\in M$, $v\in T_xM$ a vector of unit length and $\gamma$ be a geodesic such that $\gamma(0)=x$ and $\gamma'(0)=v$. Then, for any positive $t$,
\begin{eqnarray*}
\arrowvert T(\gamma(t))-T(x)-<\nabla T(x),v>t\arrowvert\leq \frac{1}{2}\arrowvert\nabla^2T\arrowvert_{C^0(M,\mathbb{R})}t^2.
\end{eqnarray*}
Hence,
\begin{eqnarray*}
\arrowvert<\nabla T(x),v>\arrowvert\leq \frac{1}{2}\|T\|_{C^2(M,\mathbb{R})}t+\frac{2\| T\|_{C^0(M,\mathbb{R})}}{t},
\end{eqnarray*}
for any positive $t$. By minimizing the right hand side of the previous inequality on $t$, one gets
\begin{eqnarray*}
\arrowvert<\nabla T(x),v>\arrowvert\leq C\|T\|^{1/2}_{C^2(M,\mathbb{R})}\| T\|^{1/2}_{C^0(M,\mathbb{R})}.
\end{eqnarray*}

Let us prove $ C^1(M,E)\in K_{1/2}(C^0(M,E),C^2(M,E)).$ With the same notations as before, consider the decomposition : $h=a_{\rho}+b_{\rho}$. Then, by construction of the partition of unity, one checks that
\begin{eqnarray*}
\|a_{\rho}\|_{C^0(M,E)}&\leq& (2\rho)\|h\|_{C^1(M,E)},\\
\|b_{\rho}\|_{C^{1}(M,E)}&\leq&C \|h\|_{C^1(M,E)},\\
\|\nabla^2b_{\rho}\|_{C^{0}(M,E)}&\leq& C\rho^{-1}\|h\|_{C^1(M,E)},
\end{eqnarray*}
hence the claim since 
\begin{eqnarray*}
\rho^{-1/2}K(\rho,h,C^{0}(M,E),C^{2}(M,E))&\leq& \rho^{-1/2}\left(\|a_{\rho^{1/2}}\|_{C^{0}(M,E)}+\rho\|b_{\rho^{1/2}}\|_{C^{2}(M,E)}\right)\\
&\leq& C\|h\|_{C^{1}(M,E)}.
\end{eqnarray*}

\end{proof}
Now, if $2\theta<1$, by applying the Reiteration theorem to $E_1:=C^0(M,E)\in J_{0}(C^0(M,E),C^2(M,E))\cap K_{0}(C^0(M,E),C^2(M,E))$ and to $E_2:=C^{1}(M,E)$, we get 
\begin{eqnarray*}
(E_1,E_2)_{2\theta,\infty}&=&(C^0(M,E),C^2(M,E))_{(1-2\theta)\times0+2\theta \times 1/2,\infty},
\end{eqnarray*} 
that is,
\begin{eqnarray*}
(C^0(M,E),C^2(M,E))_{\theta,\infty}=C^{2\theta}(M,E),
\end{eqnarray*}
by the previous results.
\end{itemize}

\end{proof}
Therefore, by the interpolation theorem and proposition \ref{Int-Spa-Hol}, we get the
\begin{coro}\label{inter-sch-semi-gp}
With the previous notations, 
\begin{eqnarray*}
\|T(t)\|_{\Li(C^{\theta}(M,E),C^{\alpha}(M,E))}\leq\frac{Ce^{\omega t}}{t^{\frac{\alpha-\theta}{2}}},\quad 0\leq\theta\leq\alpha\leq 3,
\end{eqnarray*}
where $C$ is a time independent positive constant (depending on $\theta$ and $\alpha$). 
\end{coro}
\subsubsection{End of the proof of theorem \ref{Main-theo-weig-sch-est}}

As noticed in \cite{Lun-Sch-Est}, the semigroup $T(t)$ is not necessarily strongly continuous on $C^0(M,E)$. Nevertheless, one can manage to make sense of a realization of the operator $\Delta_{V}$ in $C^0(M,E)$. Indeed, for any $\lambda>0$, consider the operator
\begin{eqnarray*}
(R(\lambda)h)(x):=\int_0^{\infty}e^{-\lambda t}(T(t)h)(x)dt,\quad x\in M.
\end{eqnarray*}
This is well-defined since $\|T(t)h\|_{C^0(M,E)}\leq \|h\|_{C^0(M,E)}$ for any positive time $t$. Moreover, 
\begin{eqnarray*}
\|R(\lambda)\|_{\Li(C^0(M,E))}\leq\frac{1}{\lambda}.
\end{eqnarray*}

Since $(T(t))_t$ is a semigroup, $R(\cdot)$ satisfies the first resolvent identity, i.e. for any positive $\lambda,\mu$, $R(\lambda)-R(\mu)=(\mu-\lambda)R(\lambda)\circ R(\mu)$. Finally, as $R(\lambda)(h)(x)$ is the Laplace transform of the tensor $t\rightarrow T(t)(h)(x)$, $R(\lambda)$ is injective. Therefore, there exists a closed operator 
\begin{eqnarray*}
A:D(A)\rightarrow C^0(M,E),\quad D(A)=\mathrm{Range} (R(\lambda)),\quad \lambda >0,
\end{eqnarray*}
such that $R(\lambda)$ is the resolvent of $A$, i.e. $R(\lambda) =R(\lambda,A)$. As in proposition $4.1$ of \cite{Lun-Sch-Est}, one identifies $D(A)$ :
\begin{prop}[Lunardi]\label{iden-dom-op}
\begin{eqnarray*}
&&D(A)=\{h\in\cap_{p\geq 1}W_{loc}^{2,p}(M,E)\cap C^{0}(M,E)\quad : \quad \Delta_{V}h\in C^0(M,E)\}=:D^2_{V}(M,E),\\
&& Ah=\Delta_{V}h,\quad\forall h\in D(A).
\end{eqnarray*}
Moreover, for any $\theta \in (0,2)$, there is some positive constant $C$ such that 
\begin{eqnarray*}
\| h\|_{C^{\theta}(M,E)}\leq C\|h\|_{C^0(M,E)}^{1-\theta/2}\|h\|_{D(A)}^{\theta/2},\quad \forall h\in D(A),
\end{eqnarray*}
where $\|h\|_{D(A)}:=\|h\|_{C^0(M,E)}+\|Ah\|_{C^0(M,E)}$.
\end{prop}

We are now in a position to finish the proof of theorem \ref{Main-theo-weig-sch-est}. Let $H\in C^{\theta}(M,E)$ and $\lambda >0$. Then, 
\begin{eqnarray*}
h(x):=\int_0^{\infty}e^{-\lambda t}(T(t)H)(x)dt,\quad x\in M,
\end{eqnarray*}
 is well-defined, $h\in D(A)$ and satisfies $\lambda h-\Delta_{V} h=H$. It remains to show that $h\in C^{2,\theta}(M,E)$.
 
 First, $h\in C^{\theta}(M,E)$ by proposition \ref{iden-dom-op} since for any $h\in D(A)$,
 \begin{eqnarray*}
\| h\|_{C^{\theta}(M,E)}&\leq& C\|h\|_{C^0(M,E)}^{1-\theta/2}\|h\|_{D(A)}^{\theta/2},\\
&\leq&\frac{C}{\lambda^{1-\theta/2}}\|H\|_{C^{\theta}(M,E)}.
\end{eqnarray*}

Now, if $\omega\in\mathbb{R}$ is as in corollary \ref{inter-sch-semi-gp}, let $\eta>\omega$. Then $h$ satisfies $\eta h-\Delta_{V}h=H+(\eta-\lambda)h=:\tilde{H},$ with
\begin{eqnarray*}
\|\tilde{H}\|_{C^{\theta}(M,E)}\leq\left(1+\frac{C|\eta-\lambda|}{\lambda^{1-\theta/2}}\right)\|H\|_{C^{\theta}(M,E)}.
\end{eqnarray*}

As $\eta>\omega\geq0$, one can represent $h$ in terms of $\tilde{H}$ as well : 
\begin{eqnarray*}
h(x)=\int_0^{\infty}e^{-\eta t}(T(t)\tilde{H})(x)dt.
\end{eqnarray*}

According to proposition \ref{Int-Spa-Hol}, it suffices to prove that $h\in(C^{\alpha}(M,E),C^{2,\alpha}(M,E))_{1-(\alpha-\theta)/2,\infty}$ for any $\alpha\in(\theta,1)$. As in \cite{Lun-Sch-Est}, one considers the following decomposition of $h$,
\begin{eqnarray*}
h=a_{\rho}+b_{\rho},\quad a_{\rho}(x):=\int_0^{\rho}e^{-\eta t}(T(t)\tilde{H})(x)dt,\quad x\in M,\quad\rho >0.
\end{eqnarray*}

Again, by corollary \ref{inter-sch-semi-gp}, one estimates
\begin{eqnarray*}
\|a_{\rho}\|_{C^{\alpha}(M,E)}&\leq& C\rho^{1-(\alpha-\theta)/2}\|\tilde{H}\|_{C^{\theta}(M,E)},\\
\|b_{\rho}\|_{C^{2,\alpha}(M,E)}&\leq& C \rho^{-(\alpha-\theta)/2}\|\tilde{H}\|_{C^{\theta}(M,E)}.
\end{eqnarray*}

Hence, $h\in C^{2,\theta}(M,E)$ and $\|h\|_{C^{2,\theta}(M,E)}\leq C\|H\|_{C^{0,\theta}(M,E)}$ with $C$ independent of $h$.

\subsection{Solutions in $D^{2+k,\theta}_{\nabla f}(M,E)$}\label{Sec-Sol-C^k}
We now study the realization of the weighted laplacian in function spaces with higher regularity. We begin with a uniqueness statement about subsolutions of weighted elliptic equations with controlled growth at infinity :

\begin{prop}\label{ppe-max-pol-sub-sol}
Let $(M^n,g,\nabla f)$ be an expanding gradient Ricci soliton such that $f$ is proper. Let $u:M^n\rightarrow\mathbb{R}$ be a $C_{loc}^2$ function such that 
\begin{eqnarray*}
\Delta_{v-\gamma\ln v}u-\alpha u\geq 0,\quad u=\textit{O}(v^{\beta}),
\end{eqnarray*}

 where $\alpha$, $\beta$ are real numbers such that $\min\{1,\alpha\}>\beta(\geq0)$ and where $\gamma$ is any real number. Then,
 $\sup_{M^n}u\leq 0$.

\end{prop}

\begin{proof}
Let $k$ be a positive integer and let $\epsilon$ be a positive number such that $\beta+\epsilon<\min\{1,\alpha\}$. Then,
\begin{eqnarray}
\Delta_{v-\gamma\ln v}\left(u-\frac{v^{\beta+\epsilon}}{k}\right)&\geq& \alpha u-\alpha \frac{v^{\beta+\epsilon}}{k}-\frac{C(\beta,\epsilon,\gamma)}{k}\arrowvert\nabla\ln v\arrowvert^2v^{\beta+\epsilon}\\
&\geq&\alpha\left(u-\frac{v^{\beta+\epsilon}}{k}\right)-\frac{C(\beta,\epsilon,\gamma)}{k}\arrowvert\nabla\ln v\arrowvert^2v^{\beta+\epsilon}.\label{sub-sol-cov-der}
\end{eqnarray}
On the other hand, by the very choice of $\epsilon$ and as $f$, hence $v$, is proper, $$\limsup_{+\infty}\left(u-\frac{v^{\beta+\epsilon}}{k}\right)=-\infty,$$ for any positive integer $k$. Let $x_k\in M$ be a point such that 
\begin{eqnarray*}
\sup_{M}\left(u-\frac{v^{\beta+\epsilon}}{k}\right)=\left(u-\frac{v^{\beta+\epsilon}}{k}\right)(x_k).
\end{eqnarray*}
Then, according to the maximum principle applied to $\left(u-v^{\beta+\epsilon}/k\right)$, inequality (\ref{sub-sol-cov-der}) gives
\begin{eqnarray*}
\alpha\left(u-\frac{v^{\beta+\epsilon}}{k}\right)(x_k)\leq\frac{C(\beta,\epsilon,\gamma)}{k}\arrowvert\nabla\ln v\arrowvert^2v^{\beta+\epsilon}(x_k).
\end{eqnarray*}
Now, $\arrowvert\nabla\ln v\arrowvert^2\leq v^{-1}$, therefore, as $\beta+\epsilon\leq 1$,
 \begin{eqnarray*}
\sup_{M}\left(u-\frac{v^{\beta+\epsilon}}{k}\right)=\left(u-\frac{v^{\beta+\epsilon}}{k}\right)(x_k)\leq \frac{C(\beta,\epsilon,\gamma)}{\alpha k},
\end{eqnarray*}
which means that if $k$ tends to zero, $\sup_{M^n}u\leq 0$.

\end{proof}
 
 As a straightforward application of proposition \ref{ppe-max-pol-sub-sol}, we get 
 
\begin{coro}\label{cor-uni-sol-pol-gro}
Let $(M^n,g,\nabla f)$ be an expanding gradient Ricci soliton such that $f$ is proper. Let $T$ be a tensor such that $T\in\cap_{p\geq 1}W^{2,p}_{loc}(M,E)$ and such that
\begin{eqnarray*}
\Delta_{v-\gamma\ln v}T-\alpha T= 0,\quad T=\textit{O}(v^{\beta}),
\end{eqnarray*}

 where $\alpha$, $\beta$ are real numbers such that $\min\{1,\alpha\}>\beta(\geq0)$ and where $\gamma$ is any real number. Then,
 $T\equiv 0$.
\end{coro}

We are now in a position to state the main result of this section :

\begin{theo}\label{iso-weighted-lap-II}
Let $(M^n,g,\nabla f)$ be an asymptotically conical expanding gradient Ricci soliton. Let $H\in C^{k,\theta}_{con}(M,E)$, $k\geq1$, then there exists a unique $h\in D^{k+2,\theta}_{\nabla f}(M,E)$ such that 
\begin{eqnarray*}
\Delta_{v-2\alpha \ln v}h-\alpha h=H, \quad\alpha\in (1/2,+\infty), \quad\theta\in[0,1).
\end{eqnarray*}
Moreover, the following holds.
\begin{itemize}
\item There exists a positive constant $C$ (independent of $H$) such that 
\begin{eqnarray*}
\| h\|_{D^{k+2,\theta}_{\nabla f}(M,E)}\leq C\|H\|_{C^{k,\theta}_{con}(M,E)},
\end{eqnarray*}
i.e. the operator
\begin{eqnarray*}
\Delta_{v-2\alpha\ln v}-\alpha:D^{k+2,\theta}_{\nabla f}(M,E)\rightarrow C^{k,\theta}_{con}(M,E),
\end{eqnarray*}
is an isomorphism of Banach spaces.
\item For $\theta\in(0,1)$, the space 
\begin{eqnarray*}
\C^{k;1,\theta}(M,E)&:=&\left\{h\in C^{k+1,\theta}_{loc}(M,E)\quad|\quad v^{i/2}\nabla^ih\in C^{1,\theta}(M,E),\quad \forall i=0,...,k\right\}\\
\end{eqnarray*}

 embeds continuously in $D^{k+2}_{\nabla f}(M,E)$.

\item There exists a positive constant $C$ such that, for $\theta\in(0,1)$,
\begin{eqnarray*}
\|h\|_{\C^{k;2,\theta}(M,E)}\leq C\|H\|_{C^{k,\theta}_{con}(M,E)},
\end{eqnarray*}
where,
\begin{eqnarray*}
\C^{k;2,\theta}(M,E)&:=&\left\{h\in C^{k+2,\theta}_{loc}(M,E)\quad|\quad v^{i/2}\nabla^ih\in C^{2,\theta}(M,E),\quad \forall i=0,...,k\right\}\\
\end{eqnarray*}
\end{itemize}
\end{theo}

\begin{proof}
Uniqueness comes from theorem \ref{first-iso-sch}. We will prove theorem \ref{iso-weighted-lap-II} by induction on $k$. We can start the induction with the case $k=0$ : this is exactly the content of theorem \ref{first-iso-sch} without any restriction on $\alpha$. Let $k$ be a positive integer and let $H\in C^{k,\theta}_{con}(M,E)$. In particular, $H\in C^{k-1}_{con}(M,E)$ and there exists a solution $h\in D^{2+k-1}_{\nabla f}(M,E)$ such that $\Delta_{v-2\alpha \ln v}h-\alpha h=H.$ Moreover, there exists a positive constant $C$ such that 
\begin{eqnarray}
&&\| h\|_{D^{2+k-1}_{\nabla f}(M,E)}\leq C\|H\|_{C^{k-1}_{con}(M,E)},\label{ind-1-wei-lap-II}\\
&&\|h\|_{\C^{k-1;1,\theta}(M,E)}\leq C\|h\|_{D^{2+k-1}_{\nabla f}(M,E)}^{\theta/2}\|h\|_{C^{k-1}_{con}(M,E)}^{1-\theta/2}.
\end{eqnarray}
In particular, we get 
\begin{eqnarray}
\|v^{{(k-1)}/2}\nabla^{k-1}h\|_{C^{1,\theta}(M,E)}\leq C\|H\|_{C^{k-1}_{con}(M,E)}.\label{ind-2-wei-lap-II}
\end{eqnarray}

We claim that $h\in D^{k+2,\theta}_{\nabla f}(M,E)$. 

First of all, as $H\in C^{k,\theta}_{loc}(M,E)$, $h\in C^{k+2,\theta}_{loc}(M,E)$ by elliptic regularity. We now consider $h_i:=v^{i/2}\nabla^i h$ for $i=0,...,k$. 

We compute the evolution of $h_k$ as follows :
\begin{eqnarray*}
\Delta_{v-(2\alpha+k)\ln v}h_k-\alpha h_k=(\Rm(g)+\arrowvert\nabla\ln v\arrowvert^2)\ast h_k+H_k+\sum_{i=0}^{k-1}v^{\frac{k-i}{2}}\nabla^{k-i}\Rm(g)\ast h_i,
\end{eqnarray*}
where $H_k:=v^{k/2}\nabla^kH$. Now, thanks to estimates (\ref{ind-1-wei-lap-II}) and (\ref{ind-2-wei-lap-II}) together with the fact that $(M^n,g,\nabla f)$ is asymptotically conical, the righthand side of the previous equation can be bounded by
\begin{eqnarray*}
\|(\Rm(g)+\arrowvert\nabla\ln v\arrowvert^2)\ast h_k\|_{C^{0,\theta}(M,E)}&+&\|H_k+\sum_{i=0}^{k-1}v^{\frac{k-i}{2}}\nabla^{k-i}\Rm(g)\ast h_i\|_{C^{0,\theta}(M,E)}\\
&&\leq C\|H\|_{C^{k,\theta}_{con}(M,E)}.
\end{eqnarray*}

Therefore, by theorem \ref{first-iso-sch}, there exists a solution $\tilde{h}_k\in D^{2,\theta}_{\nabla f}(M,E)$ to
\begin{eqnarray*}
&&\Delta_{v-(2\alpha+k)\ln v}\tilde{h}_k-\alpha \tilde{h}_k=(\Rm(g)+\arrowvert\nabla\ln v\arrowvert^2)\ast h_k+H_k+\sum_{i=0}^{k-1}v^{\frac{k-i}{2}}\nabla^{k-i}\Rm(g)\ast h_i,\\
&&\| \tilde{h}_k\|_{C^{2,\theta}(M,E)}\leq C\|H\|_{C^{k,\theta}_{con}(M,E)},\quad\mbox{if $\theta\in(0,1)$}.
\end{eqnarray*}

On the other hand, the difference $T:=h_k-\tilde{h}_k$ satisfies
\begin{eqnarray*}
&&T\in\cap_{p\geq 1}W^{2,p}_{loc}(M,E)\quad;\quad \Delta_{v-(2\alpha+k)\ln v}T-\alpha T=0\quad;\quad T=\textit{O}(v^{1/2}).
\end{eqnarray*}
Corollary \ref{cor-uni-sol-pol-gro} ensures that $T\equiv 0$ since $1/2<\min\{\alpha,1\}$. In particular, if $\theta\in(0,1)$, $h_k\in D^{2,\theta}_{\nabla f}(M,E)$ and 
\begin{eqnarray*}
\|h_k\|_{C^{2,\theta}(M,E)}\leq C\|H\|_{C^{k,\theta}_{con}(M,E)}.
\end{eqnarray*}

\end{proof}

Here is a straightforward application of theorem \ref{iso-weighted-lap-II} : 
\begin{coro}\label{iso-sch-laplacian}
Let $(M^n,g,\nabla f)$ be an asymptotically conical expanding gradient Ricci soliton.

Then, for any $\alpha\in(1/2,+\infty)$ and any $\theta\in[0,1)$,
\begin{eqnarray*}
\Delta_f:D^{k+2,\theta}_{f^{\alpha},\nabla f}(M,S^2T^*M)\rightarrow C^{k,\theta}_{con,f^{\alpha}}(M,S^2T^*M),
\end{eqnarray*}
is an isomorphism.
\end{coro}

We state the analogue of corollary \ref{sec-iso-sch} with higher regularity.

\begin{coro}\label{sec-iso-sch-II}
Let $(M^n,g,\nabla f)$ be an asymptotically conical expanding gradient Ricci soliton with positive curvature operator.

Then, for any $\alpha\in(1/2,+\infty)$ and any $\theta\in[0,1)$, 
\begin{eqnarray*}
L_0:D^{k+2,\theta}_{f^{\alpha},\nabla f}(M,S^2T^*M)\rightarrow C^{k,\theta}_{con,f^{\alpha}}(M,S^2T^*M),
\end{eqnarray*}
is an isomorphism.
\end{coro}

\begin{rk}\label{rk-fin-nb-asy-cur-rat}
We stated theorem \ref{iso-weighted-lap-II} and corollaries \ref{iso-sch-laplacian} and \ref{sec-iso-sch-II} in the setting of asymptotically conical expanders. An inspection of the proof shows actually that one only needs the finiteness of the invariants $(\A^i_g(\Rm(g)))_{0\leq i\leq k+1}$, where $k$ is a positive integer, introduced in \cite{Der-Asy-Com-Egs}. Recall that :
$$\A^i_g(\Rm(g)):=\limsup_{+\infty}r_p^{2+i}\arrowvert\nabla^i\Rm(g)\arrowvert.$$
These invariants quantify the regularity of the convergence of the expander to its asymptotic cone as shown in \cite{Der-Asy-Com-Egs}.
\end{rk}


\subsection{Deformations of expanders : the Banach version}\label{sec-exp-str-I}
Let $(M^n,g_0,\nabla^0f_0)$ be a fixed expanding gradient Ricci soliton asymptotical to $(C(X),dr^2+r^2\gamma_0,r\partial_r/2)$. Denote by $M_s:=\{f=s\}$ the level set of the potential function. Let $\phi$ be the following diffeomorphism introduced in \cite{Der-Asy-Com-Egs} :
 \begin{eqnarray*}
\phi:(t_0,+\infty)\times M_{t_0^2/4}=:(t_0,+\infty)\times X &\rightarrow& M_{> t_0^2/4}\\
(t,x)&\rightarrow& \phi_{\frac{t^2}{4}-\frac{t_0^2}{4}}(x),
\end{eqnarray*}
for $t_0$ large enough, where $(\phi_t)_t$ is the Morse flow associated to the potential function $f$, i.e. $$\partial_t\phi_t=\frac{\nabla f}{\arrowvert\nabla f\arrowvert^2}.$$
Let $\theta\in(0,1)$.
Define $\cat{M}et^{k,\theta}(X)$ to be the convex cone of metrics on $X$ in $C^{k,\theta}(X,S^2T^*X)$ and let $B^{k,\theta}(\gamma_0)$ be the open unit ball centered at $\gamma_0$ in $C^{k,\theta}(X,S^2T^*X)$.
Next, we define a first approximation as in Lee \cite{Lee-Hyp} (and other related works on Einstein conformally compact manifolds previously mentioned)
\begin{eqnarray*}
 B^{k,\theta}(\gamma_0)\subset \cat{M}et^{k,\theta}(X)&\stackrel{T}{\longmapsto}&\cat{M}et^{k,\theta}(M)\cap\{ g_0+C^{k,\theta}_{con}(M,S^2T^*M)\}\\
 \gamma&{\longmapsto}& g_0+\psi(4f)\phi_{*}(\gamma-\gamma_0),
\end{eqnarray*}
where $\psi$ is a smooth function on $M$ such that $\psi\equiv 1$ on $M\setminus \{f\leq 2(t_0^2/4)\}$ and zero on $\{f\leq t_0^2/4\}$. $T$ is an affine map hence smooth and by its very definition, $(M^n,T(\gamma),p)$ is asymptotically conical to $(C(X),dr^2+r^2\gamma,o)$. Moreover, $T(\gamma_0)=g_0$.\\

Define the quadratic map 
\begin{eqnarray}
Q(g_2,g_1)&:=&-2\Ric(g_2)-g_2+\Li_{\nabla^{g_2}f_0}(g_2)+\Li_{V(g_2,g_1,f_0)}(g_2),\\
V(g_2,g_1,f_0)&:=&\div_{g_1,f_0}(g_2-g_1)-\frac{\nabla^{g_1}\tr_{g_1}(g_2-g_1)}{2}\\
&=&\div_{g_0}(g_2-g_1)+\nabla^{g_0}_{\nabla^{g_0}f_0}(g_2-g_1)-\frac{\nabla^{g_1}\tr_{g_1}(g_2-g_1)}{2}.
\end{eqnarray}

\begin{theo}\label{Loc-Dif-Ban}
Let $(M^n,g_0,\nabla^{g_0} f_0)$ be an asymptotically conical expanding gradient Ricci soliton with positive curvature operator. 

Then, for any $\alpha\in(1/2,1]$, and any nonnegative integer $k$, the map 
\begin{eqnarray*}
B^{k+2,\theta}(\gamma_0)\times (U(0)\subset D^{k+2,\theta}_{f^{\alpha},\nabla f}(M,S^2T^*M))&\stackrel{\Phi_{\alpha}}{\longmapsto}&B^{k+2,\theta}(\gamma_0)\times C^{k,\theta}_{con,f^{\alpha}}(M,S^2T^*M)\\
(\gamma,h)&{\longmapsto}&(\gamma,Q(T(\gamma)+h,T(\gamma)))
\end{eqnarray*}
is well-defined provided $U(0)$ is a neighborhood of $0$ sufficiently small in  $D^{k+2,\theta}_{f^{\alpha},\nabla f}(M,S^2T^*M)$. Moreover, $\Phi_{\alpha}$ is a local diffeomorphism at $(\gamma_0,0)$.\\

In particular, if $k\geq 1$, for any deformation $\gamma\in \cat{M}et^{k+2,\theta}(X)$ close enough to $\gamma_0$, there exists an expanding Ricci soliton with positive curvature operator whose asymptotic cone is $(C(X),dr^2+r^2\gamma,o)$.\\
\end{theo}
\begin{rk}
\begin{itemize}
\item As noticed in remark \ref{rk-fin-nb-asy-cur-rat}, theorem \ref{Loc-Dif-Ban} could have been stated with much less regularity at infinity : we only need the finiteness of a finite number of invariants $(\A_g^i(\Rm(g)))_i$ to perform the analysis.\\

\item The reason why we need $k\geq 1$ to ensure that the curvature of the implicit expander is positive comes from the radial curvature of an expanding gradient Ricci soliton given in term of the divergence of the curvature operator. \\

\item By choosing $\gamma_1$ close enough to $\gamma_0$ in $C^{k,\theta}(X,S^2T^*X)$ but not in $C^{k+1,\theta}(X,S^2T^*X)$, this proves the existence of Ricci expanders that are NOT smoothly asymptotic to their asymptotic cone. In particular, this shows that the finiteness of some of the asymptotic curvature ratios $\A^i_g(\Rm(g))$ does not imply the finiteness of all of them, i.e. there are no Shi's estimates for the rescaled curvature in general (unless the solution is ancient).\\

\item In the course of the proof of theorem \ref{Loc-Dif-Ban}, we prove more by establishing some tame estimates for the map $\Phi_{\alpha}$ that will be useful for section \ref{Sec-Fre-Def}.
\\
\end{itemize}
\end{rk}

Before giving the proof of theorem \ref{Loc-Dif-Ban}, we actually need two technical lemmata :

\begin{lemma}\label{lin-qua-egs}
Let $(M,g)$ be a Riemannian manifold, let $f$ be a smooth function and let $h$ be a symmetric $2$-tensor such that $g+h$ is metric. Then, for $s\in[0,1]$,
\begin{eqnarray}
D_{g+sh}(-2\Ric)(h)&=&\Delta_{g+sh}h+2\Rm(g+sh)\ast h-\Sym(\Ric(g+sh)\otimes h)\\
&&+\nabla^{2,g+sh}\tr h-\Li_{\div_{g+sh}h}(g+sh).\\
D^2_{g+sh}(-2\Ric)(h,h)&=&(g+sh)^{-2}\ast\nabla^{g+sh,2}h\ast h+(g+sh)^{-2}\ast(\nabla^{g+sh}h)^{\ast2}.\label{Der-Sec-Ric}\\
Q(g+h,g)&=&Q(g,g)+\Delta_{g,f}h+2\Rm(g)\ast h+\frac{1}{2}\Sym(Q(g,g)\otimes h)\\
&&+\Li_{\div_{g}h-\nabla^g\tr_gh/2}(h)+\Li_{h^{\ast2}\ast\nabla^gf}(g+h),\\
&&+\int_0^1(1-s)D^2_{g+sh}(-2\Ric)(h,h)ds,
\end{eqnarray}
where $(2\Rm(g)\ast h)_{ij}:=2\Rm(g)_{iklj}h_{kl}$.
Moreover, if $(M^n,g,\nabla f)$ is an expanding gradient Ricci soliton,
\begin{eqnarray*}
Q(g+h,g+h)&=&\Delta_{g,f}h+2\Rm(g)\ast h-\Li_{\div_{g}h-\nabla^g\tr_gh/2}(g)\\
&&+\int_0^1(1-s)D^2_{g+sh}(-2\Ric)(h,h)ds+\Li_{\nabla^{g+h}f-\nabla^gf}(g+h),
\end{eqnarray*}

\end{lemma}

\begin{proof}
The first and second variation of the Ricci curvature are standard and can be found in [Chap.$2$, \cite{Ben}].

Concerning the variation of $Q$ : 
\begin{eqnarray*}
Q(g+h,g)&=&-2\Ric(g+h)-(g+h)+\Li_{\nabla^{g+h}f}(g+h)+\Li_{\div_{g,f}h-\nabla^g\tr_gh/2}(g+h)\\
&=&-2\Ric(g)-g+\Li_{\nabla^gf}(g)+(-2D_{g}\Ric(h))-h\\
&&+\Li_{\nabla^gf}(h)+\Li_{\div_{g}h-\nabla^g\tr_gh/2}(g)\\
&&+\int_0^1(1-s)(-2D^2_{g+sh})\Ric(h,h)ds+\Li_{\div_{g}h-\nabla^g\tr_gh/2}(h)\\
&&+\Li_{h^{\ast2}\ast\nabla^gf}(g+h),
\end{eqnarray*}
since $\nabla^{g+h}f-\nabla^gf+h(\nabla^gf)=h^{\ast2}\ast\nabla^gf$. 

Now, using the first variation of the Ricci curvature,
\begin{eqnarray*}
Q(g+h,g)&=&Q(g,g)+\Delta_{g,f}h+2\Rm(g)\ast h-\Sym\left(\left(\Ric(g)+\frac{g}{2}\right)\otimes h\right)\\
&&+\frac{1}{2}\Sym(\Li_{\nabla^g f}(g)\otimes h)+\int_0^1(1-s)(-2D^2_g)\Ric(h,h)ds\\
&&+\Li_{\div_{g}h-\nabla^g\tr_gh/2}(h)+\Li_{h^{\ast2}\ast\nabla^gf}(g+h).
\end{eqnarray*}
since $\Li_{\nabla^gf}(h)=\nabla^g_{\nabla^gf}h+\frac{1}{2}\Sym(\Li_{\nabla^g f}(g)\otimes h)$. That is, by using the definition of $Q(g,g)$,
\begin{eqnarray*}
Q(g+h,g)&=&Q(g,g)+\Delta_{g,f}h+2\Rm(g)\ast h+\frac{1}{2}\Sym(Q(g,g)\otimes h)\\
&&+\int_0^1(1-s)D^2_{g+sh}(-2\Ric)(h,h)ds+\Li_{\div_{g}h-\nabla^g\tr_gh/2}(h)\\
&&+\Li_{h^{\ast2}\ast\nabla^gf}(g+h).
\end{eqnarray*}

If $(M^n,g,\nabla^g f)$ is an expanding gradient Ricci soliton, 
\begin{eqnarray*}
Q(g+h,g+h)&=&-2\Ric(g+h)-(g+h)+\Li_{\nabla^{g+h}f}(g+h)\\
&=&-2\Ric(g)-g+\Li_{\nabla^gf}(g)\\
&&+\Delta_gh+2\Rm(g)\ast h-\Sym(\Ric(g)\otimes h)-h+\Li_{\nabla^gf}(h)\\
&&-\Li_{\div_{g}h-\nabla^g\tr_gh/2}(g)+\int_0^1(1-s)(-2D^2_{g+sh}\Ric)(h,h)ds\\
&&+\Li_{\nabla^{g+h}f-\nabla^gf}(g+h)\\
&=&\Delta_{g,f}h+2\Rm(g)\ast h-\Li_{\div_{g}h-\nabla^g\tr_gh/2}(g)\\
&&+\int_0^1(1-s)(-2D^2_{g+sh}\Ric)(h,h)ds+\Li_{\nabla^{g+h}f-\nabla^gf}(g+h).
\end{eqnarray*}

\end{proof}

\begin{lemma}\label{Fir-App-Egs}
With the above notations, 
\begin{eqnarray*}
Q(T(\gamma),T(\gamma))&=&\nabla^{g_0,2}h(\gamma)+(\Rm(g_0)+f_0^{-1})\ast h(\gamma)\\
&&+\int_0^1(1-s)(-2D^2_{g_0+sh(\gamma)}\Ric)(h(\gamma),h(\gamma))ds,
\end{eqnarray*}
outside a compact set, where $h(\gamma):=T(\gamma)-T(\gamma_0)$.
\end{lemma}

\begin{proof}
Using lemma \ref{lin-qua-egs}, and defining $h(\gamma):=T(\gamma)-T(\gamma_0)$,
\begin{eqnarray*}
Q(T(\gamma),T(\gamma))&=&Q(g_0+h(\gamma),g_0+h(\gamma))\\
&=&\Delta_{g_0,f_0}h(\gamma)+2\Rm(g_0)\ast h(\gamma)  -\Li_{\div_{g_0}h(\gamma)-\nabla^{g_0}\tr_{g_0}h(\gamma)/2}(g_0)\\
&&+\Li_{\nabla^{T(\gamma)}f_0-\nabla^{T(\gamma_0)}f_0}(g_0+h(\gamma))+\int_0^1(1-s)(-2D^2_{g_0+sh(\gamma)}\Ric)(h(\gamma),h(\gamma))ds\\
&=&\nabla^{g_0,2}h(\gamma)+\nabla^{g_0}_{\nabla^{g_0}f_0}h(\gamma)+\Rm(g_0)\ast h(\gamma)\\
&&+\Li_{\nabla^{T(\gamma)}f_0-\nabla^{T(\gamma_0)}f_0}(g_0+h(\gamma))+\int_0^1(1-s)(-2D^2_{g_0+sh(\gamma)}\Ric)(h(\gamma),h(\gamma))ds.
\end{eqnarray*}
Now, using the soliton identities, outside a compact set,
\begin{eqnarray*}
\nabla^{g_0}_{\nabla^{g_0}f_0}h(\gamma)&=&\Li_{\nabla^{g_0}f_0}(h(\gamma))-\frac{1}{2}\Sym(\Li_{\nabla^{g_0}f_0}(g_0)\otimes h(\gamma))\\
&=&4\arrowvert\nabla^{g_0}f_0\arrowvert^2\phi_{*}(\gamma-\gamma_0)+(4f_0)\Li_{\nabla^{g_0}f_0}(\phi_{*}(\gamma-\gamma_0))\\
&&-h(\gamma)-\Sym(\Ric(g_0)\otimes h(\gamma))\\
&=&\frac{\mu(g_0)-\R_{g_0}}{f_0}h(\gamma)-\Sym(\Ric(g_0)\otimes h(\gamma))+(4f_0)\Li_{\nabla^{g_0}f_0}(\phi_{*}(\gamma-\gamma_0)).
\end{eqnarray*}
As the diffeomorphism $\phi$ is preserving the expanding structure, that is $\phi^*f(t,x)=t^2/4$ for $(t,x)\in C(X)$, and as the deformation $h(\gamma)$ is transversal to the radial direction,
\begin{eqnarray*}
 \phi^*(h(\gamma)(\nabla^{g_0}f_0))(t,x)&=&t^2(\gamma-\gamma_0)(\phi^*(\nabla^{g_0}f_0))(t,x)=0,\\
 \nabla^{T(\gamma)}f_0-\nabla^{T(\gamma_0)}f_0&=&\nabla^{g_0+h(\gamma)}f_0-\nabla^{g_0}f_0\\
 &=&0,
\end{eqnarray*}
outside a compact set.

Therefore, outside a compact set,
\begin{eqnarray*}
Q(T(\gamma),T(\gamma))&=&\nabla^{2,g_0}h(\gamma)+\Rm(g_0)\ast h(\gamma)\\
&&+\int_0^1(1-s)(-2D^2_{g_0+sh(\gamma)}\Ric)(h(\gamma),h(\gamma))ds.
\end{eqnarray*}

\end{proof}

We are now in a position to prove theorem \ref{Loc-Dif-Ban}.

\begin{proof}[Proof of theorem \ref{Loc-Dif-Ban}]
$\Phi_{\alpha}$ is actually well-defined. We will only consider the case $\alpha=1$ to simplify notations. Indeed, by lemmata \ref{lin-qua-egs} and \ref{Fir-App-Egs},
\begin{eqnarray}
Q(T(\gamma)+h,T(\gamma))&=&\nabla^{g_0,2}h(\gamma)+(\Rm(g_0)+v_0^{-1})\ast h(\gamma)\label{lig1-loc}\\
&&+\int_0^1(1-s)D^2_{g_0+sh(\gamma)}(-2\Ric)(h(\gamma),h(\gamma))ds\label{lig2-loc}\\
&&+\Delta_{T(\gamma),f_0}h+2\Rm(T(\gamma))\ast h+\frac{1}{2}\Sym(Q(T(\gamma),T(\gamma))\otimes h)\label{lig3-loc}\\
&&+\int_0^1(1-s)D^2_{T(\gamma)+sh}(-2\Ric)(h,h)ds\label{lig4-loc}\\
&&+\nabla^{T(\gamma)}h\ast\nabla^{T(\gamma)}h+\nabla^{T(\gamma),2}h\ast h+\Li_{h^{*2}\ast\nabla^{T(\gamma)}f_0}(T(\gamma)+h).\label{lig5-loc}
\end{eqnarray}
Note first that by its very definition, $$h(\gamma)\in (C^{k+2}_{con}(M,S^2T^*M),C^{k+2+1}_{con}(M,S^2T^*M))_{\theta,\infty}\subset C^{k+2,\theta}_{con}(M,S^2T^*M),$$ by remark \ref{rk-non-int-spa}. By inspecting each term, one can actually check that $\Phi_{\alpha}$ is well-defined. We anticipate on the next section on Fréchet deformations by proving some "tame" estimates on the map $\Phi_{\alpha}$.

The proof is actually inspired by \cite{Ham-Nas-Mos}. One of the main tool are the interpolation inequalities that hold here when the parameter $\theta$ is $0$. As noticed in remark \ref{rk-non-int-spa}, corresponding interpolation inequalities for an arbitrary $\theta\in(0,1)$ are not straightforward. To circumvent this issue, we will use the simple fact that, for a given nonnegative integer $k$ and some $\theta\in(0,1)$, the norm $\|\cdot\|_{k,\theta}$ (on the spaces considered above) is dominated by the norm $\|\cdot\|_{k+1}$. The drawback is that we establish less precise estimates : this will not affect the use of the Nash-Moser theorem.

We temporarily simplify the notations : denote by $(\|\cdot\|_{k+2,\theta}^{D_f})_{k\geq 0}$ (respectively $(\|\cdot\|_{k,\theta}^{C_f})_{k\geq 0}$, respectively $(\|\cdot\|_{k,\theta}^{C})_{k\geq 0}$)   the collection of norms induced on $(D^{k+2,\theta}_{f^{\alpha},\nabla f}(M,S^2T^*M))_{k\geq 0}$ (respectively on $(C^{k,\theta}_{con,f^{\alpha}}(M,S^2T^*M))_{k\geq 0}$, respectively on $(C^{k,\theta}_{con}(M,S^2T^*M))_{k\geq 0}$).

\begin{itemize}
\item Concerning the first line (\ref{lig1-loc}), by invoking theorem \ref{iso-weighted-lap-II} : 
\begin{eqnarray*}
\|v_0^{k/2}\nabla^{g_0,k}(\nabla^{g_0,2}h(\gamma))\|^{C_f}_{0,\theta}&\lesssim& \|h(\gamma)\|^C_{k+2,\theta},\\
\|v_0^{k/2}\nabla^{g_0,k}((\Rm(g_0)+v_0^{-1})\ast h(\gamma))\|^{C_f}_{k,\theta}&\lesssim&\sum_{i=0}^k\|v_0\Rm(g_0)\|^C_{k-i,\theta} \|h(\gamma)\|^{C}_{i,\theta}\\
&\lesssim& \|h(\gamma)\|^{C}_{k+2,\theta},
\end{eqnarray*}
where $\lesssim$ means up to a multiplicative positive constant independent of the variables.

The main difficulty in estimating the remaining terms is the presence of connections with respect to metrics different from $g_0$, indeed, the function spaces we consider are defined in terms of the metric $g_0$.

We first remark that for two metrics $g_1$ and $g_2$ and for some tensor $T$, one has 
\begin{eqnarray}\label{Var-For-2-Met-Con}
\nabla^{g_2}T=\nabla^{g_1}T+g_2^{-1}\ast \nabla^{g_1}(g_2-g_1)\ast T.\label{fir-var-cov-der}
\end{eqnarray}

\item Concerning the last line (\ref{lig5-loc}) :

\begin{eqnarray*}
\nabla^{T(\gamma)}h&=&\nabla^{g_0}h+T(\gamma)^{-1}\ast \nabla^{g_0}h(\gamma)\ast h\\
&&\\
\nabla^{g_0}T(\gamma)^{-1}&=&T(\gamma)^{-2}\ast\nabla^{g_0}h(\gamma),\\
&&\\
\nabla^{T(\gamma),2}h&=&\nabla^{T(\gamma)}(\nabla^{g_0}h+T(\gamma)^{-1}\ast \nabla^{g_0}h(\gamma)\ast h)\\
&=&\nabla^{g_0}(\nabla^{g_0}h+T(\gamma)^{-1}\ast \nabla^{g_0}h(\gamma)\ast h)\\
&&+T(\gamma)^{-1}\ast \nabla^{g_0}h(\gamma)\ast (\nabla^{g_0}h+T(\gamma)^{-1}\ast \nabla^{g_0}h(\gamma)\ast h)\\
&=&\nabla^{g_0,2}h+T(\gamma)^{-1}\ast \nabla^{g_0,2}h(\gamma)\ast h\\
&&+T(\gamma)^{-1}\ast \nabla^{g_0}h(\gamma)\ast \nabla^{g_0}h+T(\gamma)^{-2}\ast \nabla^{g_0}h(\gamma)^{*2}\ast h.
\end{eqnarray*}

We need to estimate the (rescaled) derivatives of the inverse of a small perturbation of $g_0$ : 

\begin{claim}\label{cla-inv-ham}
There exists some large but fixed nonnegative integer $k_0$ such that for $k\geq k_0$,
\begin{eqnarray*}
\|(g_0+h_0)^{-1}\|_{k,\theta}^C\lesssim \|h_0\|^C_{k,\theta}+1,
\end{eqnarray*}
for $h_0\in U(0)\subset C^{k_0,\theta}_{con}(M,S^2T^*M)$.
\end{claim}

\begin{proof}[Proof of claim \ref{cla-inv-ham}]
For $k=1$, one has
\begin{eqnarray*}
\nabla^{g_0}(g_0+h_0)^{-1}=(g_0+h_0)^{-2}\ast\nabla^{g_0}h_0,
\end{eqnarray*}
which implies the result if $\|h_0\|^C_{1,\theta}$ is small.

\begin{eqnarray*}
0&=&\nabla^{g_0,k}\left[(g_0+h_0)\circ(g_0+h_0)^{-1}\right]\\
&=&(g_0+h_0)\nabla^{g_0,k}(g_0+h_0)^{-1}+\sum_{i=0}^{k-1}\nabla^{g_0,k-i}h_0\ast \nabla^{g_0,i}(g_0+h_0)^{-1},
\end{eqnarray*}
therefore, by induction on $k$,
\begin{eqnarray*}
\| v_0^{k/2}\nabla^{g_0,k}(g_0+h_0)^{-1}\|^C_{0,\theta}&\lesssim& \sum_{i=0}^{k-1}\|h_0\|^C_{k-i,\theta}\|(g_0+h_0)^{-1}\|_{i,\theta}^C\\
&\lesssim&\sum_{i=0}^{k-1}\|h_0\|^C_{k-i,\theta}(\|h_0\|_{i,\theta}^C+1)\\
&\lesssim& \|h_0\|^C_{k,\theta}+ \sum_{i=0}^{k-1}\|h_0\|^C_{k-i,\theta}\|h_0\|_{i,\theta}^C\\
&\lesssim&  \|h_0\|^C_{k-1,\theta}+\sum_{i=1}^{k-2}\|h_0\|^C_{k-i+1}\|h_0\|_{i+1}^C,
\end{eqnarray*}
if $k\geq 3$ (the case $k=2$ can be adapted) and if $\|h_0\|^C_{1,\theta}$ is bounded. By using interpolation inequalities :
\begin{eqnarray*}
\| v_0^{k/2}\nabla^{g_0,k}(g_0+h_0)^{-1}\|^C_{0,\theta}&\lesssim& \|h_0\|^C_{k,\theta}+\|h_0\|^C_{k}\|h_0\|_{3}^C\\
&\lesssim&\|h_0\|^C_{k,\theta},
\end{eqnarray*}
provided $\|h_0\|^C_{3,\theta}$ is bounded. Hence the result.

\end{proof}
 
 Now, combining these estimates with the results of theorem \ref{iso-weighted-lap-II}, one can bound the following terms as follows :
 \begin{eqnarray*}
\|v_0^{k/2}\nabla^{g_0,k}(\nabla^{g_0,2}h\ast h)\|_{0,\theta}^{C_f}&\lesssim&\sum_{i=0}^k\|\nabla^{g_0,2}h\|^{C_f}_{i,\theta}\|h\|^{C_f}_{k-i,\theta}\\
&\lesssim&\sum_{i=0}^k\|h\|^{D_f}_{i+2,\theta}\|h\|^{C_f}_{k-i,\theta}\\
&\lesssim&\sum_{i=0}^k\|h\|^{D_f}_{i+2,\theta}\|h\|^{D_f}_{k-i+2,\theta}\\
&\lesssim&\|h\|^{D_f}_{k+2,\theta}+\sum_{i=1}^{k-1}\|h\|^{D_f}_{i+3}\|h\|^{D_f}_{k-i+3}\\
&\lesssim&\|h\|^{D_f}_{k+2,\theta}+\|h\|^{D_f}_{k+2}\|h\|^{D_f}_{4}\\
&\lesssim&\|h\|^{D_f}_{k+2,\theta},
\end{eqnarray*}
provided $\|h\|^{D_f}_{4,\theta}$ is bounded. Similarly, one has : 
\begin{eqnarray*}
\|v_0^{k/2}\nabla^{g_0,k}(\nabla^{g_0}h\ast \nabla^{g_0}h)\|_{0,\theta}^{C_f}&\lesssim&\sum_{i=0}^k\|\nabla^{g_0}h\|^{C_f}_{i,\theta}\|\nabla^{g_0}h\|^{C_f}_{k-i,\theta}\\
&\lesssim&\sum_{i=0}^k\|h\|^{D_f}_{i+2,\theta}\|h\|^{D_f}_{k-i+2,\theta}\\
&\lesssim&\|h\|^{D_f}_{k+2,\theta},
\end{eqnarray*}
provided $\|h\|^{D_f}_{4,\theta}$ is bounded.
 
 Using claim \ref{cla-inv-ham} together with interpolation inequalities, one gets in a similar way :
 \begin{eqnarray*}
\|\nabla^{T(\gamma)}h\ast \nabla^{g_0}h+\nabla^{T(\gamma),2}h\ast h\|^{C_f}_{k,\theta}\lesssim \|h\|^{D_f}_{k+2,\theta}+\|h(\gamma)\|^C_{k+2,\theta}(\|h\|_{k_0+2,\theta}^{D_f})^2,
\end{eqnarray*}
 for some fixed nonnegative integer $k_0$.
 
 The Lie derivative can be rewritten in the following way :
 \begin{eqnarray*}
\Li_{h^{*2}\ast\nabla^{T(\gamma)}f_0}(T(\gamma)+h)&=&\nabla^{T(\gamma)+h}(h^{*2}\ast\nabla^{T(\gamma)}f_0)\\
&=&\nabla^{T(\gamma)+h}(h^{*2}\ast\nabla^{g_0}f_0)\\
&=&\nabla^{g_0}(h^{*2}\ast\nabla^{g_0}f_0)\\
&&+(T(\gamma)+h)^{-1}\ast\nabla^{g_0}(T(\gamma)+h)\ast h^{*2}\ast\nabla^{g_0}f_0,
\end{eqnarray*}
since $T(\gamma)+h$ is a metric and $\nabla^{T(\gamma)}f_0=\nabla^{g_0}f_0$ outside a compact set. Again, one can estimate this term by 
\begin{eqnarray*}
\|\Li_{h^{*2}\ast\nabla^{T(\gamma)}f_0}(T(\gamma)+h)\|^{C_f}_{k,\theta}\lesssim \|h\|^{D_f}_{k+2,\theta}+\|h(\gamma)\|^C_{k+2,\theta},
\end{eqnarray*}
provided $(\gamma,h)\in U(\gamma_0,0)$ is in a neighborhood of $(\gamma_0,0)$ sufficiently small.\\

\item Concerning the integral terms in lines (\ref{lig2-loc}) and (\ref{lig4-loc}), one applies the previous methods, that is, the use of  interpolation inequalities together with formula (\ref{fir-var-cov-der}), to formula (\ref{Der-Sec-Ric}) to get : 
\begin{eqnarray*}
\left\|\int_0^1(1-s)D^2_{g_0+sh(\gamma)}(-2\Ric)(h(\gamma),h(\gamma))ds\right\|^{C_f}_{k,\theta}\lesssim \|h(\gamma)\|^C_{k+2,\theta},
\end{eqnarray*}
provided $\|h(\gamma)\|_{k_0+2,\theta}^C$ is bounded for some fixed nonnegative integer $k_0$. Similarly, one gets : 
\begin{eqnarray*}
\left\|\int_0^1(1-s)D^2_{T(\gamma)+sh}(-2\Ric)(h,h)ds\right\|^{C_f}_{k,\theta}\lesssim \|h(\gamma)\|^C_{k+2,\theta}+\|h\|^{D_f}_{k+2,\theta},
\end{eqnarray*}
provided $\|h(\gamma)\|_{k_0+2,\theta}^C+\|h\|^{D_f}_{k_0+2,\theta}$ is bounded for some fixed nonnegative integer $k_0$.\\

\item Concerning the third line \ref{lig3-loc}, the only term that could cause some  trouble is $\nabla^{T(\gamma)}_{\nabla^{T(\gamma)}f_0}h$. Nonetheless, outside a compact set : 
\begin{eqnarray*}
\nabla^{T(\gamma)}_{\nabla^{T(\gamma)}f_0}h&=&(\nabla^{T(\gamma)}-\nabla^{g_0})h\ast \nabla^{g_0}f_0+\nabla^{g_0}_{ \nabla^{g_0}f_0}h\\
&=&T(\gamma)^{-1}\ast\nabla^{g_0}h(\gamma)\ast h\ast  \nabla^{g_0}f_0+\nabla^{g_0}_{\nabla^{g_0}f_0}h.
\end{eqnarray*}
Now, by definition of the spaces $D^{k+2,\theta}_{f,\nabla f}(M,S^2T^*M)$, $\nabla^{g_0}_{\nabla^{g_0}f_0}h\in D^{k+2,\theta}_{f,\nabla f}(M,S^2T^*M)$ and since $h(\gamma)\in C^{k+2,\theta}_{con}(M,S^2T^*M)$, $\nabla^{g_0} h(\gamma)\ast \nabla^{g_0}f_0\in C^{k,\theta}_{con}(M,S^2T^*M)$.
 \end{itemize}
 
 To sum it up, we obtained the following (tame) estimate as expected :
 \begin{eqnarray}
 \|Q(T(\gamma)+h,T(\gamma))\|_{k,\theta}^{C_f}\lesssim \|(h(\gamma),h)\|^{C\times D_f}_{k+2,\theta},\label{tame-est-phi}
 \end{eqnarray}
  if $\|(h(\gamma),h)\|^{C\times D_f}_{k_0+2,\theta}$ is bounded with some fixed nonnegative integer $k_0$.\\

By the Banach version of the inverse function theorem, it suffices to show that the differential of $\Phi_{\alpha}$ at $(\gamma_0,0)$ is invertible.

\begin{eqnarray*}
D_{(\gamma_0,0)}\Phi_{\alpha}(\bar{\gamma},\bar{h})&=&(\bar{\gamma},D^1_{(g_0,g_0)}Q(D_{\gamma_0}T(\bar{\gamma})+\bar{h})+D^2_{(g_0,g_0)}Q(D_{\gamma_0}T(\bar{\gamma})))\\
&=&(\bar{\gamma},D^1_{(g_0,g_0)}Q(\bar{h})+K_{g_0}(\bar{\gamma})),\\
\end{eqnarray*}
where $K_{g_0}(\bar{\gamma}):=D^1_{(g_0,g_0)}Q(D_{\gamma_0}T(\bar{\gamma}))+D^2_{(g_0,g_0)}Q(D_{\gamma_0}T(\bar{\gamma})).$
By lemma \ref{lin-qua-egs}, one easily identifies the operator $D^1_{(g_0,g_0)}Q$ with the weighted Lichnerowicz operator $\Delta_{g_0,f_0}+2\Rm(g_0)\ast$.
Now, by theorem \ref{iso-weighted-lap-II},  $D^1_{(g_0,g_0)}Q$ is an isomorphism of Banach spaces which proves that $D_{(\gamma_0,0)}\Phi_{\alpha}$ is an isomorphism.

Finally, assume that $(M^n,g_0,\nabla^{g_0}f_0)$ has positive curvature and let $\gamma\in \cat{M}et^{k+2,\theta}(X)$ close enough to $\gamma_0$ so that there exists a (unique) deformation $h\in D^{k+2,\theta}_{f,\nabla f}(M,S^2T^*M)$ such that $Q(T(\gamma)+h,T(\gamma))=0$, i.e. such that $$-2\Ric(g)-g+\Li_{\nabla^gf_0}(g)+\Li_{V(g,T(\gamma),f_0)}(g)=0,$$ where $g:=T(\gamma)+h$.
If $V:=\nabla^gf_0+V(g,T(\gamma),f_0)$ then $$-2\Ric(g)-g+\Li_{V}(g)=0.$$
It suffices to prove that the radial curvatures of $g$ are positive since :
\begin{eqnarray*}
v\arrowvert \Rm(g)|_{sph}-\Rm(g_0)|_{sph}\arrowvert\leq C(n)\| g-g_0\|_{C^2_{con}(M,S^2T^*M)}.
\end{eqnarray*}

Now,
\begin{eqnarray*}
2\div\Rm(g)_{ijk}&=&\nabla^g_i2\Ric(g)_{jk}-\nabla^g_j2\Ric(g)_{ik}=\nabla^g_i\Li_{V}(g)_{jk}-\nabla^g_j\Li_{V}(g)_{ik}\\
&=&\nabla^g_i(\nabla^g_j{V}_k+\nabla^g_k{V}_j)-\nabla^g_j(\nabla^g_i{V}_k+\nabla^g_k{V}_i)\\
&=&\Rm(g)_{ij{V}k}+\nabla^g_{i}\nabla^g_k{V}_j-\nabla^g_j\nabla^g_k{V}_i\\
&=&2\Rm(g)_{ij{V}k}+\nabla^g_i(\nabla^g_k{V}_j-\nabla^g_j{V}_k)-\nabla^g_j(\nabla^g_k{V}_i-\nabla^g_i{V}_k),
\end{eqnarray*}
that is, using the definition of ${V}$, $$\div\Rm(g)(\cdot,\cdot,\cdot)=\Rm(g)(\cdot,\cdot,{V},\cdot)+\nabla^{g,2}V(g,T(\gamma),f_0).$$
On the one hand, as $(M^n,g_0,\nabla^{g_0}f_0)$ is an asymptotically conical expanding gradient Ricci soliton with positive curvature, by lemma \ref{id-EGS} and [Section $3$, \cite{Der-Asy-Com-Egs}],
\begin{eqnarray*}
&&\div\Rm(g_0)(\cdot,\cdot,\cdot)=\Rm(g_0)(\cdot,\cdot,\nabla^{g_0}f_0,\cdot),\\
&&\inf_{x\in M^n}v_0(x)\min_{u\perp \nabla^{g_0}f_0;\arrowvert u\arrowvert_{g_0}=1}\Rm(g_0)(u,\nabla^{g_0}f_0,\nabla^{g_0}f_0,u)=c>0.
\end{eqnarray*}
Now, as $g-g_0\in C^{2+k,\theta}_{con}(M,S^2T^*M)$, with $k\geq 1$, we have 
\begin{eqnarray*}
v_0^{3/2}\arrowvert\div_g\Rm(g)-\div_{g_0}\Rm(g_0)\arrowvert_{g_0}\leq C(n)\|g-g_0\|_{C_{con}^{3}(M,S^2T^*M)},
\end{eqnarray*}
that is, if $(\gamma,h)$ is sufficiently close to $(\gamma_0,0)$ in $C^{k+2,\theta}(X,S^2T^*X)\times D^{k+2,\theta}_{f,\nabla f}(M,S^2T^*M)$ for $k\geq 1$, $$\inf_{x\in M}v_0(x)\min_{u\perp V;\arrowvert u\arrowvert_{g}=1}\left(\Rm(g)(u,V,V,u)+\nabla^{2,g}V(g,T(\gamma),f_0)\ast V\right)\geq\frac{c}{2}>0.$$
On the other hand, 
\begin{eqnarray*}
v_0^{3/2}\arrowvert\nabla^{g,2}V(g,T(\gamma),f_0)\arrowvert&=&v_0^{3/2}\arrowvert\nabla^{g,2}(\nabla^{T(\gamma)}h+h(\nabla^{T(\gamma)}f_0))\arrowvert\\
&\leq&C(n,g_0)\| h\|_{C^3_{con,f}(M,S^2T^*M)}.
\end{eqnarray*}
Therefore, if $(\gamma,h)$ is sufficiently close to $(\gamma_0,0)$ in $C^{k+2,\theta}(X,S^2T^*X)\times D^{k+2,\theta}_{f,\nabla f}(M,S^2T^*M)$ for $k\geq 1$, then the radial curvatures of $g$ are positive and decay at most like $v_0^{-2}$, i.e. $r_p^{-4}$ in terms of the distance function to a fixed point $p$, at infinity.

\end{proof}

\section{Fréchet deformations}\label{Sec-Fre-Def}

\subsection{Functions spaces}\label{Sec-Fct-Spa-Fre}

We recall several definitions from \cite{Ham-Nas-Mos} :
\begin{itemize}
\item A Fréchet space $F$ is \textit{graded} if it admits a collection of seminorms $\{\|\cdot\|_k\}_{k\geq 0}$ defining the topology of $F$ and such that for any $k\geq 0$, $\|\cdot\|_k\leq\|\cdot\|_{k+1}$.\\
\item Let $L:F\rightarrow G$ be a linear map between graded Fréchet spaces. $L$ is \textit{tame} if it satisfies a so called \textit{tame estimate} of degree $r$ and base $b$, i.e.
\begin{eqnarray*}
\|Lh\|_{k}\leq C(k)\|h\|_{k+r},\quad k\geq b,
\end{eqnarray*}
for any $h\in F$.\\
\item Let $P:U\subset F\rightarrow G$ be a continuous map of an open subset $U$ of $F$ into $G$ where $F$ and $G$ are graded Fréchet spaces. $P$ is \textit{tame} of degree $r$ and base $b$ if $\|P(h)\|_{n}\leq C(1+\|h\|_{n+r})$, for all $h\in U$ and all $n\geq b$ for some positive constant $C$ eventually depending on $n$ and $U$.\\
\item Let $B$ be a Banach manifold. One denotes by $\Sigma(B)$ the space of exponentially decreasing sequences in $B$, i.e. the graded Fréchet space : 
$$\Sigma(B):=\{(h_k)_{k}, h_k\in B\quad|\quad \|(h_k)_{k}\|_n:=\sum_{k\geq 0}e^{nk}\|h_k\|_B<+\infty, \forall n\geq 0\}.$$
\item A graded Fréchet space $F$ is \textit{tame} if there exists a Banach space $B$ such that there exists two tame linear maps 
$$F\stackrel{L}{\longmapsto}\Sigma(B)\stackrel{M}{\longmapsto} F,$$ such that $M\circ L=Id_F$.
\end{itemize}

We consider the following function spaces :\\
\begin{itemize}
\item If $X$ is a smooth compact manifold, $$C^{\infty}(X,E):=\cap_{k\geq 0} C^{k,\theta}(X,E),$$
 endowed with the collection of norms $\left(\|\cdot\|_{C^{k,\theta}(M,E)}\right)_{k\geq 0}$, for some fixed $\theta\in[0,1)$.
 \item Let $(M,g,\nabla f)$ be a complete expanding gradient Ricci soliton. For some nonnegative integer $k$ and some $\theta\in[0,1)$, we define :
  $$C^{k,\theta}_{0,con}(M,E):=\{h\in C^{k,\theta}_{con}(M,E)\quad|\quad \lim_{+\infty}v^{i/2}\nabla^ih=0,\quad\forall i=0,...,k\}.$$
  
  \item Let $V$ be a smooth vector field on $M$ and $\theta\in[0,1)$. 
$$D^{k+2,\theta}_{0,V}(M,E):=\{ h\in D^{k+2,\theta}_{V}(M,E)\quad|\quad h\in C^{k,\theta}_{0,con}(M,E)\quad|\quad \Delta_{V}h\in C^{k,\theta}_{0,con}(M,E)\}.$$
  \item Let $w:M\rightarrow\mathbb{R_+^*}$ be a smooth function.
  
  $$C_{0,con,w}^{k,\theta}(M,E):=w^{-1}C_{0,con}^{k,\theta}(M,E),$$ equipped with $\|h\|_{C_{0,con,w}^{k,\theta}(M,E)}:=\|wh\|_{C_{con}^{k,\theta}(M,E)}.$

 $$D^{k+2,\theta}_{0,w,V}(M,E):=w^{-1}\cdot D^{k+2,\theta}_{0,V}(M,E),$$ equipped with the norm $\|h\|_{D^{k+2,\theta}_{0,w,V}(M,E)}:=\|wh\|_{D^{k+2,\theta}_{0,V}(M,E)}.$

$$C^{\infty}_{0,con,w}(M,E):=\cap_{k\geq 0} C^{k,\theta}_{0,con,w}(M,E),$$
  endowed with the collection of norms $\left(\|\cdot\|_{C^{k,\theta}_{con,w}(M,E)}\right)_{k\geq 0}$.

$$D^{\infty}_{0,w,V}(M,E):=\cap_{k\geq 0}D^{k+2,\theta}_{0,w,V}(M,E),$$  endowed with the collection of norms $\left(\|\cdot\|_{D^{k+2,\theta}_{0,w,V}(M,E)}\right)_{k\geq 0}$.
\end{itemize}


\subsection{Solutions in $C^k_0(M,E)$, $k\geq 0$}\label{Sec-Sol-C^k_0}
\subsubsection{The case $k=0$}
We start with an ad-hoc version of the maximum principle for weighted spaces : 
\begin{prop}\label{ppe-max-int-wei-hea-equ}
Let $(M,g)$ be a complete Riemannian manifold. Let $w: M\rightarrow\mathbb{R}_+$ be a smooth non negative function on $M$.
Let $u:M\times[0,T]$ be a subsolution to the weighted heat equation with weight $w$, i.e.
\begin{eqnarray*}
\partial_tu-\Delta_wu\leq0,\quad\mbox{on $M\times(0,T]$},
\end{eqnarray*}
such that $\sup_{M}u(0,\cdot)\leq 0$. Assume that there is some positive $\alpha$ such that
\begin{eqnarray*}
\int_0^T\int_Me^{-\alpha d_g^2(x,p)}\max\{u(t,x),0\}^2e^{w(x)}d\mu_g(x)dt<+\infty.
\end{eqnarray*}
Then $\sup_{M}u(t,\cdot)\leq 0$ for any $t\in(0,T]$.

\end{prop}

The proof of proposition \ref{ppe-max-int-wei-hea-equ} is a straightforward adaptation of the unweighted case $w=0$ given in an unpublished paper by Karp and Li : [Chap.7,\cite{Ben}].

\begin{theo}\label{iso-weighted-lap-I-bis}
Let $(M,g,\nabla f)$ be an expanding gradient Ricci soliton with bounded curvature such that $f$ is proper.
Then, for any positive $\alpha$ and any real number $\gamma$, the operator
\begin{eqnarray*}
h\in D^{2,\theta}_{0,\nabla f}(M,E)\longmapsto\Delta_{v-\gamma\ln v}h-\alpha h\in C^{0,\theta}_0(M,E)
\end{eqnarray*}
is an isomorphism of Banach spaces for $\theta\in(0,1)$. In particular, the operator $$\Delta_f: D_{0,f^{\alpha},\nabla f}^{2,\theta}(M,E)\rightarrow C_{0,con,f^{\alpha}}^{0,\theta}(M,E)$$ is an isomorphism of Banach spaces.
 \end{theo}
 
 \begin{proof}
 According to theorem \ref{first-iso-sch}, it suffices to prove that if $H\in C^{0,\theta}_0(M,E)$ and if $h$ is the unique solution in $D^{2,\theta}_{\nabla f}(M,E)$ to $\Delta_{v-\gamma\ln v}h-\alpha h=H$ then $\lim_{+\infty}h=0$. We recall that $h$ is given by the following formula : 
 \begin{eqnarray*}
h=\int_0^{\infty}e^{-\alpha t}T(t)Hdt,
\end{eqnarray*}
where $(T(t))_{t\geq 0}$ is the one parameter semigroup generated by $\Delta_{v-\gamma\ln v}$.
Therefore, it suffices to prove that $\lim_{+\infty}T(t)H=0$ for any positive time at a controlled rate. More specifically, we claim that, for any positive $\epsilon$, there is some positive radius $r(\epsilon)$ such that for any positive time $t$,
\begin{eqnarray*}
\sup_{v\geq r(\epsilon)}\arrowvert T(t)H\arrowvert\leq \epsilon e^{\epsilon t}.
\end{eqnarray*}
By considering the norm of $T(t)H$, we are reduced to consider a nonnegative subsolution of the heat equation for the weighted laplacian $\Delta_{v-\gamma\ln v}$. From now on, let $u:M^n\times[0,T]\rightarrow \mathbb{R}_+$ be a subsolution to the weighted heat equation with weight $v-\gamma\ln v$ such that $\lim_{+\infty}u(0,\cdot)=0$ and such that $u(t,\cdot)$ is bounded for any nonnegative time $t$ by $\sup_Mu(0,\cdot)$.

Let $\psi_{\epsilon,r}:=\psi(v/r)+\epsilon,$ where $\epsilon$ and $r$ are positive and where $\psi:[0,+\infty)\rightarrow[0,1]$ is such that,
\begin{eqnarray*}
&&\psi\arrowvert[0,1]\equiv 0\quad;\quad\psi\arrowvert[2,+\infty)\equiv1\quad;\\
&&\quad 0\leq\psi'\leq C\quad;\quad \arrowvert\psi''\arrowvert\leq C.
\end{eqnarray*}
Let $r_0(\epsilon)>0$ such that $u(x,0)\leq \epsilon$ if $v(x)\geq r_0(\epsilon)$. We compute the evolution of the function $\psi_{\epsilon,r}u$ as follows :
\begin{eqnarray*}
\partial_t(\psi_{\epsilon,r}u)&=&\psi_{\epsilon,r}\partial_tu\leq \psi_{\epsilon,r}\Delta_{v-\gamma\ln v}u\\
&\leq&\Delta_{v-\gamma\ln v}(\psi_{\epsilon,r}u)-2<\nabla\psi_{\epsilon,r},\nabla u>-u\Delta_{v-\gamma\ln v}\psi_{\epsilon,r}\\
&\leq&\Delta_{v-\gamma\ln v-2\ln\psi_{\epsilon,r}}(\psi_{\epsilon,r}u)+\left(2\arrowvert\nabla\ln \psi_{\epsilon,r}\arrowvert^2-\frac{\Delta_{v-\gamma\ln v}\psi_{\epsilon,r}}{\psi_{\epsilon,r}}\right)(\psi_{\epsilon,r}u).
\end{eqnarray*}
Now, on $\{r\leq v\leq 2r\}$,
\begin{eqnarray*}
\nabla\psi_{\epsilon,r}=\psi^{'} \frac{\nabla v}{r}\quad;\quad\Delta \psi_{\epsilon,r}=\psi^{''}\frac{\arrowvert\nabla v\arrowvert^2}{r^2}+\psi^{'}\frac{\Delta v}{r}\quad;\quad\\
<\nabla v,\nabla\psi_{\epsilon,r}>=\psi^{'} \frac{\arrowvert\nabla v\arrowvert^2}{r}\geq 0\quad;\quad <\nabla \ln v,\nabla\psi_{\epsilon,r}>=\psi^{'} \frac{\arrowvert\nabla v\arrowvert^2}{vr},
\end{eqnarray*}
so that, 
\begin{eqnarray*}
2\arrowvert\nabla\ln \psi_{\epsilon,r}\arrowvert^2-\frac{\Delta_{v-\gamma\ln v}\psi_{\epsilon,r}}{\psi_{\epsilon,r}}\leq \frac{C}{r\epsilon^2}+\frac{C}{r\epsilon},
\end{eqnarray*}
where $C$ is a positive constant uniform in $r$ and $\epsilon$, since $\arrowvert\nabla v\arrowvert^2\leq v$ by proposition \ref{pot-fct-est}.
Define $r(\epsilon):=\max\{\epsilon^{-3},r_0(\epsilon)\}$ so that
\begin{eqnarray*}
\partial_t(\psi_{\epsilon,r}u)\leq \Delta_{v-\gamma\ln v-2\ln\psi_{\epsilon,r}}(\psi_{\epsilon,r}u)+C\epsilon (\psi_{\epsilon,r}u),
\end{eqnarray*}
i.e. $U_{\epsilon,r}(t,\cdot):=e^{-C\epsilon t}\psi_{\epsilon,r}u$ is a subsolution to the heat equation with weight $v-\gamma\ln v-2\ln\psi_{\epsilon,r}$. Moreover, 
\begin{eqnarray*}
U_{\epsilon,r}(0,\cdot)\leq C\epsilon\quad;\quad U_{\epsilon,r}(t,\cdot)\leq (1+\epsilon)\sup_Mu(0,\cdot)\quad\forall t\geq 0.
\end{eqnarray*}
We are now in a position to apply proposition \ref{ppe-max-int-wei-hea-equ} to the subsolution $U_{\epsilon,r}(t,\cdot)-C\epsilon$. The only thing to check is the integrability assumption :
\begin{eqnarray*}
&&\int_0^T\int_{M}e^{-\beta d_g^2(x,p)}\max\{U_{\epsilon,r}(t,\cdot)-C\epsilon,0\}^2e^{v}v^{-\gamma}\psi_{\epsilon,r}^{-2}d\mu_gdt\\&&\leq C(\epsilon,g)\int_0^T\int_{M}e^{(1/4-\beta)d_g^2(x,p)}d\mu_gdt\\
&&\leq C(\epsilon,g,T)\int_{0}^{+\infty}e^{(1/4-\beta)r^2}A(p,r)dr\\
&&\leq C(\epsilon,g,T)\int_{0}^{+\infty}e^{(1/4-\beta)r^2+C(n,g)r}dr<+\infty
\end{eqnarray*}
as soon as $\beta >1/4$ for some fixed point $p\in M$, where we used the fact that the function $v$ is positive and at most quadratically growing together with the Bishop theorem since the (Ricci) curvature is bounded from below.

Therefore, since $T$ is arbitrary,
\begin{eqnarray*}
U_{\epsilon,r}(t,x)\leq C\epsilon,\quad\mbox{on $M^n\times[0,+\infty)$},
\end{eqnarray*}
i.e.
\begin{eqnarray*}
u(t,x)\leq C\epsilon e^{C\epsilon t}, \quad\mbox{on $\{f\geq 2r(\epsilon)\}\times[0,+\infty)$}.
\end{eqnarray*}

By integrating over $[0,+\infty)$, this shows that
\begin{eqnarray*}
\sup_{v\geq r(\epsilon)}\arrowvert h\arrowvert\leq C\epsilon,
\end{eqnarray*}
i.e. $\lim_{+\infty}h=0$ since $v$ is proper.

 \end{proof}
 
 \subsubsection{The cases $k\geq 1$}
 We state and prove the following result that is the analogue of theorem \ref{iso-weighted-lap-II}.
 
 \begin{theo}\label{iso-weighted-lap-II-bis}
Let $(M^n,g,\nabla f)$ be an asymptotically conical expanding gradient Ricci soliton. Then, the operator
\begin{eqnarray*}
\Delta_{v-2\alpha\ln v}-\alpha:D^{k+2,\theta}_{0,\nabla f}(M,E)\rightarrow C^{k,\theta}_{0,con}(M,E),
\end{eqnarray*}
is an isomorphism of Banach spaces for any $\alpha\in(1/2,+\infty)$ and any $\theta\in(0,1)$. In particular, the operator $$\Delta_f: D_{0,f^{\alpha},\nabla f}^{k+2,\theta}(M,E)\rightarrow C_{0,con,f^{\alpha}}^{k,\theta}(M,E),$$ is an isomorphism of Banach spaces  for any $\alpha\in(1/2,+\infty)$ and any $\theta\in(0,1)$.
\end{theo}

\begin{proof}
Let $H\in C^{k,\theta}_{0,con}(M,E)$ and let $h$ be the unique solution in $D^{k+2,\theta}_{\nabla f}(M,E)$ to $\Delta_{v-2\alpha \ln v}h-\alpha h=H$ given by theorem \ref{iso-weighted-lap-II}. Then we claim that $h\in C^{k,\theta}_{0,con}(M,E)$.
Indeed, it is shown in the proof of theorem \ref{iso-weighted-lap-II} that $v^{k/2}\nabla^kh=:h_k$ is the unique bounded solution to 
\begin{eqnarray*}
\Delta_{v-(2\alpha+k)\ln v}h_k-\alpha h_k=(\Rm(g)+v^{-1})\ast h_k+H_k+\sum_{i=0}^{k-1}v^{\frac{k-i}{2}}\nabla^{k-i}\Rm(g)\ast h_i,
\end{eqnarray*}
where $H_k:=v^{k/2}\nabla^kH$.
By induction on $k$ and by assumption on the curvature decay, the right hand side goes to zero at infinity and so is $h_k$ by theorem \ref{iso-weighted-lap-I-bis} with $\gamma:=2\alpha+k$.

\end{proof}

\begin{coro}\label{sec-iso-sch-II-bis}
Let $(M^n,g,\nabla f)$ be an asymptotically conical expanding gradient Ricci soliton with positive curvature operator.

Then, for any $\alpha\in(1/2,+\infty)$ and any $\theta\in(0,1)$, $L_{\alpha}$ is an isomorphism or equally,
\begin{eqnarray*}
L_0:D^{k+2,\theta}_{0,f^{\alpha},\nabla f}(M,S^2T^*M)\rightarrow C^{k,\theta}_{0,con,f^{\alpha}}(M,S^2T^*M)
\end{eqnarray*}
is an isomorphism of Banach spaces.
\end{coro}

\subsection{Tameness of the linearized operator}\label{Sec-Tam-Lin-Ope}
As said in the introduction, in order to apply the Nash-Moser theorem, one needs to invert the linearized operator on a neighborhood of $(\gamma_0,0)$ in $C^{\infty}(X,S^2T^*X)\times D^{\infty}_{0,f^{\alpha},\nabla f}(M,S^2T^*M)$. This is the purpose of the following theorem :
\begin{theo}\label{Lin-Iso-Fre}
Let $(M^n,g_0,\nabla^{g_0} f_0)$ be an asymptotically conical expanding gradient Ricci soliton. Let $(C(X),dr^2+r^2\gamma_0,o)$ be its asymptotic cone. Then, 
\begin{enumerate}
\item  For $\alpha\in(1/2,+\infty)$, the Fréchet spaces $D^{\infty}_{0,f^{\alpha},\nabla f}(M,E)$ and $C^{\infty}_{0,con,f^{\alpha}}(M,E)$ are tame.\\
\item There exists a neighborhood $U(\gamma_0,0)\subset (C^{\infty}(X,S^2T^*X))\times D^{\infty}_{0,f^{\alpha},\nabla f}(M,S^2T^*M))$ of $(\gamma_0,0)$ such that the family of linear maps $$L:U((\gamma_0,0))\times D^{\infty}_{0,f^{\alpha},\nabla f}(M,S^2T^*M))\rightarrow C^{\infty}_{0,con,f^{\alpha}}(M,S^2T^*M),$$
defined by 
\begin{eqnarray*}
L(\gamma,h)\bh&:=&\Delta_{T(\gamma)+h,f_0}\bh+2\Rm(T(\gamma)+h)\ast \bh\\
&&+\frac{1}{2}\Sym(Q(T(\gamma)+h,T(\gamma))\otimes \bh)+\Li_{\bh(\nabla^{T(\gamma)}f_0-\nabla^{T(\gamma)+h}f_0)}(T(\gamma)+h),
\end{eqnarray*}
 defines a smooth invertible tame map whose inverse is also tame for $\alpha\in(1/2,+\infty)$.
\end{enumerate}
\end{theo}

\begin{rk}
The reason why we consider the map $L$ in theorem \ref{Lin-Iso-Fre} will be legitimated by the proof of theorem \ref{Sec-Def-Exp-Fre}.
\end{rk}

\begin{proof}
\begin{enumerate}
\item According to theorem \ref{iso-weighted-lap-II-bis}, the operator 
\begin{eqnarray*}
\Delta_{v}: D^{\infty}_{0,f^{\alpha},\nabla f}(M,E)\rightarrow C^{\infty}_{0,con,f^{\alpha}}(M,E),
\end{eqnarray*}
is well-defined, is invertible, tame and its inverse is also tame for $\alpha>1/2$. Therefore, $D^{\infty}_{0,f^{\alpha},\nabla f}(M,E)$ is tame if and only if $C^{\infty}_{0,con,f^{\alpha}}(M,E)$ is. We claim that $C^{\infty}_{0,con,f^{\alpha}}(M,E)$ is tame for any positive $\alpha$. We do the proof for the bundle $E= S^2T^*M$, the proof in general can be adapted.

First, we remark that we can localize the argument outside a sufficiently large compact set of $M$ by using a suitable cutoff function.
Secondly, as $(M,g,\nabla f)$ is smoothly asymptotic to a metric cone $(C(X),dr^2+r^2g_X,r\partial_r/2)$ as explained at the beginning of section \ref{sec-exp-str-I}, by pulling back tensors on $M$, we are reduced to show that the corresponding weighted spaces $C^{\infty}_{0,con,r^{\alpha}}(C(X)\setminus B(o,1))$ where $\alpha$ is a nonnegative number, are tame. Actually, we are even reduced to show that $C^{\infty}_{0,con}(C(X)\setminus B(o,1))$ is tame since the map $h\in C^{\infty}_{0,con,r^{\alpha}}(C(X)\setminus B(o,1))\rightarrow r^{\alpha}h\in C^{\infty}_{0,con}(C(X)\setminus B(o,1))$ is an isometry.

Now, we reduce again the claim to manifolds that are asymptotically cylindrical by the following correspondence. Indeed, define the following diffeomorphism :
\begin{eqnarray*}
((0,+\infty)\times X,dt^2+g_X)&\stackrel{\psi}{\longmapsto}&(C(X)\setminus \bar{B}(o,1),dr^2+r^2g_X)\\
(t,x)&\longmapsto &(e^t,x).
\end{eqnarray*}

Then, $\psi^*(dr^2+r^2g_X)=e^{2t}(dt^2+g_X)$. If $g_{con}:=dr^2+r^2g_X$ and if $g_{cyl}:=dt^2+g_X$ then the previous equality reads $\psi^*g_{con}=e^{2t}g_{cyl}.$ Define the corresponding weighted spaces in the cylindrical case :
\begin{eqnarray*}
&&C^{\infty}_{0,cyl}([0,+\infty)\times X):=\\
&&\left\{h\in C^{\infty}_{loc}\quad|\quad \sup_{[0,+\infty)\times X}\arrowvert\nabla^kh\arrowvert<+\infty,\quad\limsup_{+\infty}\arrowvert\nabla^kh\arrowvert=0\quad\forall k\geq 0\right\}.
\end{eqnarray*}

One can check that $h_{con}\in C^{\infty}_{0,con}(C(X)\setminus B(o,1))$ if and only if $h_{cyl}\in C^{\infty}_{0,cyl}([0,+\infty)\times X)$ with the notations previously introduced.

Finally, we invoke corollary $1.3.8$ of \cite{Ham-Nas-Mos} that ensures that, if $N$ is a compact manifold with boundary, the space $C_{0}^{\infty}(N,E)$ of sections vanishing at the boundary of $N$ together with all their derivatives is tame.\\
\item
As in the proof of theorem \ref{Loc-Dif-Ban}, our approach is actually an adaptation of the proof of theorem $3.3.1$ of \cite{Ham-Nas-Mos}. We use massively interpolation inequalities again that hold here when the parameter $\theta$ is $0$. As noticed in remark \ref{rk-non-int-spa}, corresponding interpolation inequalities for an arbitrary $\theta\in(0,1)$ are not straightforward. We will only use the simple fact that, for a given nonnegative integer $k$ and some $\theta\in(0,1)$, the norm $\|\cdot\|_{k,\theta}$ (on the spaces considered above) is dominated by the norm $\|\cdot\|_{k+1}$. The drawback is that we establish less precise tame estimates for the map $L$ and its inverse : this will not affect the use of the Nash-Moser theorem.

We temporarily simplify the notations. \\

Denote by $(\|\cdot\|_{k+2,\theta}^{D_f})_{k\geq 0}$ (respectively $(\|\cdot\|_{k,\theta}^{C_f})_{k\geq 0}$, respectively $(\|\cdot\|_{k,\theta}^{C})_{k\geq 0}$)   the collection of norms induced on $D^{\infty}_{0,f^{\alpha},\nabla f}(M,S^2T^*M)$ (respectively on $C^{\infty}_{0,con,f^{\alpha}}(M,S^2T^*M)$, respectively on $C^{\infty}_{con}(M,S^2T^*M)$).
Finally, we denote by $\Gamma:=(T(\gamma),h)$ an element of $C^{\infty}_{con}(M,S^2T^*M)\times D^{\infty}_{0,f^{\alpha},\nabla f}(M,S^2T^*M)$. 
\begin{claim}\label{cla-tam-est-fam-lin-map}
There exists a nonnegative integer $k_0$ such that if $\|\Gamma-\Gamma_0\|_{k_0,\theta}^{C\times D_f}$ is bounded then
\begin{eqnarray}
\| L(\Gamma)\bh-L(\Gamma_0)\bh\|_{k,\theta}^{C_f}&\lesssim&\sum_{i=0}^k\|\bh\|^{D_f}_{k-i+2,\theta}\|\Gamma-\Gamma_0\|^{C\times D_f}_{i+2,\theta},\\
&&\\
\| L(\Gamma)\bh\|^{C_f}_{k,\theta}&\lesssim& \|\bh\|^{D_f}_{k+2,\theta}+\|\Gamma-\Gamma_0\|^{C\times D_f}_{k+2,\theta}\|\bh\|^{D_f}_{2+2,\theta},\label{eq-tam-est-lin-ope}
\end{eqnarray}
where the symbol $\lesssim$ means up to a multiplicative positive constant (independent of the variables).
\end{claim}
\begin{proof}[Proof of claim \ref{cla-tam-est-fam-lin-map}]
 Define $\Gamma_0:=(\gamma_0,0)$ and let $\Gamma\in U(\Gamma_0)$ such that $T(\gamma)$ is well-defined, i.e. is a metric. Then, by using the variation formula (\ref{Var-For-2-Met-Con}), and the fact that $\nabla^{T(\gamma)}f_0=\nabla^{T(\gamma_0)}f_0$ outside a compact set,
 \begin{eqnarray*}
(\Delta_{T(\gamma)+h}-\Delta_{T(\gamma_0)})\bh&=&\left[(T(\gamma)+h)^{-1}\ast (h(\gamma)+h)\right]\ast\nabla^{g_0,2}\bh\\
&&+\left[(T(\gamma)+h)^{-2}\ast\nabla^{g_0}(h(\gamma)+h)\right]\ast\nabla^{g_0}\bh\\
&&+[(T(\gamma)+h)^{-2}\ast\nabla^{g_0,2}(h(\gamma)+h)\\
&&+(T(\gamma)+h)^{-3}\ast\nabla^{g_0}(h(\gamma)+h)^{*2}]\ast \bh,\\
&&\\
\left[\nabla^{T(\gamma)+h}_{\nabla^{T(\gamma)+h}f_0}-\nabla^{T(\gamma_0)}_{\nabla^{T(\gamma_0)}f_0}\right]\bh&=&\nabla^{g_0}_{\left[\nabla^{T(\gamma)+h}-\nabla^{T(\gamma_0)}\right]f_0}\bh\\
&&+(T(\gamma)+h)^{-1}\ast\nabla^{g_0}(h(\gamma)+h)\ast\nabla^{T(\gamma)+h}f_0\ast \bh,\\
&&\\
\left[\nabla^{T(\gamma)+h}-\nabla^{T(\gamma_0)}\right]f_0&=&\left[\nabla^{T(\gamma)+h}-\nabla^{T(\gamma)}\right]f_0\\
&=&(T(\gamma)+h)^{-1}\ast h\ast\nabla^{g_0}f_0,\\
&&\\
\Rm(T(\gamma)+h)-\Rm(g_0)&=&(T(\gamma)+h)^{-1}\ast\nabla^{g_0,2}(h(\gamma)+h)\\
&&+(T(\gamma)+h)^{-2}\ast\nabla^{g_0}(h(\gamma)+h)^{*2},\\
&&\\
\bh(\nabla^{T(\gamma)}f_0-\nabla^{T(\gamma)+h}f_0)&=&\bh\ast(T(\gamma)+h)^{-1}\ast h \ast\nabla^{g_0}f_0,\\
&&\\
\Li_{\bh(\nabla^{T(\gamma)}f_0-\nabla^{T(\gamma)+h}f_0)}(T(\gamma)+h)&=&\nabla^{g_0}(h(\gamma)+h)\ast\bh\ast(T(\gamma)+h)^{-1}\ast h\ast\nabla^{g_0}f_0\\
&&+\nabla^{g_0}\left[\bh\ast(T(\gamma)+h)^{-1}\ast h\ast \nabla^{g_0}f_0\right]\ast (T(\gamma)+h).
\end{eqnarray*}
To sum it up, we get, after suppressing the contractions with $(T(\gamma)+h)$ or its inverse that are inoffensive as seen in the proof of theorem \ref{Loc-Dif-Ban} (claim \ref{cla-inv-ham}) :
\begin{eqnarray*}\label{est-dif-op-lin-1}
L(\gamma,h)\bh-L(\gamma_0,0)\bh&=&\nabla^{g_0,2}\bh\ast\left[ (h(\gamma)+h)\right]\\
&&\\
&&+\nabla^{g_0}\bh\ast[\nabla^{g_0}(h(\gamma)+h)+h\ast\nabla^{g_0}f_0]\label{est-dif-op-lin-2}\\
&&\\
&&+\bh\ast[\nabla^{g_0,2}(h(\gamma)+h)+\nabla^{g_0}(h(\gamma)+h)^{*2}]\\
&&\\
&&+\bh\ast[\nabla^{g_0}(h(\gamma)+h)\ast\nabla^{g_0}f_0\ast(1+h)]\\
&&\\
&&+\bh\ast[h\ast\nabla^{g_0,2}f_0+Q(T(\gamma)+h,T(\gamma)).]\label{est-dif-op-lin-3}
\end{eqnarray*}

We then estimate each term as we did in the proof of theorem \ref{Loc-Dif-Ban}. The term involving $Q(T(\gamma)+h,T(\gamma))$ has already been estimated : (\ref{tame-est-phi}). To simplify notations again, we will denote $\nabla^{g_0}$ by $\nabla$. By using theorem \ref{iso-weighted-lap-II-bis}, 
\begin{eqnarray*}
\|v_0^{k/2}\nabla^{k}(\nabla^{2}\bh\ast (\Gamma-\Gamma_0))\|^{C_f}_{0,\theta}&\lesssim&\sum_{i=0}^k\|\nabla^{2}(v^{i/2}\nabla^{i}\bh)\|^{C_f}_{0,\theta}\|\Gamma-\Gamma_0\|^{C\times D_f}_{k+2-i,\theta}\\
&\lesssim&\sum_{i=0}^k\|\bh\|_{i+2,\theta}^{D_f}\|\Gamma-\Gamma_0\|^{C\times D_f}_{k+2-i,\theta}\\
&&\\
\| v_0^{k/2}\nabla^{k}(\nabla\bh\ast [\nabla(\Gamma-\Gamma_0)+h\ast\nabla f_0])\|^{C_f}_{0,\theta}&\lesssim&\sum_{i=0}^k\|\nabla(v_0^{i/2}\nabla^{i}\bh)\|^{C_f}_{0,\theta}\|\Gamma-\Gamma_0\|^{C\times D_f}_{k+2-i,\theta}\\
&&+\sum_{i=0}^k\|\nabla(v_0^{(k-i)/2}\nabla^{k-i}\bh)\|^{C_f}_{0,\theta}\|v_0^{i/2}\nabla^i(\nabla f_0\ast h)\|^C_{0,\theta}\\
&\lesssim& \sum_{i=0}^k\|\bh\|_{k-i+2,\theta}^{D_f}\|\Gamma-\Gamma_0\|_{i+2,\theta}^{C\times D_f}\\
&&+\sum_{i=0}^k\|\bh\|^{D_f}_{k-i+2,\theta}\|h\|^{C_f}_{i,\theta}\\
&\lesssim& \sum_{i=0}^k\|\bh\|_{k-i+2,\theta}^{D_f}\|\Gamma-\Gamma_0\|_{i+2,\theta}^{C\times D_f}\\
&&\\
\| v_0^{k/2}\nabla^{k}(\nabla^{2}(\Gamma-\Gamma_0)\ast \bh)\|^{C_f}_{0,\theta}&\lesssim& \sum_{i=0}^k\|\Gamma-\Gamma_0\|_{k+2-i,\theta}^{C\times D_f}\|v_0^{i-2}\nabla^i \bh\|^{C_f}_{0,\theta}\\
&\lesssim&\sum_{i=0}^{k}\|\Gamma-\Gamma_0\|_{k+2-i,\theta}^{C\times D_f}\| \bh\|^{D_f}_{i+2,\theta},\\
&&\\
\| v_0^{k/2}\nabla^k(\nabla f_0\ast\nabla(\Gamma-\Gamma_0)\ast \bh)\|^{C_f}_{0,\theta}&\lesssim&\sum_{i=0}^k\|\bh\|_{k-i+2,\theta}^{D_f}(\|h(\gamma)\|^{C}_{i+1,\theta}+\|\nabla (v_0^{i/2}\nabla^i(\nabla f_0h))\|^{C}_{0,\theta})\\
&\lesssim&\sum_{i=0}^k\|\bh\|_{k-i+2,\theta}^{D_f}(\|h(\gamma)\|^{C}_{i+2,\theta}+\|h\|^{D_f}_{i+2,\theta})\\
&\lesssim&\sum_{i=0}^k\|\Gamma-\Gamma_0\|_{i+2,\theta}^{C\times D_f}\|\bh\|_{k-i+2,\theta}^{D_f},\\
&&\\
\| v_0^{k/2}\nabla^k(\nabla f_0\ast\nabla(\Gamma-\Gamma_0) \ast h\ast\bh)\|^C_{0,\theta}&\lesssim&\sum_{i=0}^k\|\bh\|^{D_f}_{k-i+2,\theta}\left(\sum_{j=0}^i\|\Gamma-\Gamma_0\|^{C\times D_f}_{j+2,\theta}\|h\|^{D_f}_{i-j+2,\theta}\right)\\
&\lesssim&\sum_{i=0}^k\|\bh\|^{D_f}_{k-i+2,\theta}\left(\sum_{j=0}^i\|\Gamma-\Gamma_0\|^{C\times D_f}_{j+2,\theta}\|\Gamma-\Gamma_0\|^{C\times D_f}_{i-j+2,\theta}\right)\\
&\lesssim&\sum_{i=0}^k\|\bh\|^{D_f}_{k-i+2,\theta}\|\Gamma-\Gamma_0\|^{C\times D_f}_{i+2,\theta},
\end{eqnarray*}
if $\|\Gamma-\Gamma_0\|_{k_0,\theta}^{C\times D_f}\leq C$ with $k_0$ large enough ($k_0=100$ does the job !). So far, we have proved that, if $\Gamma$ is in a neighborhood of $\Gamma_0$ : 
\begin{eqnarray*}
\| L(\Gamma)\bh\|_{k,\theta}^{C_f}&\lesssim&\| L(\Gamma_0)\bh\|_{k,\theta}^{C_f}+ \| L(\Gamma)\bh-L(\Gamma_0)\bh\|_{k,\theta}^{C_f}\\
&\lesssim&\|\bh\|_{k+2,\theta}^{D_f}+\sum_{i=0}^k\|\bh\|^{D_f}_{k-i+2,\theta}\|\Gamma-\Gamma_0\|^{C\times D_f}_{i+2,\theta}\\
&\lesssim&\|\bh\|_{k+2,\theta}^{D_f}+\sum_{i=1}^{k-1}\|\bh\|^{D_f}_{k-i+2+1}\|\Gamma-\Gamma_0\|^{C\times D_f}_{i+2+1}+\|\bh\|^{D_f}_{2+1,\theta}\|\Gamma-\Gamma_0\|^{C\times D_f}_{k+1,\theta}\\
&\lesssim&\|\bh\|_{k+2,\theta}^{D_f}+\|\Gamma-\Gamma_0\|^{C\times D_f}_{k+2}\|\bh\|^{D_f}_{2+2}+\|\Gamma-\Gamma_0\|^{C\times D_f}_{k+2,\theta}\|\bh\|^{D_f}_{2+2,\theta}\\
\\
&\lesssim&\|\bh\|_{k+2,\theta}^{D_f}+\|\Gamma-\Gamma_0\|^{C\times D_f}_{k+2,\theta}\|\bh\|^{D_f}_{2+2,\theta}\\
\end{eqnarray*}
hence the claim by using interpolation inequalities.

\end{proof}

We follow the arguments due to Hamilton \cite{Ham-Nas-Mos} to show that the family of inverse maps $L^{-1}$ is tame on a neighborhood of $\Gamma_0=(\gamma_0,0)$. We need three ingredients : the three first steps $k=0,1,2$ (because of claim \ref{cla-tam-est-fam-lin-map}), the tame estimates for the map $L$ given by claim \ref{cla-tam-est-fam-lin-map} and adequate interpolation inequalities. As we said before, we do have interpolation inequalities when the parameter $\theta$ is $0$.  

The first steps $k=0,1,2$ are straightforward, regarding claim \ref{cla-tam-est-fam-lin-map} together with the fact that $L(\Gamma_0)$ is an isomorphism by theorem \ref{iso-weighted-lap-II-bis} : if $\|\Gamma-\Gamma_0\|^{C\times D}_{k_0,\theta}$ is small enough for some large $k_0$, then 
\begin{eqnarray*}
\|\bh\|^{D_f}_{2+2,\theta}\lesssim \|L(\Gamma)\bh\|_{2,\theta}^{C_f},\quad\bh\in D^{\infty}_{0,f^{\alpha},\nabla f}(M,S^2T^*M).
\end{eqnarray*}

\begin{claim}\label{tame-est-inv-map}
For any integer $k\geq 2$,
\begin{eqnarray*}
&&\|\bh\|_{k+2,\theta}^{D_f}\lesssim \| L(\Gamma)(\bh)\|_{k,\theta}^{C_f}+\Sigma_{i=0}^{k-2}\|L(\Gamma)(\bh)\|_{i+2,\theta}^{C_f}\|\Gamma-\Gamma_0\|^{C\times D_f}_{k-i+2,\theta},\\
&&\\
&&(\Gamma,\bh)\in  (C^{\infty}_{con}(M,S^2T^*M)\times D^{\infty}_{0,f^{\alpha},\nabla f}(M,S^2T^*M))\times D^{\infty}_{0,f^{\alpha},\nabla f}(M,S^2T^*M),\\
&&\\
&&\Gamma\in U(\Gamma_0),
\end{eqnarray*}
where the symbol $\lesssim$ means up to a multiplicative positive constant (independent of the variables).
\end{claim}

\begin{proof}[Proof of claim \ref{tame-est-inv-map}]
We prove claim \ref{tame-est-inv-map} by induction on $k$. \\

By applying the induction assumption to $v_0^{1/2}\nabla \bh$, we get :
\begin{eqnarray*}
\| \bh\|_{k+1+2,\theta}^{D_f}&\lesssim& \| L(\Gamma)(v_0^{1/2}\nabla \bh)\|_{k,\theta}^{C_f}+\| \bh\|_{k+2,\theta}^{D_f}\\
&&+\sum_{i=0}^{k-2}\|L(\Gamma)(v_0^{1/2}\nabla\bh)\|_{i+2,\theta}^{C_f}\|\Gamma-\Gamma_0\|^{C\times D_f}_{k-i+2,\theta}.
\end{eqnarray*}
Now, as explained in lemma $3.3.2$ of \cite{Ham-Nas-Mos}, by the Leibniz rule,
\begin{eqnarray*}
L(\Gamma)(v_0^{1/2}\nabla \bh)=v_0^{1/2}\nabla(L(\Gamma)\bh)-L(v_0^{1/2}\nabla \Gamma)\bh,
\end{eqnarray*}
where $L(v_0^{1/2}\nabla \Gamma)$ denotes some second order differential operator satisfying the same tame estimate (\ref{eq-tam-est-lin-ope}) satisfied by $L(\Gamma)$. Therefore,
\begin{eqnarray*}
\| \bh\|_{k+1+2,\theta}^{D_f}&\lesssim& \| L(\Gamma)\bh\|_{k+1,\theta}^{C_f}+\| L(v_0^{1/2}\nabla\Gamma)\bh\|_{k,\theta}^{C_f}\\
&&+\sum_{i=0}^{k-2}\|L(\Gamma)\bh)\|_{i+2+1,\theta}^{C_f}\|\Gamma-\Gamma_0\|^{C\times D_f}_{k+1-(i+1)+2,\theta}+\| \bh\|_{k+2,\theta}^{D_f}\\
&&+\sum_{i=0}^{k-2}\|L(v_0^{1/2}\nabla\Gamma)\bh)\|_{i+2,\theta}^{C_f}\|\Gamma-\Gamma_0\|^{C\times D_f}_{k-i+2,\theta}\\
&&\\
&\lesssim&\| L(\Gamma)\bh\|_{k+1,\theta}^{C_f}+\left(\| \bh\|_{k+2,\theta}^{D_f}+\|v_0^{1/2}\nabla (\Gamma-\Gamma_0)\|_{k+2,\theta}^{C\times D_f}\|\bh\|^{D_f}_{2+2,\theta}\right)\\
&&+\sum_{i=1}^{(k+1)-2}\|L(\Gamma)\bh)\|_{i+2,\theta}^{C_f}\|\Gamma-\Gamma_0\|^{C\times D_f}_{k+1-i+2,\theta}\\
&&+\sum_{i=0}^{k-2}\left(\| \bh\|_{i+2+2,\theta}^{D_f}+\|v_0^{1/2}\nabla( \Gamma-\Gamma_0)\|_{i+2,\theta}^{C\times D_f}\|\bh\|_{2+2,\theta}^{D_f}\right)\|\Gamma-\Gamma_0\|^{C\times D_f}_{k-i+2,\theta}.\\
\end{eqnarray*}
Using the initial step $k=2$ gives :
\begin{eqnarray*}
\| \bh\|_{k+1+2,\theta}^{D_f}&\lesssim&\| L(\Gamma)\bh\|_{k+1,\theta}^{C_f}+\sum_{i=0}^{(k+1)-2}\|L(\Gamma)\bh)\|_{i+2,\theta}^{C_f}\|\Gamma-\Gamma_0\|^{C\times D_f}_{(k+1)-i+2,\theta}\\
&&+\sum_{i=0}^{k-2}\left(\| \bh\|_{i+2+2,\theta}^{D_f}+\| \Gamma-\Gamma_0\|_{i+3,\theta}^{C\times D_f}\|L(\Gamma)\bh\|^{C_f}_{2,\theta}\right)\|\Gamma-\Gamma_0\|^{C\times D_f}_{k-i+2,\theta}.\\
\end{eqnarray*}

Using the induction assumption again for $i\in\{0,...,k-2\}$ and interpolation inequalities :
\begin{eqnarray*}
\| \bh\|_{k+1+2,\theta}^{D_f}&\lesssim&\| L(\Gamma)\bh\|_{k+1,\theta}^{C_f}+\sum_{i=0}^{(k+1)-2}\|L(\Gamma)\bh)\|_{i+2,\theta}^{C_f}\|\Gamma-\Gamma_0\|^{C\times D_f}_{(k+1)-i+2,\theta}\\
&&+\left(\sum_{i=0}^{k-2}\| \Gamma-\Gamma_0\|_{i+3,\theta}^{C\times D_f}\|\Gamma-\Gamma_0\|^{C\times D_f}_{k-i+2,\theta}\right)\|L(\Gamma)\bh\|^{C_f}_{2,\theta}\\
&&+\sum_{i=0}^{k-2}\| L(\Gamma)(\bh)\|_{i+2,\theta}^{C_f}\|\Gamma-\Gamma_0\|^{C\times D_f}_{k-i+2,\theta}\\
&&+\Sigma_{i=0}^{k-2}\Sigma_{j=0}^{i}\|L(\Gamma)(\bh)\|_{j+2,\theta}^{C_f}\|\Gamma-\Gamma_0\|^{C\times D_f}_{i+2-j+2,\theta}\|\Gamma-\Gamma_0\|^{C\times D_f}_{k-i+2,\theta}\\
&&\\
&\lesssim&\| L(\Gamma)\bh\|_{k+1,\theta}^{C_f}+\sum_{i=0}^{(k+1)-2}\|L(\Gamma)\bh)\|_{i+2,\theta}^{C_f}\|\Gamma-\Gamma_0\|^{C\times D_f}_{(k+1)-i+2,\theta}\\
&&+\left(\sum_{i=0}^{k-2}\| \Gamma-\Gamma_0\|_{i+3,\theta}^{C\times D_f}\|\Gamma-\Gamma_0\|^{C\times D_f}_{k-i+2,\theta}\right)\|L(\Gamma)\bh\|^{C_f}_{2,\theta}\\
&&+\sum_{j=0}^{k-2}\left(\| L(\Gamma)\bh\|^{C_f}_{j+2,\theta}\left(\sum_{i=j}^{k-2}\|\Gamma-\Gamma_0\|^{C\times D_f}_{i+2-j+2,\theta}\|\Gamma-\Gamma_0\|_{k-i+2,\theta}^{C\times D_f}\right)\right)\\
&&\\
&\lesssim&\| L(\Gamma)\bh\|_{k+1,\theta}^{C_f}+\sum_{i=0}^{(k+1)-2}\|L(\Gamma)\bh)\|_{i+2,\theta}^{C_f}\|\Gamma-\Gamma_0\|^{C\times D_f}_{(k+1)-i+2,\theta}\\
&&+\|\Gamma-\Gamma_0\|_{(k+1)+2,\theta}^{C\times D_f}\|L(\Gamma)\bh\|^{C_f}_{2,\theta}+\sum_{j=0}^{k-2}\| L(\Gamma)\bh\|^{C_f}_{j+2,\theta}\|\Gamma-\Gamma_0\|^{C\times D_f}_{k+1-j+2,\theta}
\end{eqnarray*}
for $\Gamma\in U(\Gamma_0)$ where $U(\Gamma_0)$ is a neighborhood of $\Gamma_0$ independent of $k$.

\end{proof}
\end{enumerate}

\end{proof}
\subsection{Deformation of expanders : the Fréchet version}\label{Sec-Def-Exp-Fre}

The main purpose of this section is to prove the following theorem that is the $C^{\infty}$ version of theorem \ref{Loc-Dif-Ban} :

\begin{theo}\label{Loc-Dif-Fre}
Let $(M^n,g_0,\nabla^{g_0} f_0)$ be an asymptotically conical expanding gradient Ricci soliton with positive curvature operator. 

Then, for any $\alpha\in(1/2,1)$, the map 
\begin{eqnarray*}
U(\gamma_0,0)\subset \cat{M}et^{\infty}(X)\times D^{\infty}_{0,f^{\alpha},\nabla f}(M,S^2T^*M)&\stackrel{\Phi_{\alpha}}{\longmapsto}&\cat{M}et^{\infty}(X)\times C^{\infty}_{0,con,f^{\alpha}}(M,S^2T^*M)\\
(\gamma,h)&{\longmapsto}&(\gamma,Q(T(\gamma)+h,T(\gamma)))
\end{eqnarray*}
is well-defined provided $U(\gamma_0,0)$ is a neighborhood of $(\gamma_0,0)$ sufficiently small in  $\cat{M}et^{\infty}(X)\times D^{\infty}_{0,f^{\alpha},\nabla f}(M,S^2T^*M)$. Moreover, $\Phi_{\alpha}$ is a local diffeomorphism at $(\gamma_0,0)$ in the category of tame maps.\\

In particular, for any deformation $\gamma\in \cat{M}et^{\infty}(X)$ close enough to $\gamma_0$, there exists an expanding Ricci soliton with positive curvature operator whose asymptotic cone is $(C(X),dr^2+r^2\gamma,o)$.\\

\end{theo}

\begin{proof}
The hard work has already been done so that it suffices to apply the Nash-Moser theorem proved by Hamilton \cite{Ham-Nas-Mos}.
Indeed, thanks to theorem \ref{Lin-Iso-Fre}, $$\cat{M}et^{\infty}(X)\times D^{\infty}_{0,f^{\alpha},\nabla f}(M,S^2T^*M),$$ is a tame Fréchet space. By theorem \ref{Loc-Dif-Ban} applied to any positive integer $k$, $\Phi_{\alpha}$ is a well-defined smooth tame map. Finally,
the linearized maps  
\begin{eqnarray*}
D_{(\gamma,h)}\Phi_{\alpha}(\bar{\gamma},\bar{h})&=&(\bar{\gamma},D^1_{(T(\gamma)+h,T(\gamma))}Q(D_{\gamma}T(\bar{\gamma})+\bar{h})+D^2_{(T(\gamma)+h,T(\gamma))}Q(D_{\gamma}T(\bar{\gamma})))\\
&=&(\bar{\gamma},D^1_{(T(\gamma)+h,T(\gamma))}Q(D_{\gamma}T(\bar{\gamma}))+D^2_{(T(\gamma)+h,T(\gamma))}Q(D_{\gamma}T(\bar{\gamma}))+L(\gamma,h)(\bar{h}))\\
\end{eqnarray*}
define smooth invertible tame maps whose inverse are also tame for $(\gamma,h)$ in a neighborhood sufficiently small of $(\gamma_0,0)$ according to theorem \ref{Lin-Iso-Fre}.

This is exactly what we need to use the Nash-Moser theorem : hence $\Phi_{\alpha}$ is a local diffeomorphism at $(\gamma_0,0)$ in the category of tame maps.

\end{proof}

\section{Harnack inequality and Cao-Hamilton entropy}\label{cao-ham-sec}
\subsection{Motivations}

We state and prove a general procedure that contains the Einstein case that enables to remove the DeTurck's term present in the deformation.

\begin{prop}\label{Mot-Ent-Cao-Ham}
Let $(M^n,g,\nabla^gf+V)$ be an expanding Ricci soliton, i.e. such that $$2\Ric(g)+g-\Li_{\nabla f}(g)=\Li_{V}(g).$$
Then, 
\begin{eqnarray}
\Delta_fV+\Ric_f(g)(V)+\nabla(\div_{f}V)=\nabla(\R_g+f-2\Delta f-\arrowvert\nabla f\arrowvert^2),
\end{eqnarray}
where $\Ric_f(g):=\Ric(g)-\nabla^{g,2}f$ is the Bakry-Émery tensor.

In particular, if 
\begin{eqnarray}
 \lim_{+\infty}V=0,\quad \mbox{$\Ric_f(g)$ is negative definite on $M$} \label{con-3-mot}
 \end{eqnarray}
  and if
\begin{eqnarray}
\R_g+f-2\Delta f-\arrowvert\nabla f\arrowvert^2=Cst\quad&\mbox{and}&\quad  \div_{f}V=0,\label{con-1-mot}
\end{eqnarray}
or
\begin{eqnarray}
\R_g+f-2\Delta f-\arrowvert\nabla f\arrowvert^2-\div_fV&=&Cst, \label{con-2-mot}
\end{eqnarray}
then $V\equiv 0$.
\end{prop}
\begin{rk}
Proposition \ref{Mot-Ent-Cao-Ham} ensures that the implicit expanding Ricci soliton obtained in theorems \ref{Loc-Dif-Ban} and \ref{Loc-Dif-Fre} is \textbf{gradient} under several assumptions we discuss now. Condition (\ref{con-1-mot}) asks the function $f$ to be a critical point of the entropy $W^+$ introduced by Feldman, Ilmanen and Ni in \cite{Fel-Ilm-Ni}. Nonetheless, $W^+$ is only defined on compact manifolds and makes no sense (for the moment) on noncompact manifolds. The other condition in line (\ref{con-1-mot}) is a gauge condition. Finally, the first condition in line (\ref{con-3-mot}) is actually satisfied by $V=V(T(\gamma)+h,T(\gamma),f_0)$ with the notations of theorem \ref{Loc-Dif-Ban}, the second condition in line (\ref{con-3-mot}) is satisfied as soon as the deformation is $C^2$ close to the original Ricci expander. To sum it up, the reason why we need to introduce an ad-hoc entropy is essentially due to the lack of a general well-defined entropy on noncompact expanders. It might seem artificial to separate condition (\ref{con-1-mot}) from condition (\ref{con-2-mot}) : the first one deals with the non Einstein case whereas the second one applies to the Einstein one.
\end{rk}
\begin{proof}
Let us compute the weighted divergence of $2\Ric(g)+g-\Li_{\nabla f}(g)$ first :
\begin{eqnarray*}
\div_{f}(2\Ric(g)+g)=\nabla \R_g+2\Ric(g)(\nabla f)+\nabla f,
\end{eqnarray*}
by the trace Bianchi identity. Now,
\begin{eqnarray*}
\div_f(\Li_{\nabla f}(g))&=&\nabla\Delta f+\Delta\nabla f+\Ric(g)(\nabla f)+\Li_{\nabla f}(g)(\nabla f)\\
&=&2(\nabla \Delta f+\Ric(g)(\nabla f))+\nabla(\arrowvert\nabla f\arrowvert^2),
\end{eqnarray*}
because of the Bochner formula for functions. Hence,
\begin{eqnarray*}
\div_f(2\Ric(g)+g-\Li_{\nabla f}(g))=\nabla(\R_g+f-2\Delta f-\arrowvert\nabla f\arrowvert^2).
\end{eqnarray*}
On the other hand,
\begin{eqnarray*}
\div_f(\Li_V(g))&=&\nabla(\div V)+\Delta V+\Ric(g)(V)+\Li_V(g)(\nabla f)\\
&=&\nabla(\div_fV)-\nabla <\nabla f,V>+\Delta_f V+\Ric(g)(V)+<\nabla_{\cdot}V,\nabla f>\\
&=&\nabla(\div_fV)+\Delta_f V+\Ric_f(g)(V).
\end{eqnarray*}

\end{proof}
The main purpose of this section is to prove the following proposition :
\begin{theo}\label{trivial-exp-sol-gra}
Let $(M^n,g,V)$ be the implicit expanding Ricci soliton obtained in theorem \ref{Loc-Dif-Ban} or in theorem \ref{Loc-Dif-Fre}
, then there exists a potential function $f:M^n\rightarrow\mathbb{R}$ such that $(M^n,g,\nabla f)$ is an expanding gradient Ricci soliton.
\end{theo}

We give two proofs of theorem \ref{trivial-exp-sol-gra}.\\

The first proof is straightforward and is a direct application of pointwise estimates due to Chow, Hamilton and Ni [Chap. $6$, Sec. $6$, \cite{Ben}]. For the convenience of the reader, we explain the main steps in our setting in the next section.
Note that Schulze and Simon \cite{Sch-Sim} have generalized these arguments in order to associate to any Riemannian manifold with positive asymptotic volume ratio and bounded nonnegative curvature operator an expanding gradient Ricci soliton. Their procedure is based on a blow-down of the initial manifold together with the corresponding Ricci flow and is therefore more involved. 
 
 \subsection{First proof of theorem \ref{trivial-exp-sol-gra}.}\label{Sec-Fir-Pro-No-Bre}

Let $(M^n,g,V)$ be the implicit expanding Ricci soliton obtained in theorem \ref{Loc-Dif-Ban} or in theorem \ref{Loc-Dif-Fre}
. Then $(M^n,g)$ has positive curvature operator (if the regularity at infinity is high enough) and $M^n$ is simply connected (since $M^n$ is diffeomorphic to $\mathbb{R}^n$). Let $(M^n,g(\tau))_{\tau>0}$ be the associated Ricci flow to $(M^n,g,V)$. Then $g(\tau)=\tau \phi_{\tau-1}^*g$, where $(\phi_{s})_{s\in(-1,+\infty)}$ is the flow generated by $-V/(1+s)$. This flow is well-defined since the vector field $V$ grows linearly in the distance at infinity by its very definition.
Following the proof of theorem $10.46$ of \cite{Ben}, consider the quantity
\begin{eqnarray*}
Z(Y):=\Delta\R+2\arrowvert\Ric\arrowvert^2+2<\nabla\R,Y>+2\Ric(Y,Y)+\frac{\R}{\tau},
\end{eqnarray*}
for any vector field $Y$, where we omit the reference to the metric $g(\tau)$.
As $(M^n,g)$ has positive curvature operator, the vector field $$Y:=-\Ric^{-1}\div (\Ric)=-\Ric^{-1}(\nabla\R)/2,$$ is well defined and a (tedious) computation shows that,
\begin{eqnarray}
\partial_{\tau}Z(Y)&=&\Delta Z(Y)+2\left<\Ric, M+2P(Y)+\Rm(Y,\cdot,\cdot,Y)\right>-\frac{2}{\tau}Z(Y)\label{evo-lin-tra-est-1}\\
&&+2\Ric\left(\nabla Y-\Ric-\frac{g}{2\tau},\nabla Y-\Ric-\frac{g}{2\tau}\right),\label{evo-lin-tra-est-2}
\end{eqnarray}
where $M+2P(Y)+\Rm(Y,\cdot,\cdot,Y)$ is the matrix Harnack quadratic given by 
\begin{eqnarray*}
M&:=&\Delta \Ric-\frac{1}{2}\nabla^2\R+2\Rm\ast\Ric-\Ric\otimes\Ric+\frac{\R}{2\tau}\\
P(Y)_{ij}&:=&\nabla_Y\Ric_{ij}-\nabla_i\Ric_{Yj}=\Cod(\Ric)(Y,i,j),
\end{eqnarray*}
where $\Cod(T)$ stands for the Codazzi tensor associated to a tensor $T$.

By the matrix Harnack estimate, one has
\begin{eqnarray}
\partial_{\tau}Z(Y)&\geq&\Delta Z(Y)-\frac{2}{\tau}Z(Y) \label{evo-lin-tra-est-3}.
\end{eqnarray}

Now, since $(M^n,g(\tau))_{\tau>0}$ is an expanding Ricci soliton and since the curvature goes to zero at infinity ($(M^n,g)$ is asymptotically conical !), we have
\begin{eqnarray*}
\sup_{M^n\times (0,+\infty)}\tau\R_{g(\tau)}=\sup_{M^n}\R_{g(1)}=\R_{g(1)}(p_1),
\end{eqnarray*}
for some $p_1\in M^n$. Therefore,
\begin{eqnarray*}
\partial_{\tau}(\tau\R_{g(\tau)})(p_1,1)=0\quad;\quad\nabla^{g(1)} \R_{g(1)}(p_1)=0.
\end{eqnarray*}
In particular, it implies that $Z(Y)(p_1,1)=0.$ With the help of the strong maximum principle applied to (\ref{evo-lin-tra-est-3}) , $Z(Y)\equiv 0$. Going back to the evolution of $Z(Y)$ given by (\ref{evo-lin-tra-est-1}) and (\ref{evo-lin-tra-est-2}), one has
\begin{eqnarray*}
\nabla_iY_j-\Ric_{ij}-\frac{g_{ij}}{2}=0.
\end{eqnarray*}
In particular, $\nabla_iY_j=\nabla_jY_i$, which implies that the one-form associated to $Y$ is closed on a simply connected manifold hence exact. This finishes the proof.\\

The second proof of theorem \ref{trivial-exp-sol-gra} is more in the spirit of the no breathers theorem due to Perelman \cite{Per-Ent} involving some relevant entropy introduced by Cao and Hamilton defined originally on compact manifolds. This is the purpose of the next sections.
\subsection{Cao-Hamilton's entropy on positively curved Riemannian manifolds}\label{Sec-Sec-Pro-No-Bre}

\subsubsection{Existence of minimizers}
Cao and Hamilton \cite{Cao-Ham-Har} introduced the following entropy on a general Riemannian manifold $(M^n,g)$. We follow here the presentation and the proofs of Zhang \cite{Zha-Log} on Perelman's entropy : most of the time, the arguments are a straightforward adaptation of the ones given in \cite{Zha-Log}, therefore, we will only point out the necessary modifications.

A complete Riemannian manifold $(M^n,g)$ has \textbf{bounded geometry} if 
\begin{eqnarray*}
\inf_{x\in M}\vol B(x,1)\geq v>0\quad;\quad \arrowvert\nabla^k\Rm(g)\arrowvert\leq C_k<+\infty,\quad\forall k\geq 0.
\end{eqnarray*}
There are various definitions of a Riemannian manifold with bounded geometry in the literature that allows more or less bounded covariant derivatives. Regarding our main application in the setting of Ricci flow, Shi's estimates will ensure the boundedness of the covariant derivatives of the curvature operator as soon as the curvature is bounded.

  Define first
\begin{eqnarray*}
\mathcal{W}(g,v):= \int_{M}(4\arrowvert\nabla v\arrowvert^2-3\R_gv^2-v^2\ln v^2) d\mu(g),
\end{eqnarray*}
 The best log-Sobolev constant of $(M^n,g)$ is defined by 
\begin{eqnarray}
\lambda(g):=\inf\left\{\mathcal{W}(g,v)\quad|\quad v\in C_0^{\infty}(M),\quad\|v\|_{L^2}=1\right\}.
\end{eqnarray}

Again, the setting of \cite{Cao-Ham-Har} concerns compact manifolds. Nonetheless, this invariant is well-defined on a non compact Riemannian manifold as soon as it has bounded scalar curvature and satisfies a  Sobolev type inequality that holds, for instance, on manifolds with bounded geometry or with non negative Ricci curvature and positive asymptotic volume ratio : see the proof of theorem $1.1$ of \cite{Zha-Log}.

\begin{defn}
The best Log-Sobolev constant of $(M^n,g)$ at infinity is defined by 
\begin{eqnarray*}
\lambda_{\infty}(g):=\liminf_{r\rightarrow+\infty}\left\{\mathcal{W}(g,v)\quad|\quad v\in C_0^{\infty}(M\setminus B(p,r)),\quad\|v\|_{L^2}=1\right\}.
\end{eqnarray*}
 \end{defn}

Following closely the work of Zhang \cite{Zha-Log}, one can prove the existence of a minimizer as soon as the geometry at infinity is restricted :
\begin{theo}(Zhang)\label{exi-min}
Let $(M^n,g)$ be a complete connected Riemannian manifold with bounded geometry such that $\lambda(g)<\lambda_{\infty}(g).$ Then there exists a positive smooth minimizer $v$ for $\lambda$. Moreover, there exist positive constants $C$ and $c$ such that $$v(x)\leq Ce^{-cd_g(p,x)^2},$$ for any $x\in M$.
\end{theo}

\subsubsection{Monotonicity of the Cao-Hamilton entropy}

In the spirit of Perelman's entropy, Cao and Hamilton introduced the following entropy for a Riemannian manifold $(M^n,g)$, that we will denote by $\mathcal{W}_{C-H}$ :
\begin{eqnarray*}
\mathcal{W}_{C-H}(g,v,\tau):=\int_M\left[\tau(4\arrowvert\nabla v\arrowvert^2-3\R_gv^2)-v^2\ln v^2-\frac{n}{2}(\ln 4\pi \tau)v^2-nv^2\right]d\mu(g),
\end{eqnarray*}
for $\tau>0$ and $v\in W^{1,2}(M,\mathbb{R})$.

We gather some important remarks in the following proposition :
\begin{prop}\label{prop-prem-entr}
\begin{enumerate}
\item \textit{Scaling invariance property} : if $c>0$, $\tau>0$ and $v\in W^{1,2}(M,\mathbb{R})$ then $$\mathcal{W}_{C-H}(c^2g,c^{-n/2}v,c^2\tau)=\mathcal{W}_{C-H}(g,v,\tau).$$
\item \textit{Diffeomorphism invariance} : if $\psi\in \Diff(M)$, $\tau>0$ and $v\in W^{1,2}(M,\mathbb{R})$ then $$\mathcal{W}_{C-H}(\psi^*g,\psi^*v,\tau)=\mathcal{W}_{C-H}(g,v,\tau).$$
\item If $v\in W^{1,2}(M,\mathbb{R})$ is such that $\| v\|_{L^2}=1$ then $\mathcal{W}_{C-H}(g,v,1)=\mathcal{W}(g,v)-\frac{n}{2}(\ln 4\pi)-n.$
\end{enumerate}
\end{prop}
Therefore, one can define the corresponding entropies (at infinity) with analogy with the previous section : 
\begin{eqnarray*}
\mu_{C-H}(g,\tau)&:=&\inf\left\{\mathcal{W}_{C-H}(g,v,\tau)\quad|\quad v\in C_0^{\infty}(M),\quad\|v\|_{L^2}=1\right\},\\
\mu_{C-H,\infty}(g,\tau)&:=&\liminf_{r\rightarrow+\infty}\left\{\mathcal{W}_{C-H}(g,v,\tau)\quad|\quad v\in C_0^{\infty}(M\setminus B(p,r)),\quad\|v\|_{L^2}=1\right\}.\\
\end{eqnarray*}
Again, we collect from \cite{Cao-Ham-Har} and \cite{Zha-Log} the results we need :
\begin{theo}\label{theo-cao-ham-zha}
Let $(M^n,g(t))_{t\in(0,T]}$ be a complete non compact Ricci flow with bounded nonnegative curvature operator. Let $[t_1,t_2]\subset(0,T]$. Suppose $v_1$ is a minimizer for the Cao-Hamilton entropy $\mathcal{W}_{C-H}$ at time $t_1$. Let $(u(t))_{t\in[t_1,t_2]}$ be the bounded solution to the following heat equation :
\begin{eqnarray}\label{hea-equ-cao-ham}
\left\{
\begin{array}{rl}
&\partial_tg=-2\Ric\\
&\\
& \partial_tu=\Delta u+\R u\\
&\\
&u(t_1):=v_1^2>0,\\
\end{array}
\right.
\end{eqnarray}
with $\|v_1\|_{L^2}=1. $ Then the following holds :\\
\begin{enumerate}
\item (Cao-Hamilton) Define the pointwise entropy related to $u(t)$ : 
\begin{eqnarray*}
&&P(u):=2\Delta f-\arrowvert\nabla f\arrowvert^2-3\R+\frac{f}{t}-\frac{n}{t},\quad u(t)=:\frac{e^{-f(t)}}{(4\pi t)^{n/2}}.
\end{eqnarray*}
Then 
\begin{eqnarray*}
\partial_t(tP(u))&=&\Delta(tP(u))-2\nabla(tP(u))\cdot\nabla f\\
&&-2t\left\arrowvert\nabla^2f-\Ric-\frac{g}{2t}\right\arrowvert^2-2t\Har,
\end{eqnarray*}
where $$\Har=\Har(g(t),u(t)):=\partial_t\R+\frac{\R}{t}+2\nabla\R\cdot\nabla f+2\Ric(\nabla f,\nabla f)$$ is the Harnack quantity associated to $(g(t),u(t))$. In particular,
\begin{eqnarray*}
\partial_t(tP(u)u)&=&\Delta(tP(u)u)+\R(tP(u)u)-2t\left(\left\arrowvert\nabla^2f-\Ric-\frac{g}{2t}\right\arrowvert^2+\Har\right)u
\end{eqnarray*}
\\
\item (Zhang) If $t\in(t_1,t_2]$,
\begin{eqnarray}\label{Zha-Evo-Cao-Ham-Ent}
\partial_tW_{C-H}(g(t),\sqrt{u(t)},t)=-2t\int_{M^n}\left(\left\arrowvert\nabla^2f-\Ric-\frac{g}{2t}\right\arrowvert^2+\Har\right)ud\mu.
\end{eqnarray}
In particular,
\begin{eqnarray*}
\mu_{C-H}(g(t_2),t_2)\leq\mu_{C-H}(g(t_1),t_1).
\end{eqnarray*}

\end{enumerate}
\end{theo}
\begin{proof}
The first part of theorem \ref{theo-cao-ham-zha} is entirely due to \cite{Cao-Ham-Har}.\\

The second part consists in three steps, as described in \cite{Zha-Log}, that we explain and modify now :
\begin{itemize}
\item \begin{claim}\label{dec-hea-equ}
There exist positive constants $A$ and $a$ such that for any $t\in[t_1,t_2]$ and any $x\in M$,
\begin{eqnarray*}
u(x,t)\leq Ae^{-ad_{g(t_1)}^2(p,x)},
\end{eqnarray*}
for some point $p\in M$.
\end{claim}
First of all, remark that, as the curvature is uniformly bounded on $[t_1,t_2]$, the metrics $(g(t))_{t\in[t_1,t_2]}$ are uniformly equivalent.
Secondly, theorem $5.1$ of \cite{Cha-Tam-Yu} gives the following upper bound for the fundamental solution $G(x,t,y,t_1)$ associated to the operator $\partial_t-\Delta-\R$ for $t_1<t$ :
\begin{eqnarray*}
G(x,t,y,t_1)\leq \frac{C_1}{\sqrt{\vol B(x,\sqrt{t-t_1})\vol B(y,\sqrt{t-t_1})}}e^{-c_1\frac{d^2_{g(t_1)}(x,y)}{t-t_1}},
\end{eqnarray*}
for some positive constants $C_1$ and $c_1$ depending only on a bound of the Ricci curvature on $[t_1,t_2]$ and the dimension $n$.
As $(M^n,g(t))_{t\in[t_1,t_2]}$ is a Ricci flow with uniformly bounded non negative sectional curvature, the Sharafutdinov retraction \cite{Sha-Ret} together with Shi's estimates ensure that $(M^n,g(t))_{t\in[t_1,t_2]}$ has bounded geometry where the constants appearing in the definition are uniform in time. Therefore,
\begin{eqnarray*}
G(x,t,y,t_1)\leq \frac{C_1}{(t-t_1)^{n/2}}e^{-c_1\frac{d^2_{g(t_1)}(x,y)}{t-t_1}},
\end{eqnarray*}
for any $t>t_1$ with possibly different constants $c_1$ and $C_1$.
 
Now, by (the proof) of theorem \ref{exi-min}, any minimizer of the Cao-Hamilton entropy satisfies :
\begin{eqnarray*}
v_1(x)\leq A_1e^{-a_1d_{g(t_1)}^2(p,x)},\quad\forall x\in M,
\end{eqnarray*}
for some positive constants $A_1$ and $a_1$. The claim follows by using the representation formula for the solution $u(x,t)$, i.e. 
\begin{eqnarray*}
u(x,t)=\int_{M^n}G(x,t,y,t_1)v_1^2(y)d\mu(g(t_1)),
\end{eqnarray*}
together with the previous upper bounds both of the fundamental solution and the initial condition.

\item Again, following the notations of \cite{Zha-Log}, denote the integrand of the Cao-Hamilton entropy by 
\begin{eqnarray*}
i(u):=t\left(\frac{\arrowvert\nabla u\arrowvert^2}{u}-3\R u\right)-u\ln u-\ln(4\pi t)u-nu.
\end{eqnarray*}
To ensure that $\mathcal{W}_{C-H}(g(t),\sqrt{u(t)},t)$ is well-defined for any $t\in[t_1,t_2]$, it suffices to show that $\arrowvert\nabla u\arrowvert^2/u$ has the same decay as $u$ given in claim \ref{dec-hea-equ}.

The strategy consists in computing the evolution of $\arrowvert\nabla u\arrowvert^2/u$.
On the one hand, define $v(x,t):=\sqrt{u(x,t)}$ for $(x,t)\in M\times [t_1,t_2]$. Then,
\begin{eqnarray*}
\partial_t\left(\frac{\arrowvert\nabla u\arrowvert^2}{u}\right)&=&4\partial_t\arrowvert\nabla v\arrowvert^2\\
&=&8(<\nabla v,\nabla\partial_tv>+\Ric(\nabla v,\nabla v)),\\
\Delta\left(\frac{\arrowvert\nabla u\arrowvert^2}{u}\right)&=&4\Delta\arrowvert\nabla v\arrowvert^2\\
&=&8(\arrowvert\nabla^2v\arrowvert^2+\Ric(\nabla v,\nabla v)+<\nabla v,\nabla\Delta v>).
\end{eqnarray*}
On the other hand,
\begin{eqnarray*}
(\partial_t-\Delta)(v)&=&(\partial_t-\Delta)(\sqrt{u})=\frac{1}{2}\frac{(\partial_t-\Delta)(u)}{\sqrt{u}}+\frac{1}{4}\frac{\arrowvert\nabla u\arrowvert^2}{u^{3/2}}\\
&=&\frac{\R v}{2}+\frac{\arrowvert\nabla v\arrowvert^2}{v}.
\end{eqnarray*}
Finally, we get 
\begin{eqnarray*}
(\partial_t-\Delta-\R)\left(\frac{\arrowvert\nabla u\arrowvert^2}{u}\right)&=&4\left(v<\nabla v,\nabla \R>+2<\nabla v,\nabla\left(\frac{\arrowvert\nabla v\arrowvert^2}{v}\right)>-2\arrowvert\nabla^2v\arrowvert^2\right)\\
&=&4\left(v<\nabla v,\nabla \R>-2\left\arrowvert\nabla^2v-\frac{\nabla v\otimes\nabla v}{v}\right\arrowvert^2\right).
\end{eqnarray*}
In order to absorb the term $v<\nabla v,\nabla \R>$, one observes that :
\begin{eqnarray*}
(\partial_t-\Delta-\R)(\R u)&=&(\partial_t\R-\Delta\R)u+\R(\partial_tu-\Delta u-\R u)-2<\nabla \R,\nabla u>\\
&=&2\arrowvert\Ric\arrowvert^2u-4v<\nabla v,\nabla \R>.
\end{eqnarray*}
Therefore, using that the Ricci curvature is nonnegative (bounded from below would suffice) :
\begin{eqnarray*}
(\partial_t-\Delta-\R)\left(\frac{\arrowvert\nabla u\arrowvert^2}{u}+\R u\right)\leq2\R (\R u)\leq C\left(\frac{\arrowvert\nabla u\arrowvert^2}{u}+\R u\right),
\end{eqnarray*}
on $M\times(t_1,t_2]$, i.e. 
\begin{eqnarray}
&&(\partial_t-\Delta-\R)Q\leq0,\label{sub-sol-gra-est-ent}\\
&&Q(u):=e^{Ct}\left(\frac{\arrowvert\nabla u\arrowvert^2}{u}+\R u\right),
\end{eqnarray}
on $M^n\times(t_1,t_2]$.
Now, one can argue exactly as in \cite{Zha-Log}. In order to apply the maximum principle to (\ref{sub-sol-gra-est-ent}), we need to control the decay at the initial time and we need to know a priori the growth of the solution at infinity.

Concerning the initial condition, first note that :
\begin{eqnarray*}
Q(u)(t_1)=e^{Ct_1}(4\arrowvert\nabla v_1\arrowvert^2+\R v_1^2)
\end{eqnarray*}

As $v_1$ is a minimizer for the Cao-Hamilton entropy, it satisfies the associated Euler-Lagrange equation :
\begin{eqnarray}\label{Eul-Lag-Equ-Cao-Ham-Ent}
4\Delta v_1+3\R v_1+2v_1\ln v_1+\ln(4\pi t_1)v_1+nv_1+\mu_{C-H}(g(t_1),t_1)v_1=0.
\end{eqnarray}
From this and the Bochner formula, $\arrowvert\nabla v_1\arrowvert^2$ satisfies :
\begin{eqnarray*}
\Delta\arrowvert\nabla v_1\arrowvert^2&\geq&2<\nabla\Delta v_1,\nabla v_1>\\
&\geq&-C(\arrowvert\nabla v_1\arrowvert^2+v_1^2).
\end{eqnarray*}
since $v_1$ is in particular bounded from above by theorem \ref{exi-min}. Again, the nonnegativity of the Ricci curvature can be relaxed to bounded Ricci curvature from below. Then, a Moser's iteration shows that 
\begin{eqnarray*}
\sup_{B(x,1/2)}\arrowvert\nabla v_1\arrowvert^2&\leq& C\int_{B(x,1)}(\arrowvert\nabla v_1\arrowvert^2+v_1^2)d\mu(g(t_1))\\
&\leq& C\int_{B(x,2)}v_1^2d\mu(g(t_1)),
\end{eqnarray*}
where, in the last inequality, we use the Euler-Lagrange equation (\ref{Eul-Lag-Equ-Cao-Ham-Ent}) again.
Because of the decay of the minimizer $v_1$ at infinity given by theorem \ref{exi-min}, we get a quadratic exponential decay for $Q(u)(t_1)$.

Concerning the a priori growth at infinity of $Q(u)$, the difficulty comes from the lack of an appropriate lower bound for the minimizer $v_1$. Nonetheless, we circumvent this by using the same approximation procedure introduced by Zhang. Let $\epsilon>0$ and consider the unique bounded solution $u_{\epsilon}$ to 
\begin{eqnarray}\label{hea-equ-app-gra}
\left\{
\begin{array}{rl}
&\partial_tg=-2\Ric\\
&\\
& \partial_tu_{\epsilon}=\Delta u_{\epsilon}+\R u_{\epsilon}\\
&\\
&u_{\epsilon}(t_1):=v_1^2+\epsilon>0,\\
\end{array}
\right.
\end{eqnarray}
on $[t_1,t_2]\times M$. This solution has the main advantage that it is uniformly bounded from below, moreover, $u_{\epsilon}$ converges to $u$ pointwise as $\epsilon$ goes to $0$.

The quantity $Q(u_{\epsilon})$ that is now bounded from above on $[t_1,t_2]$ satisfies 
\begin{eqnarray*}
(\partial_t-\Delta-\R)(Q(u_{\epsilon}))\leq 0,
\end{eqnarray*}
as in (\ref{sub-sol-gra-est-ent}). By the maximum principle, 
\begin{eqnarray*}
Q(u_{\epsilon})(x,t)\leq\int_{M^n}G(x,t,y,t_1)Q(u_{\epsilon})(y,t_1)d\mu(g(t_1)),
\end{eqnarray*}
on $M\times[t_1,t_2]$.
As for the upper bound of $u$, one has :
\begin{eqnarray*}
Q(u_{\epsilon})(x,t)\leq C_1e^{-c_1d^2_{g(t_1)}(p,x)}+C_1\epsilon,
\end{eqnarray*}
on $M\times[t_1,t_2]$ where $C_1$ and $c_1$ are positive constants independent of $\epsilon$. If $\epsilon$ goes to $0$, then we get an appropriate decay for $\arrowvert\nabla u\arrowvert^2/u$ ensuring the finiteness of the expression $\mathcal{W}_{C-H}(g(t),\sqrt{u(t)},t)$ for $t\in[t_1,t_2]$.
\item Finally, we compute the time derivative of $t\rightarrow\mathcal{W}_{C-H}(g(t),\sqrt{u(t)},t)$ as in \cite{Zha-Log} with the help of the computations performed by Cao and Hamilton. First note that 
\begin{eqnarray*}
tP(u)u&=&t\left(-2\Delta u+\frac{\arrowvert\nabla u\arrowvert^2}{u}-3\R u\right)-u\ln u-\ln(4\pi t)u-nu\\
&=&-2t\Delta u+i(u).
\end{eqnarray*}
Therefore, if $\phi\in C_0^{\infty}(M,\mathbb{R})$,
\begin{eqnarray*}
\partial_t\int_{M}(tP(u)u)\phi d\mu(g(t))&=&\int_{M}(\partial_t-\R)(tP(u)u)\phi d\mu(g(t))\\
&=&\int_{M}\Delta(tP(u)u)\phi d\mu(g(t))\\
&&-2t\int_{M}\left(\left\arrowvert\nabla^2f-\Ric-\frac{g}{2t}\right\arrowvert^2+\Har\right)u\phi d\mu(g(t))\\
&=&\int_{M}(tP(u)u)\Delta\phi d\mu(g(t))\\
&&-2t\int_{M}\left(\left\arrowvert\nabla^2f-\Ric-\frac{g}{2t}\right\arrowvert^2+\Har\right)u\phi d\mu(g(t))\\
&=&-2t\int_{M}(\Delta\Delta \phi)ud\mu(g(t))+\int_{M}(\Delta \phi)i(u)d\mu(g(t))\\
&&-2t\int_{M}\left(\left\arrowvert\nabla^2f-\Ric-\frac{g}{2t}\right\arrowvert^2+\Har\right)u\phi d\mu(g(t)).
\end{eqnarray*}
Now, as $i(u)$ is integrable, by choosing a sequence $(\phi_k)_k$ of functions with compact support approximating the identity as in \cite{Zha-Log}, one ends up with
\begin{eqnarray*}
\mu_{C-H}(g(t_2),t_2)\leq\mathcal{W}_{C-H}(g(t_2),\sqrt{u(t_2)},t_2)\leq\mathcal{W}_{C-H}(g(t_1),\sqrt{u(t_1)},t_1)=\mu_{C-H}(g(t_1,t_1).
\end{eqnarray*}

\end{itemize}

\end{proof}
With theorem \ref{theo-cao-ham-zha} in hand, one can prove that some expanding Ricci solitons are actually gradient :
\begin{coro}\label{no-non-tri-exp-Ric-sol}
Let $(M^n,g(t))_{t\in(0,+\infty)}$ be a non compact non flat expanding Ricci soliton flow with bounded nonnegative curvature operator on compact time intervals. Assume $$\mu_{C-H}(g(t_1),t_1)<\mu_{C-H,\infty}(g(t_1),t_1),$$ for some positive time $t_1$.

Then $(M^n,g(t))_{t\in(0,+\infty)}$ is an expanding \textbf{gradient} Ricci soliton.
\end{coro}

\begin{proof}

According to theorem \ref{exi-min}, there exists a smooth positive minimizer $v_1$ for $\mu_{C-H}(g(t_1),t_1)$ with quadratic exponential decay. Let $(u(t))_t$ be the solution to the heat equation (\ref{hea-equ-cao-ham}). In particular, we get 
\begin{eqnarray*}
W_{C-H}(g(t_2),\sqrt{u(t_2)},t_2)\leq W_{C-H}(g(t_1),\sqrt{u(t_1)},t_1)=\mu_{C-H}(g(t_1),t_1).
\end{eqnarray*}
Now, by the invariance property under scalings and under diffeomorphism given by proposition \ref{prop-prem-entr}, $$\mu_{C-H}(g(t_2),t_2)=\mu_{C-H}\left(\frac{t_2}{t_1}\phi_{t_1,t_2}^*g(t_1),\frac{t_2}{t_1}t_1\right)=\mu_{C-H}(g(t_1),t_1),$$
where $\phi_{t_1,t_2}:=\phi_{t_1-1}^{-1}\circ\phi_{t_2-1}.$
Therefore, the function $t\rightarrow \mu_{C-H}(g(t),t)$ is constant on $[t_1,+\infty)$ which implies the expected result by using the evolution of the Cao-Hamilton entropy given by (\ref{Zha-Evo-Cao-Ham-Ent}).

\end{proof}

Now, we give an alternative proof of theorem \ref{trivial-exp-sol-gra} : 
\begin{proof}[Second proof of  theorem \ref{trivial-exp-sol-gra}]
Let $(M^n,g,V)$ be such an implicit expanding Ricci soliton. As it is asymptotically conical, one can check that 
\begin{eqnarray*}
\mu_{C-H,\infty}(g(t),t)=\mu_{C-H}(\eucl,t)\geq 0,
\end{eqnarray*}
where $\mu_{C-H}(\eucl,t)$ is the Cao-Hamilton entropy of the Euclidean space at a positive time $t$. The fact that $\mu_{C-H}(\eucl,t)\geq 0$ for any positive $t$ comes from the Gross logarithmic Sobolev inequality \cite{Gro-Log-Sob} for the standard Gaussian measure $(2\pi )^{-n/2}e^{-\arrowvert\cdot\arrowvert^2/2}$. On the other hand, this implicit Ricci expander is globally $C^2$ close to a fixed non flat gradient Ricci expander $(M^n,g_0,\nabla^0f_0)$ with non negative (Ricci) curvature. By \cite{Car-Ni}, 
\begin{eqnarray*}
\mu_{C-H}(g_0(1),1)<0.
\end{eqnarray*}
Hence, $\mu_{C-H}(g(1),1)<0$ too. We can now apply corollary \ref{no-non-tri-exp-Ric-sol} to show that $(M^n,g,V)$ carries a structure of  expanding gradient Ricci soliton.

\end{proof}

\subsubsection{Proof of theorem \ref{theo-I}}
Let $(X,g_X)$ be a compact simply connected Riemannian manifold with positive curvature operator. Let $(X,g_X(s))_{s\in[0,+\infty]}$ be a curve of Riemannian metrics with positive curvature operator of constant volume provided by Böhm and Wilking \cite{Boh-Wil} connecting $(X,g_X)$ to $(X,c^2g_{\mathbb{S}^{n-1}})$ where $c^{n-1}=\vol(X,g_X)/\vol(\mathbb{S}^{n-1},g_{\mathbb{S}^{n-1}})$. 

Thanks to theorems \ref{Compactness-II}, \ref{Loc-Dif-Fre} and \ref{trivial-exp-sol-gra}, there exists a curve of expanding gradient Ricci solitons $(M^n,g(s),\nabla^{g(s)}f(s))$ with positive curvature operator asymptotic to $(C(X),dr^2+r^2g_X(s),r\partial_r/2)$ for $s\in[0,+\infty]$. 

Now, if two expanding gradient Ricci solitons $(M_i,g_i,\nabla^{g_i}f_i)_{i=1,2}$ with positive curvature operator are asymptotic to the same cone, one can connect each soliton to a soliton \\$(\bar{M_i},\bar{g_i},\nabla^{\bar{g_i}}\bar{f_i})_{i=1,2}$ with positive curvature operator whose asymptotic cone is the most symmetric one, i.e. $(C(\mathbb{S}^{n-1}),dr^2+c^2g_{\mathbb{S}^{n-1}})$ with $c$ depending only on the volume of the section of the initial cone as before. According to Chodosh \cite{Cho-Exp-Asy-Con}, $(\bar{M_1},\bar{g_1},\nabla^{\bar{g_1}}\bar{f_1})$ and $(\bar{M_2},\bar{g_2},\nabla^{\bar{g_2}}\bar{f_2})$ are isometric to the same corresponding Bryant soliton. By using theorem \ref{Loc-Dif-Fre}, one gets that the two curves of implicit expanding gradient Ricci solitons are actually isometric, hence the uniqueness result.

\appendix

\section{Soliton equations}\label{sol-equ-sec}

The next lemma gathers well-known Ricci soliton identities together with the (static) evolution equations satisfied by the curvature tensor.

Recall first that an expanding gradient Ricci soliton is said \textit{normalized} if $\int_Me^{-f}d\mu_g=(4\pi)^{n/2}$ (whenever it makes sense).

\begin{lemma}\label{id-EGS}
Let $(M^n,g,\nabla f)$ be a normalized expanding gradient Ricci soliton. Then the trace and first order soliton identities are :
\begin{eqnarray}
&&\Delta f = \R_g+\frac{n}{2}, \label{equ:1} \\
&&\nabla \R_g+ 2\Ric(g)(\nabla f)=0, \label{equ:2} \\
&&\arrowvert \nabla f \arrowvert^2+\R_g=f+\mu(g), \label{equ:3}\\
&&\div\Rm(g)(Y,Z,T)=\Rm(g)(Y,Z, \nabla f,T),\label{equ:4}
\end{eqnarray}
for any vector fields $Y$, $Z$, $T$ and where $\mu(g)$ is a constant called the entropy.\\

The evolution equations for the curvature operator, the Ricci tensor and the scalar curvature are :
\begin{eqnarray}
&& \Delta_f \Rm(g)+\Rm(g)+\Rm(g)\ast\Rm(g)=0,\label{equ:5}\\
&&\Delta_f\Ric(g)+\Ric(g)+2\Rm(g)\ast\Ric(g)=0,\label{equ:6}\\
&&\Delta_f\R_g+\R_g+2\arrowvert\Ric(g)\arrowvert^2=0,\label{equ:7}
\end{eqnarray}
where, if $A$ and $B$ are two tensors, $A\ast B$ denotes any linear combination of contractions of the tensorial product of $A$ and $B$.
\end{lemma}

\begin{proof}
See [Chap.$1$,\cite{Cho-Lu-Ni-I}] for instance.

\end{proof}

\begin{prop}\label{pot-fct-est}
Let $(M^n,g,\nabla f)$ be an expanding gradient Ricci soliton.
\begin{itemize}
\item If $(M^n,g,\nabla f)$ is non Einstein,
\begin{eqnarray}
&&\Delta_fv=v,\quad v>\arrowvert\nabla v\arrowvert^2.\label{inequ:1}\\
\end{eqnarray}

\item Assume $\Ric(g)\geq 0$ and assume $(M^n,g,\nabla f)$ is normalized. Then $M^n$ is diffeomorphic to $\mathbb{R}^n$ and
\begin{eqnarray}
&&v\geq \frac{n}{2}>0.\label{inequ:2}\\
&&\frac{1}{4}r_p(x)^2+\min_{M}v\leq v(x)\leq\left(\frac{1}{2}r_p(x)+\sqrt{\min_{M}v}\right)^2,\quad \forall x\in M,\label{inequ:3}\\
&&\AVR(g):=\lim_{r\rightarrow+\infty}\frac{\vol B(q,r)}{r^n}>0,\quad\forall q\in M,\label{inequ:avr}\\
&&-C(n,V_0,R_0)\leq\min_{M}f\leq 0\quad;\quad\mu(g)\geq\max_{M}\R_g\geq 0,\label{inequ:ent}
\end{eqnarray}
where $V_0$ is a positive number such that $\AVR(g)\geq V_0$, $R_0$ is such that $\sup_{M}\R_g\leq R_0$ and $p\in M$ is the unique critical point of $v$.\\

\item Assume $\Ric(g)=\textit{O}(r_p^{-2})$ where $r_p$ denotes the distance function to a fixed point $p\in M$. Then the potential function is equivalent to $r_p^2/4$ (up to order $2$).
\end{itemize}
\end{prop}

For a proof, see \cite{Der-Asy-Com-Egs} and the references therein.

\bibliographystyle{alpha.bst}
\bibliography{bib-mod-space-egs}

\newcommand{\etalchar}[1]{$^{#1}$}
\def\cprime{$'$} \def\cprime{$'$}
\begin{thebibliography}{{Der}14a}

\bibitem[And08]{And-Pre-Con-Inf}
Michael~T. Anderson.
\newblock Einstein metrics with prescribed conformal infinity on 4-manifolds.
\newblock {\em Geom. Funct. Anal.}, 18(2):305--366, 2008.

\bibitem[Ban87]{Ban-Ana}
Shigetoshi Bando.
\newblock Real analyticity of solutions of {H}amilton's equation.
\newblock {\em Math. Z.}, 195(1):93--97, 1987.

\bibitem[Biq00]{Biq-Met-Asy-Ein}
Olivier Biquard.
\newblock M\'etriques d'{E}instein asymptotiquement sym\'etriques.
\newblock {\em Ast\'erisque}, (265):vi+109, 2000.

\bibitem[Bre13]{Bre-Rot-3d}
Simon Brendle.
\newblock Rotational symmetry of self-similar solutions to the {R}icci flow.
\newblock {\em Invent. Math.}, 194(3):731--764, 2013.

\bibitem[BW08]{Boh-Wil}
Christoph B{\"o}hm and Burkhard Wilking.
\newblock Manifolds with positive curvature operators are space forms.
\newblock {\em Ann. of Math. (2)}, 167(3):1079--1097, 2008.

\bibitem[CCG{\etalchar{+}}07]{Cho-Lu-Ni-I}
Bennett Chow, Sun-Chin Chu, David Glickenstein, Christine Guenther, James
  Isenberg, Tom Ivey, Dan Knopf, Peng Lu, Feng Luo, and Lei Ni.
\newblock {\em The {R}icci flow: techniques and applications. {P}art {I}},
  volume 135 of {\em Mathematical Surveys and Monographs}.
\newblock American Mathematical Society, Providence, RI, 2007.
\newblock Geometric aspects.

\bibitem[CG85a]{Che-Gro-Bou-Neu}
Jeff Cheeger and Mikhael Gromov.
\newblock Bounds on the von {N}eumann dimension of {$L^2$}-cohomology and the
  {G}auss-{B}onnet theorem for open manifolds.
\newblock {\em J. Differential Geom.}, 21(1):1--34, 1985.

\bibitem[CG85b]{Che-Gro-Cha-Num}
Jeff Cheeger and Mikhael Gromov.
\newblock On the characteristic numbers of complete manifolds of bounded
  curvature and finite volume.
\newblock In {\em Differential geometry and complex analysis}, pages 115--154.
  Springer, Berlin, 1985.

\bibitem[CH09]{Cao-Ham-Har}
Xiaodong Cao and Richard~S. Hamilton.
\newblock Differential {H}arnack estimates for time-dependent heat equations
  with potentials.
\newblock {\em Geom. Funct. Anal.}, 19(4):989--1000, 2009.

\bibitem[Cho14]{Cho-Exp-Asy-Con}
Otis Chodosh.
\newblock Expanding {R}icci solitons asymptotic to cones.
\newblock {\em Calc. Var. Partial Differential Equations}, 51(1-2):1--15, 2014.

\bibitem[CLN06]{Ben}
Bennett Chow, Peng Lu, and Lei Ni.
\newblock {\em Hamilton's {R}icci flow}, volume~77 of {\em Graduate Studies in
  Mathematics}.
\newblock American Mathematical Society, Providence, RI, 2006.

\bibitem[CN09]{Car-Ni}
Jos{\'e}~A. Carrillo and Lei Ni.
\newblock Sharp logarithmic {S}obolev inequalities on gradient solitons and
  applications.
\newblock {\em Comm. Anal. Geom.}, 17(4):721--753, 2009.

\bibitem[CTY11]{Cha-Tam-Yu}
Albert Chau, Luen-Fai Tam, and Chengjie Yu.
\newblock Pseudolocality for the {R}icci flow and applications.
\newblock {\em Canad. J. Math.}, 63(1):55--85, 2011.

\bibitem[{Der}14a]{Der-Asy-Com-Egs}
A.~{Deruelle}.
\newblock {Asymptotic estimates and compactness of expanding gradient Ricci
  solitons}.
\newblock {\em ArXiv e-prints}, November 2014.

\bibitem[Der14b]{Der-Sta-Sge}
A.~Deruelle.
\newblock Stability of non compact steady and expanding gradient ricci
  solitons.
\newblock {\em ArXiv e-prints}, 2014.

\bibitem[DPL95]{DaP-Lun-Orn-Uhl}
Giuseppe Da~Prato and Alessandra Lunardi.
\newblock On the {O}rnstein-{U}hlenbeck operator in spaces of continuous
  functions.
\newblock {\em J. Funct. Anal.}, 131(1):94--114, 1995.

\bibitem[FIN05]{Fel-Ilm-Ni}
Michael Feldman, Tom Ilmanen, and Lei Ni.
\newblock Entropy and reduced distance for {R}icci expanders.
\newblock {\em J. Geom. Anal.}, 15(1):49--62, 2005.

\bibitem[GL91]{Gra-Lee-Ein-Con-Inf}
C.~Robin Graham and John~M. Lee.
\newblock Einstein metrics with prescribed conformal infinity on the ball.
\newblock {\em Adv. Math.}, 87(2):186--225, 1991.

\bibitem[Gro75]{Gro-Log-Sob}
Leonard Gross.
\newblock Hypercontractivity and logarithmic {S}obolev inequalities for the
  {C}lifford {D}irichlet form.
\newblock {\em Duke Math. J.}, 42(3):383--396, 1975.

\bibitem[Gro07]{Gro-Boo}
Misha Gromov.
\newblock {\em Metric structures for {R}iemannian and non-{R}iemannian spaces}.
\newblock Modern Birkh\"auser Classics. Birkh\"auser Boston Inc., Boston, MA,
  english edition, 2007.
\newblock Based on the 1981 French original, With appendices by M. Katz, P.
  Pansu and S. Semmes, Translated from the French by Sean Michael Bates.

\bibitem[Ham82]{Ham-Nas-Mos}
Richard~S. Hamilton.
\newblock The inverse function theorem of {N}ash and {M}oser.
\newblock {\em Bull. Amer. Math. Soc. (N.S.)}, 7(1):65--222, 1982.

\bibitem[Ham86]{Ham-Fou}
Richard~S. Hamilton.
\newblock Four-manifolds with positive curvature operator.
\newblock {\em J. Differential Geom.}, 24(2):153--179, 1986.

\bibitem[Ham93]{Ham-Ete-Ric}
Richard~S. Hamilton.
\newblock Eternal solutions to the {R}icci flow.
\newblock {\em J. Differential Geom.}, 38(1):1--11, 1993.

\bibitem[Kot06]{Kot-Unp-Not}
Brett Kotschwar.
\newblock A note on the uniqueness of complete, positively-curved expanding
  ricci solitons in 2-d.
\newblock {\em Preprint}, 2006.

\bibitem[Lee06]{Lee-Hyp}
John~M. Lee.
\newblock Fredholm operators and {E}instein metrics on conformally compact
  manifolds.
\newblock {\em Mem. Amer. Math. Soc.}, 183(864):vi+83, 2006.

\bibitem[Lun98]{Lun-Sch-Est}
Alessandra Lunardi.
\newblock Schauder theorems for linear elliptic and parabolic problems with
  unbounded coefficients in {$\bold R^n$}.
\newblock {\em Studia Math.}, 128(2):171--198, 1998.

\bibitem[Lun09]{Lun-Not-Int}
Alessandra Lunardi.
\newblock {\em Interpolation theory}.
\newblock Appunti. Scuola Normale Superiore di Pisa (Nuova Serie). [Lecture
  Notes. Scuola Normale Superiore di Pisa (New Series)]. Edizioni della
  Normale, Pisa, second edition, 2009.

\bibitem[Ni05]{Ni-Mon-Kah-Ric}
Lei Ni.
\newblock Monotonicity and {K}\"ahler-{R}icci flow.
\newblock In {\em Geometric evolution equations}, volume 367 of {\em Contemp.
  Math.}, pages 149--165. Amer. Math. Soc., Providence, RI, 2005.

\bibitem[{Per}02]{Per-Ent}
G.~{Perelman}.
\newblock {The entropy formula for the Ricci flow and its geometric
  applications}.
\newblock {\em ArXiv Mathematics e-prints}, November 2002.

\bibitem[Pet06]{Pet}
Peter Petersen.
\newblock {\em Riemannian geometry}, volume 171 of {\em Graduate Texts in
  Mathematics}.
\newblock Springer, New York, second edition, 2006.

\bibitem[{\v{S}}ar77]{Sha-Ret}
V.~A. {\v{S}}arafutdinov.
\newblock The {P}ogorelov-{K}lingenberg theorem for manifolds that are
  homeomorphic to {${\bf R}^{n}$}.
\newblock {\em Sibirsk. Mat. \v Z.}, 18(4):915--925, 958, 1977.

\bibitem[Shi89]{Shi-Def}
Wan-Xiong Shi.
\newblock Deforming the metric on complete {R}iemannian manifolds.
\newblock {\em J. Differential Geom.}, 30(1):223--301, 1989.

\bibitem[Shi97]{Shi-Kah}
Wan-Xiong Shi.
\newblock Ricci flow and the uniformization on complete noncompact {K}\"ahler
  manifolds.
\newblock {\em J. Differential Geom.}, 45(1):94--220, 1997.

\bibitem[Sie13]{Sie-Phd}
Michael Siepmann.
\newblock Ricci flow of ricci flat cones.
\newblock Master's thesis, ETH, 2013.

\bibitem[SS13]{Sch-Sim}
Felix Schulze and Miles Simon.
\newblock Expanding solitons with non-negative curvature operator coming out of
  cones.
\newblock {\em Math. Z.}, 275(1-2):625--639, 2013.

\bibitem[{Zha}11]{Zha-Log}
Q.~S. {Zhang}.
\newblock {Extremal of Log Sobolev inequality and {W} entropy on noncompact
  manifolds}.
\newblock {\em ArXiv e-prints}, May 2011.

\end{thebibliography}

\end{document}